\documentclass{amsart}
\RequirePackage{amsmath}
\RequirePackage{amssymb}
\RequirePackage{amsxtra}
\RequirePackage{amsfonts}
\RequirePackage{latexsym}
\RequirePackage{euscript}
\RequirePackage{amscd}
\RequirePackage{amsthm}
\RequirePackage{xypic}
\usepackage{stmaryrd}
\usepackage{mathtools}
\xyoption{all}

\newcommand{\Mbar}{\overline{\mathcal{M}}}
\newcommand{\Mbarp}{\Mbar\vphantom{\M}'}
\newcommand{\Mmin}{M^{\min}}
\newcommand{\PS}{\operatorname{PS}}
\newcommand{\cusp}{\operatorname{cusp}}
\newcommand{\ghat}{\widehat{g}}
\newcommand{\Xbar}{\overline{X}}
\usepackage{tikz}
\makeatletter
\newlength\xvec@height%
\newlength\xvec@depth%
\newlength\xvec@width%
\newcommand{\xvec}[2][]{%
  \ifmmode%
    \settoheight{\xvec@height}{$#2$}%
    \settodepth{\xvec@depth}{$#2$}%
    \settowidth{\xvec@width}{$#2$}%
  \else%
    \settoheight{\xvec@height}{#2}%
    \settodepth{\xvec@depth}{#2}%
    \settowidth{\xvec@width}{#2}%
  \fi%
  \def\xvec@arg{#1}%
  \def\xvec@dd{:}%
  \def\xvec@d{.}%
  \raisebox{.2ex}{\raisebox{\xvec@height}{\rlap{%
    \kern.05em
    \begin{tikzpicture}[scale=1]
    \pgfsetroundcap
    \draw (.05em,0)--(\xvec@width-.05em,0);
    \draw (\xvec@width-.05em,0)--(\xvec@width-.15em, .075em);
    \draw (\xvec@width-.05em,0)--(\xvec@width-.15em,-.075em);
    \ifx\xvec@arg\xvec@d%
      \fill(\xvec@width*.45,.5ex) circle (.5pt);%
    \else\ifx\xvec@arg\xvec@dd%
      \fill(\xvec@width*.30,.5ex) circle (.5pt);%
      \fill(\xvec@width*.65,.5ex) circle (.5pt);%
    \fi\fi%
    \end{tikzpicture}%
  }}}%
  #2%
}
\makeatother

\usepackage{dsfont}

\let\emptyset\varnothing

\newcommand{\base}{\imath}
\newcommand{\Jbase}{J_{\operatorname{base}}}
\newcommand{\Jmin}{J_{\min}}
\newcommand{\Jmax}{J_{\max}}

\newcommand{\JH}{\operatorname{JH}}

\allowdisplaybreaks

\def\A{\mathbb A}

\def\C{\mathbb C}
\def\F{\mathbb F}

\def\Q{\mathbb{Q}}

\def\T{\mathbb{T}}

\def\Z{\mathbb{Z}}
\def\Fbar{\overline{\F}}

\def\Qbar{\overline{\Q}}

\def\Zbar{\overline{\Z}}

\def\m{\mathfrak m}
\newcommand{\st}{\mathrm{st}}

\def\chibar{\overline{\chi}}

\def\unif{\varpi}

\def\GL{\operatorname{GL}}

\def\Gal{\mathrm{Gal}}

\def\Sym{\mathrm{Sym}}
\def\Ext{\mathrm{Ext}}

\def\End{\mathrm{End}}

\def\Art{\mathop{\mathrm{Art}}\nolimits}

\def\Hom{\mathop{\mathrm{Hom}}\nolimits}

\def\Spf{\mathop{\mathrm{Spf}}\nolimits}
\def\Frob{\mathop{\mathrm{Frob}}\nolimits}

\def\Ind{\mathop{\mathrm{Ind}}\nolimits}

\def\Fil{\mathop{\mathrm{Fil}}\nolimits}

\def\soc{\mathop{\mathrm{soc}}\nolimits}

\def\rhobar{\overline{\rho}}

\def\m{\mathfrak{m}}

\def\iso{\buildrel \sim \over \longrightarrow}

\def\rmcap{\mathop{\mathrm{cap}}\nolimits}
\def\gauge{\cG}

\def\cosoc{\mathop{\mathrm{cosoc}}\nolimits}

\swapnumbers

\newcommand{\into}{\hookrightarrow}

\newcommand{\To}{\longrightarrow}
\newcommand{\isoto}{\stackrel{\sim}{\To}}

\newlength{\ownl}

\newcommand{\BC}{{\operatorname{BC}\,}}

\newcommand{\Fr}{{\operatorname{Fr}}}

\newcommand{\Id}{{\operatorname{Id}}}
\renewcommand{\Im}{{\operatorname{Im}\,}}

\newcommand{\val}{{\operatorname{val}}}

\newcommand{\cris}{{\operatorname{cris}}}

\newcommand{\semis}{{\operatorname{ss}}}

\newcommand{\univ}{{\operatorname{univ}}}

\newcommand{\M}{{\mathcal{M}}}

\newcommand{\calD}{{\mathcal{D}}}

\newcommand{\CS}{{\mathcal{S}}}

\newcommand{\cC}{\mathcal{C}}
\renewcommand{\cD}{\mathcal{D}}
\newcommand{\cE}{\mathcal{E}}
\newcommand{\cF}{\mathcal{F}}
\newcommand{\cG}{\mathcal{G}}
\newcommand{\calH}{\mathcal{H}}
\newcommand{\cI}{\mathcal{I}}
\newcommand{\cJ}{\mathcal{J}}

\newcommand{\cM}{\mathcal{M}}

\newcommand{\cO}{\mathcal{O}}

\newcommand{\cP}{\mathcal{P}}

\newcommand{\cS}{\mathcal{S}}

\newcommand{\cW}{\mathcal{W}}

\newcommand{\barS}{\overline{{S}}}

\newcommand{\tJ}{\widetilde{{J}}}

\newcommand{\tS}{\widetilde{{S}}}

\newcommand{\tv}{{\widetilde{{v}}}}

\newcommand{\gammabar   }{\overline{\gamma}}

\newcommand{\varepsilonbar  }{\overline{\varepsilon}}   
   
\newcommand{\etabar         }{\overline{\eta}}   
 \newcommand{\thetabar    }{\overline{\theta}}

\newcommand{\kappabar     }{\overline{\kappa}}

 \newcommand{\xibar    }{\overline{\xi}}

 \newcommand{\sigmabar   }{\overline{\sigma}}   
    
 \newcommand{\taubar     }{\overline{\tau}}

 \newcommand{\psibar   }{\overline{\psi}}

\newcommand{\Thetabar       }{\overline{\Theta}}

\newcommand{\Sigmabar   }{\overline{\Sigma}}

\def\RCS$#1: #2 ${\expandafter\def\csname RCS#1\endcsname{#2}}
\RCS$Revision: 85 $
\RCS$Date: 2011-12-16 18:54:17 -0600 (Fri, 16 Dec 2011) $

\newcommand{\bigO}{\mathcal{O}} \newcommand{\calO}{\mathcal{O}}

\newcommand{\p}{\mathfrak{p}}

\DeclareMathOperator{\ssg}{ss}

\newcommand{\rbar}{{\bar{r}}}

\newcommand{\Rbar}{\overline{R}}

\newcommand{\HT}{\operatorname{HT}}

 \newcommand{\Qp}{{\Q_p}}

\newcommand{\Zp}{{\Z_p}}
\newcommand{\Ql}{\Q_l} 
\newcommand{\Qpbar}{{\overline{\Q}_p}}
\newcommand{\Zpbar}{{\overline{\Z}_p}}

\newcommand{\Fpbar}{{\overline{\F}_p}}

\newcommand{\Fp}{{\F_p}}

\newcommand{\rank}{\operatorname{rank}}

\newcommand{\sep}{\operatorname{sep}}

\newcommand{\mar}[1]{\marginpar{\tiny #1}} 

\newcommand{\emb}{\kappa}
\newcommand{\embb}{\overline{\emb}}
\newcommand{\uni}{\pi}
\newcommand{\uMF}{\underline{\mathrm{MF}}}
\newcommand{\MFh}{\uMF^{f,h}_k}
\newcommand{\strdiv}{\underline{\EuScript{M}}_{0,k}}
\newcommand{\EM}{\EuScript{M}}
\newcommand{\EF}{\EuScript{F}}
\newcommand{\oS}{\overline{S}}
\newcommand{\phiD}{\underline{\mathfrak{D}}}
\newcommand{\phiM}{\underline{\mathfrak{M}}}

\usepackage{hyperref}

\newtheorem{theorem}[subsubsection]{Theorem}
\newtheorem{thm}[subsubsection]{Theorem}
\newtheorem{lemma}[subsubsection]{Lemma}
\newtheorem{lem}[subsubsection]{Lemma}

\newtheorem{cor}[subsubsection]{Corollary}

\newtheorem{prop}[subsubsection]{Proposition}

\newtheorem{ithm}{Theorem}

\theoremstyle{definition}

\newtheorem{defn}[subsubsection]{Definition}

\theoremstyle{remark}
\newtheorem{remark}[subsubsection]{Remark}
\newtheorem{rem}[subsubsection]{Remark}

\newtheorem{example}[subsubsection]{Example}
\newtheorem{property}[subsubsection]{Property}

\def\numequation{\begin{equation}}
\def\nummultline{\begin{multline}}
\def\anumequation{\begin{equation}}
\numberwithin{equation}{section}

\title{Lattices in the cohomology of Shimura curves}
\author{Matthew Emerton}\email{emerton@math.uchicago.edu}
\address{Department of Mathematics, University of Chicago,
5734 S.\ University Ave., Chicago, IL 60637}

\author{Toby Gee} \email{toby.gee@imperial.ac.uk} \address{Department of
  Mathematics, Imperial College London}\author{David Savitt} \email{savitt@math.arizona.edu}
\address{Department of Mathematics, University of Arizona}
\thanks{The first author was supported in part by NSF grants
  DMS-1003339 and DMS-1249548, the second author was supported in part by a Marie
  Curie Career Integration Grant and an ERC starting grant, and the third author was 
  supported in part by NSF  grant DMS-0901049 and NSF CAREER grant
  DMS-1054032. The second and third authors would like to thank the
  mathematics department of  Northwestern University for its
  hospitality.}

\begin{document}
\begin{abstract} We prove the main conjectures of
  \cite{breuillatticconj} (including a generalisation from the
  principal series to the cuspidal case) and \cite{dembeleappendix},
  subject to a mild global hypothesis that we make in order to apply
  certain $R=\T$ theorems. More precisely, we prove a multiplicity one
  result for the mod $p$ cohomology of a Shimura curve at Iwahori
  level, and we show that certain apparently globally defined lattices
  in the cohomology of Shimura curves are determined by the
  corresponding local $p$-adic Galois representations. We also
  indicate a new proof of the Buzzard--Diamond--Jarvis conjecture in
  generic cases. Our main tools are the geometric
  Breuil--M\'ezard philosophy developed in~\cite{emertongeerefinedBM},
  and a new and more functorial perspective on the
  Taylor--Wiles--Kisin patching method. Along the way, we determine
  the tamely potentially Barsotti--Tate deformation rings of generic
  two-dimensional mod $p$ representations, generalising a result
  of~\cite{BreuilMezardRaffinee} in the principal series case.
\end{abstract}
\maketitle
\setcounter{tocdepth}{2}
\tableofcontents

\section{Introduction}
\subsection{Breuil's Conjectures}The $p$-adic local Langlands
correspondence for the group $\GL_2(\Qp)$ (\emph{cf.}\ \cite{MR2642409,paskunasimage})
is one of the most important developments in the Langlands program in
recent years, and the natural and definitive nature of the results is
extremely suggestive of there being a similar correspondence for
$\GL_n(K)$, where $n$ is an arbitrary natural number and
$K$ an arbitrary finite extension of $\Qp$. Unfortunately,
all available evidence to date suggests that any more general such
correspondence will be
much more complicated to describe, even in the case of
$\GL_2(K)$ for an unramified
extension $K$ of $\Qp$. For example, one of the first steps in the
construction of the correspondence for $\GL_2(\Qp)$ was the explicit
classification of irreducible admissible smooth representations of
$\GL_2(\Q_p)$ in characteristic $p$~\cite{barthel-livne,breuil1}, from
which it is straightforward to deduce a semisimple mod $p$ local
Langlands correspondence for $\GL_2(\Qp)$~\cite{breuil1}. To date, no
analogous classification is known for any other $\GL_2(K)$, and it is
known that the irreducible representations live in infinite families, described by
parameters that have yet to be
determined~\cite{Breuil_Paškūnas_2012}. In particular, there are in a
precise sense too many representations of $\GL_2(K)$ for there to
be any simple one-to-one correspondence with Galois representations;
rather, it seems that only certain particular 
representations of
$\GL_2(K)$ will participate in the correspondence.

By the local-global compatibility result of~\cite{emerton2010local},
the completed cohomology of modular curves can be described in terms
of the $p$-adic local Langlands correspondence for
$\GL_2(\Qp)$. Similarly, it is expected that the completed cohomology
of Shimura curves can be described in terms of the anticipated $p$-adic local
Langlands correspondence for $\GL_2(K)$. Accordingly, one way to get
information about this hypothetical
correspondence is to study the action
of $\GL_2(K)$ on the cohomology of Shimura curves, for example with
characteristic $p$ coefficients. As one illustration of this approach,
we recall that the $\GL_2(\cO_K)$-socles of the
irreducible $\GL_2(K)$-subrepresentations of this mod $p$ cohomology
are described by the
Buzzard--Diamond--Jarvis conjecture (\cite{bdj}, proved
in~\cite{blggU2}, \cite{GLS12}, and \cite{geekisin}), and this information is an important guiding tool
in the work of~\cite{Breuil_Paškūnas_2012}. In particular, given
$K/\Qp$ unramified, and a generic Galois representation
$\rhobar:G_K\to\GL_2(\Fpbar)$, 
\cite{Breuil_Paškūnas_2012} constructs an infinite family of
irreducible $\Fpbar$-representations of $\GL_2(K)$ whose $\GL_2(\cO_K)$-socles
are as predicted by \cite{bdj},
and conjectures
that if $\rbar$ is a globalisation of $\rhobar$ to a Galois
representation over some totally real field, then the $\rbar$-part of
the cohomology of an appropriate Shimura curve contains one of these
representations. 

Although the representations of~\cite{Breuil_Paškūnas_2012} 
are constructed so as to have the correct socle (in the sense
of \cite{bdj}), proving that
the cohomology contains such a representation lies deeper than the
Buzzard--Diamond--Jarvis conjecture, and amounts to proving that
certain subspaces of the mod $p$ cohomology at pro-$p$ Iwahori level are
stable under the action of a Hecke
operator. In~\cite{breuillatticconj} Breuil develops two approaches to
this problem, in the process making two conjectures of independent
interest. The main results of this paper are proofs of these
conjectures, as well as some generalisations. As a consequence, we
demonstrate that the cohomology contains one of the representations
constructed in~\cite{Breuil_Paškūnas_2012}.

The subspaces of the cohomology under consideration are those obtained
by fixing the character through which the Iwahori subgroup acts on the
cohomology at pro-$p$ Iwahori level, and one approach to the
Hecke-stability problem would be to show that these spaces are
one-dimensional; that is, to prove a multiplicity one result at
Iwahori level. In~\cite{dembeleappendix} Demb\'el\'e gives
computational evidence that such a multiplicity one result should
hold. One of our main results is the following (Theorem~\ref{thm:multiplicity one for the usual lattice, including the I_1
    invariants statement}).
  \begin{ithm}
    \label{ithm: dembele conjecture intro version}Suppose that $p\ge
    5$ and the usual Taylor--Wiles condition holds. Then the
    multiplicity one conjecture of~{\em \cite{dembeleappendix}} holds.
  \end{ithm}

  The other approach taken in~\cite{breuillatticconj} is via a
  $p$-adic local-global compatibility conjecture for tamely potentially
  Barsotti--Tate representations. To be precise, let $F$ be a totally
  real field in which $p$ is inert (in fact, we only need $p$ to be
  unramified, but we assume here that $p$ is inert for simplicity of
  notation), and fix a finite extension $E/\Qp$ (our field
  of coefficients) with ring of integers
  $\cO$ and uniformiser $\varpi_E$. Write
  $f:=[F:\Q]$, and let $k$ be the residue field of $F_p$. Let $\rhobar:G_F\to\GL_2(\F)$
  be an absolutely irreducible modular representation with $\rhobar|_{G_{F_p}}$
  generic (see Definition~\ref{defn: generic local Galois representation}), and
  consider \[M_{\rhobar}:=\varinjlim_{K_p}M(K_pK^p,\cO)_{\rhobar}\]where
  $M(K_pK^p,\cO)_{\rhobar}$ is the $\rhobar$-part of the cohomology
  of an appropriate Shimura curve at level $K_pK^p$, where $K_p$ is a
  compact open subgroup of $\GL_2(F_p)$. This has a natural action of
  $\GL_2(F_p)$, and we consider some subrepresentation $\pi$ of
  $M_{\rhobar}[1/p]$ which is a tame principal series
  representation. Then $M_{\rhobar}$ determines a lattice in $\pi$,
  and in particular, a lattice in the tame principal series type
  inside $\pi$ (that is, the smallest $\GL_2(\cO_{F_p})$-representation
  inside $\pi$). There are typically many homothety classes of
  $\GL_2(\cO_{F_p})$-stable lattices inside this type, and there is no
  \emph{a priori} reason that we are aware of to suppose that the
  lattice is a local invariant of the situation. Indeed, it is easy to
  check that classical local-global compatibility applied to $\pi$
  does not in general give enough information to determine the
  lattice.

  However, the main conjecture of~\cite{breuillatticconj} is that the
  lattice is determined by the local $p$-adic Galois representation in
  $M_{\rhobar}$ corresponding to $\pi$ in an explicit way. The local
  Galois representation is potentially Barsotti--Tate of tame
  principal series type,
  and various parameters can be read off from its Dieudonn\'e
  module. Breuil classified the homothety classes of lattice in the
  tame principal series types in terms of similar parameters, and
  conjectured that the parameters for the lattice in $\pi$ agree with
  those for the corresponding Galois representation. Our other main
  result confirms this conjecture (see Theorem~\ref{thm:Breuil lattice conjecture}).
  \begin{ithm}
    \label{ithm: intro version of Breuil's lattice conjecture}Breuil's
    conjecture holds if $p\ge 5$ and the usual Taylor--Wiles condition
    holds; an analogous result holds for tame cuspidal types.
  \end{ithm}
We note that in the cuspidal case, we use an extension to (most) tame cuspidal types of Breuil's
classification of lattices, which in turn uses
a generalisation of some of the results of~\cite{Breuil_Paškūnas_2012}
to irregular weights. These results are presented in
Section~\ref{sec:lattices in tame types}, and may be of independent interest.

Our main tools in proving Theorems~\ref{ithm: dembele conjecture intro
  version} and~\ref{ithm: intro version of Breuil's lattice
  conjecture} are the geometric Breuil--M\'ezard Conjecture, and its
relation to the Taylor--Wiles--Kisin patching method. 
We have hopes that our arguments can be made to apply more generally,
e.g.~to 
show for arbitrary types that the lattices induced on them
by embeddings into cohomology arising from global $p$-adic Galois
representations are in fact completely
determined by the corresponding local $p$-adic Galois representations
(though perhaps without
giving an explicit formula for the lattices).  In the remainder
of the introduction, we outline our methods in more detail.

\subsection{Lattices and their gauges}

Most of our efforts are directed towards proving Theorem~\ref{ithm: intro
version of Breuil's lattice conjecture}, while Theorem~\ref{ithm:
dembele conjecture intro version} will come as something of a byproduct of our
method of proof via Taylor--Wiles--Kisin patching, much as Diamond in 
\cite{MR1440309} was able to deduce mod $p$ multiplicity one results as a corollary
of a patching argument.

Let $\sigma$ be a tame type. The basic problem that we face is to determine the lattice $\sigma^{\circ}$ induced by integral
cohomology  when we are given an embedding
$\sigma \hookrightarrow M_{\rhobar}[1/p]$ arising from the local
factor at $p$ of the automorphic representation generated by some Hecke eigenform
whose associated Galois representation lifts $\rhobar$.
It turns out that such a lattice can be encoded in a convenient way
by what we call its {\em gauge}. Briefly, let $\sigmabar$ be the reduction mod $\unif_E$ 
of $\sigma$ (which is well-defined up to semisimplification). Then each Jordan--H\"older 
constituent of $\sigmabar$ occurs
with multiplicity one, and these consituents may be labelled by a
certain collection $\cP$ of subsets of $\{0,\ldots, f-1\}$, 
which typically is equal to the full power set of $\{0,\ldots,f-1\}.$
For each $J \in \cP$, we may find a lattice $\sigma^{\circ}_J$ in $\sigma$,
uniquely determined up to homothety, whose cosocle equals the constituent
of $\sigmabar$ labelled by the subset $J$.
For each $J \in \cP$, we obtain a free rank one $\cO$-module
$\Hom_{\GL_2(k_v)}(\sigma^{\circ}_J, \sigma^{\circ}) \subset
\Hom_{\GL_2(k_v)}(\sigma,\sigma) = E,$
the collection of which we refer to as the gauge of $\sigma$ (see Definition~\ref{defn:gauge}).
The gauge of $\sigma^\circ$ determines $\sigma^\circ$ up to homothety (Proposition~\ref{prop:the gauge determines the lattice}),
and so our problem becomes that of determining the gauge of~$\sigma^\circ$. 

If $\lambda$ is the system of Hecke eigenvalues associated to our eigenform,
and $M_{\rhobar}[\lambda]$ denotes the $\lambda$-eigenspace in $M_{\rhobar}$,
then this amounts to understanding the collection of Hom-modules
$\Hom_{\GL_2(\cO_F)}\bigl(\sigma^{\circ}_J, M_{\rhobar}[\lambda]\bigr).$
Rather than studying these for a single eigenform, we consider this collection
simultaneously for a family of eigenforms whose associated Galois representations
lift $\rhobar$.  More precisely, we employ the patching method of Taylor--Wiles
\cite{Taylor--Wiles},
as further developed by Diamond 
\cite{MR1440309}
and Kisin \cite{kis04}.

\subsection{Patching functors}Our actual implementation of the
Taylor--Wiles method is a further refinement of the viewpoint adopted
in the papers \cite{kisinfmc,emertongeerefinedBM}, which patch not
only spaces of modular forms, but certain filtrations on them. We take
this idea to its natural conclusion, and patch together spaces of
modular forms with all possible coefficient systems simultaneously
(see Section~\ref{sec: cohomology}). The patched spaces of modular
forms obtained from the Taylor--Wiles method are naturally modules
over certain universal Galois deformation rings, and we adopt a more
geometric point of view, regarding them as coherent sheaves on deformation
spaces of Galois representations. 

In this way, we obtain what we call a \emph{patching functor}, which
is a functor from continuous representations of $\GL_2(\cO_{F_p})$ on
finite $\cO$-modules to coherent sheaves on deformation spaces. These
functors satisfy a number of natural properties (for example, certain
of the sheaves obtained are maximal Cohen--Macaulay over their
support, and this support is specified by a compatibility with the
local Langlands correspondence and the Buzzard--Diamond--Jarvis
conjecture), and we are able to define an abstract notion of a
patching functor, with no reference to the Taylor--Wiles method. (One
motivation for this is in order to clarify exactly which properties of
these functors are used in our main theorems, and another is that it
is sometimes possible to construct patching functors without using the
Taylor--Wiles method, \emph{cf.}\ ~\cite{EGP}.)

In the particular case at hand, this construction realizes (the duals
to) the above Hom-modules as the fibres of certain coherent sheaves
(the patched modules) over the deformation space which parameterizes
deformations of $\rhobar|_{ G_{F_p}}$ which are tamely potentially
Barsotti--Tate of type $\sigma$. We can then reinterpret the problem
of understanding the above Hom-modules as that of understanding these
coherent sheaves, and in particular of understanding the natural
morphisms between them induced by various inclusions between the
$\sigma^\circ_J$. We do this by exploiting an explicit description of
the deformation space, and the geometric Breuil--M\'ezard philosophy
of~\cite{emertongeerefinedBM}. 

In Section~\ref{sec: deformation spaces} we give our explicit
description of the  deformation
space, showing that it is the formal spectrum of a power series ring over
the complete local ring
$\cO\llbracket X_1,Y_1,\dots,X_r,Y_r\rrbracket/(X_iY_i-p)_{i=1,\ldots,r}$, where $r$ depends on
$\rhobar$ via the Buzzard--Diamond--Jarvis conjecture. (In the
principal series case this is proved in~\cite{BreuilMezardRaffinee},
and we extend this computation to the cuspidal case by making use of a
base change result coming from Taylor--Wiles patching
via~\cite{emertongeerefinedBM}.) The $X_i$ and $Y_i$ have an
interpretation in terms of the Dieudonn\'e modules attached to the
potentially Barsotti--Tate lifts of $\rhobar|_{G_{F_p}}$ of type
$\sigma$, and using this interpretation we can reduce Breuil's
conjecture to understanding how the $X_i$ and $Y_i$ act on our
coherent sheaves. 

We do this in Section~\ref{sec: proof of the conjecture}; the key is
to show that the cokernels of various of the morphisms induced by the
inclusions between the $\sigma^\circ_J$ are annihilated by certain
products of the $X_i$ and $Y_i$. Using the Cohen--Macaulay properties
of the patching functors, this reduces to showing that the supports of
certain cokernels are annihilated by these products, which in turn
reduces to gaining a concrete understanding of the components of the
special fibre of the deformation space in terms of the Serre weights
specified by the Buzzard--Diamond--Jarvis conjecture. It follows from
the results of~\cite{emertongeerefinedBM} that these components are
identified with the special fibres of certain Fontaine--Laffaille
deformation spaces, and we need to make this correspondence completely
explicit. We do this via an explicit $p$-adic Hodge theoretic
computation, following~\cite{BreuilMezardRaffinee} (and again reducing
to a special case via a base change argument).

\subsection{Freeness}In order to prove Conjecture~\ref{ithm: dembele
  conjecture intro version}, we prove that certain of our patched
coherent sheaves are free. By an argument originally due to
Diamond~\cite{MR1440309}, this follows from the maximal
Cohen--Macaulay property if our deformation space is regular. However,
from the description above we see that this can only happen if
$r=1$. In order to prove freeness in the case that $r>1$, we exploit
the explicit description of the patching functor that we have obtained
in the above work, and effectively proceed by induction on $r$. By
Nakayama's lemma, it suffices to prove freeness in characteristic
$p$. The functorial nature of our patching construction means that we
can compare patched sheaves corresponding to the reductions modulo
$\varpi_E$ of varying tame types $\sigma$, and glue together the
patched sheaves in general from those types which give $r=1$, where we
know freeness. 

\subsection{The generic Buzzard--Diamond--Jarvis Conjecture}As a
byproduct of our constructions, we are able to give a new proof of the
Buzzard--Diamond--Jarvis conjecture in the generic case (and indeed, we
expect that this proof would allow a slight relaxing of the global
hypotheses in the generic case when $p=5$; see Remark~\ref{rem:
  removing $A_5$ restriction}). This proof is related to the approach
taken in~\cite{geebdj}, but avoids both the combinatorial arguments of
that paper and the dependence on the modularity lifting theorems
proved in~\cite{kis04,MR2280776}. The method is closely related to
the arguments used in~\cite{tay} to avoid the use of Ihara's lemma in
proving modularity lifting theorems. The idea is that our deformation
spaces have irreducible generic fibres, but their special fibres have
many irreducible components; from this and the properties of our patching
functors, it follows that a patched module which is supported on the
generic fibre of the deformation space is necessarily supported on
each irreducible component of the special fibre. 

As mentioned above, these components correspond to different Serre
weights, and if we choose our tame type $\sigma$ to maximise the
number of components in the special fibre, we are able to prove
modularity in every weight predicted by the Buzzard--Diamond--Jarvis
conjecture. (In contrast, the arguments of~\cite{geebdj} choose tame
types to ensure that the special fibres are irreducible, and use a
base change argument and the results of~\cite{kis04,MR2280776} to
prove modularity in each tame type.) We explain this argument in
Section~\ref{sec: BDJ}, and in Appendix~\ref{subsec:avoiding GK} we
explain how to remove the dependence on~\cite{kis04,MR2280776} which
is implicit in our use of the results
of~\cite{emertongeerefinedBM,geekisin} in proving our explicit
descriptions of the tamely Barsotti--Tate deformation spaces.

\subsection{A Morita-theoretic interpretation}
A Morita-theoretic interpretation of our arguments
was suggested to us by David Helm.  Namely, if
we write
$$P := \oplus_{J \in \cP} \sigma^{\circ}_J,$$
then $P$ is
a projective generator in a certain abelian category of $\cO[\GL_2(k_v)]$-modules,
namely the full abelian subcategory of the category of all such modules
generated by the collection of all lattices in $\sigma$.
Our theory of gauges is thus a special case of general Morita theory,
which shows that if $\cE := \End_{\GL_2(k_v)}(P),$ then
$\sigma$ may be recovered from the $\cE$-module $\Hom_{\GL_2(k_v)}(P,\sigma)$
via the formula
$$\sigma = \Hom_{\GL_2(k_v)}(P,\sigma)\otimes_{\cE} P.$$
Furthermore, Theorem~\ref{thm: the output of gauges.tex}
may be regarded as an explicit description of the $\cO$-algebra $\cE$.

Since $\Hom_{\GL_2(k_v)}(P,\sigma)$ may be recovered as the fibre
of the patched module for $P$ at an appropriately chosen point, in order
to describe this $\cE$-module, it suffices to describe
the $\cE$-module structure on the corresponding patched module,
and one interpretation of our results is that they give an
explicit description of this patched module with its $\cE$-action (and show, in particular,
that it is purely local).

\subsection{Organisation of the paper}For the convenience of the
reader, we now briefly outline the structure of the paper. In
Section~\ref{sec:Galois reps and weights} we recall the definition of
a generic Galois representation, explain the Buzzard--Diamond--Jarvis
conjecture for these representations, and give some definitions
relating to our later base change arguments. In Section~\ref{sec:base
  change} we give a similar treatment of tame
types. Section~\ref{sec:lattices} contains some general results on the
structure of lattices in irreducible representations whose reductions
mod $p$ have all Jordan--H\"older factors occurring with multiplicity
one. In Section~\ref{sec:lattices in tame types} we specialise to give
an explicit description of the lattices in tame types, extending the
results of~\cite{breuillatticconj} for tame principal series types. 

In Section~\ref{sec: cohomology} we define the notion of a patching
functor, and use the Taylor--Wiles--Kisin method to construct examples
from the cohomology of Shimura curves, as well as from spaces of
algebraic modular forms on definite quaternion algebras and definite
unitary groups. Our $p$-adic Hodge theoretic arguments appear in
Section~\ref{sec: deformation spaces}, which gives the explicit
description of the tame Barsotti--Tate deformation spaces (and their
special fibres) explained above. In Section~\ref{sec: proof of the
  conjecture} we prove Theorem~\ref{ithm: intro version of Breuil's
  lattice conjecture}, and in Section~\ref{sec:
  BDJ} we explain our new proof of the Buzzard--Diamond--Jarvis
conjecture in the generic case. Finally, Theorem~\ref{ithm: dembele
  conjecture intro version} is proved in Section~\ref{sec: freeness and
  multiplicity one}. 

There are two appendices; in Appendix~\ref{sec:unipotent-FL} we
explain the connection between unipotent Fontaine--Laffaille theory
and the theory of unipotent $\varphi$-modules, and in
Appendix~\ref{sec: appendix on geom BM} we discuss the use of the
geometric Breuil--M\'ezard conjecture in our arguments in greater
detail. In particular, we explain how to prove our main theorems
without making use of the results of~\cite{kis04,MR2280776}, and we
give an argument that hints at a deeper relationship between the
structures of the lattices in tame types and of the tamely potentially
Barsotti--Tate deformation spaces.

\subsection{Notation and Conventions}
If $K$ is any field let $G_K = \Gal(\overline{K}/K)$, where
$\overline{K}$ is a fixed separable closure of $K$.  Suppose that $p$
is prime.  When $K$ is a finite extension of $\Qp$, let $I_{K} \subset
G_{K}$ denote the inertia subgroup.

Let $F_v$ be a finite unramified extension of $\Qp$. (Usually,
$F_v$ will be the completion of a number field $F$ at a finite place $v$,
but it will be convenient for us to use this notation more generally.) We write $k_v$ to denote the residue field 
of~$v$, and let $f = [k_v:\F_p]$.  We write $\cO_{F_v}$ for $W(k_v)$, so
that $F_v$ is the fraction field of $\cO_{F_v}$.
Write $f=[F_v:\Qp]$, $q_v=p^f$, and $e = p^f-1$.

Let $L_v$ denote the quadratic unramified
extension of $F_v$, with residue field $l_v$, the quadratic extension
of $k_v$.

Let $E$ be a
finite extension of $\Qp$ (which will be our field of coefficients)
with ring of integers $\cO$, uniformiser $\varpi_E$, and residue field $\F$. We will assume
throughout the paper that $E$ is sufficiently large
in the sense that there exists an embedding of $l_v$ into $\F$.  We will sometimes make specific choices for $E$;
in particular, at certain points in our arguments it will be important
to take $E=W(\F)[1/p]$.

Let $\varepsilon : G_{F_v} \to \cO^{\times}$ denote
the cyclotomic character.

If $V$ is a de Rham representation of $G_{F_v}$ over
$E$ and $\emb$ is an embedding $F_v \into E$ then the multiset
$\HT_\emb(V)$ of \emph{$\emb$-labeled Hodge--Tate weights} of $V$ is
defined to contain the integer $i$ with multiplicity $$\dim_{E} (V
\otimes_{\emb,K} \widehat{\overline{F}}_v(i))^{G_{F_v}},$$ with the usual notation
for Tate twists.  Thus for example
$\HT_\emb(\varepsilon)=\{ -1\}$. 

If $L$ is a local or global field, we write $\Art_L$ for the Artin
reciprocity map, normalised so that the global map is compatible with
the local ones, and the local map takes a uniformiser to a geometric
Frobenius element. If $L$ is a number field and $w$ is a finite place
of $L$, we will write $\Frob_w$ for a geometric Frobenius element at
$w$.

If $V$ is a representation of some group which
has a composition series, we write $\JH(V)$ for the
Jordan--H\"older factors of $V$.

When discussing Galois deformations, we will use the terms ``framed deformation'' (which originated
in~\cite{kis04}) and ``lifting'' (which originated in~\cite{cht})
interchangeably.

If $\sigma$ is a finite-dimensional $E$-representation
of a group $\Gamma$, then $\sigma^{\circ}$ (possibly with
additional decorations) will denote a  
$\Gamma$-stable $\cO$-lattice in $\sigma$ (sometimes a general
such lattice, and sometimes a particular lattice that has been
defined in the context under consideration). We will then let
$\sigmabar^\circ$ (with the same decorations as $\sigma^{\circ}$,
if appropriate) denote the reduction of $\sigma^\circ$ modulo
$\varpi_E$. We will write $\sigmabar$ for some choice of
$\sigmabar^\circ$; but we will only use the notation $\sigmabar$ in situations
where the choice of a particular $\sigmabar^\circ$ is unimportant, for
example when discussing the Jordan--H\"older factors of $\sigmabar$.
We note that $\sigmabar^{\circ}$ is finite-dimensional over $\F$, and
so admits both a socle filtration and a cosocle filtration as a
$\Gamma$-representation, each of finite length. The $\Gamma$-representation
$\sigma^{\circ}$ itself is of infinite length, and does not admit a
socle filtration.  However, it does admit a cosocle filtration;
furthermore, the cosocles of $\sigma^{\circ}$ and of
$\sigmabar^{\circ}$ coincide.

If $A$ is a finitely generated $\cO$-module, we write $A^*$ for the
Pontrjagin dual of~$A$. If $A$ is torsion, then we will sometimes write $A^\vee$
for $A^*$. If $A$ is a finite free $\cO$-module, we will write $A^\vee$ for its
$\cO$-dual. Note that with these conventions and in the notation of the previous
paragraph, we have
$(\sigmabar^\circ)^\vee\cong\overline{(\sigma^\circ)^\vee}$.

Throughout the paper, we will number
our socle and cosocle filtrations starting at $0$, so that for example
the $0$-th layer of the cosocle filtration is the cosocle.

We let $\cS$ denote the set $\{0,\ldots,f-1\},$ which we sometimes
identify with the quotient $\Z/f\Z$ in the evident way. (Typically this 
identification will be invoked by the expression that certain
algebraic expressions involving the elements of $\cS$ are to be understood
{\em cyclically}, which is to say, modulo $f$.)
We draw the reader's attention to two involutions related to $\cS$
that will play a role in the paper. Namely,
in Subsection~\ref{subsec: specific lattices and gauges},
and again in Section~\ref{sec: proof of the conjecture},
we will fix a particular subset $\Jbase$ of $\cS$, 
and will let $\base$ denote the involution on the power
set of $\cS$ defined by $J \mapsto  J \triangle \Jbase,$
while in Section~\ref{sec: deformation spaces} we will let
$\iota$ denote the involution on $\cS$ itself defined by
$j \mapsto f- 1 - j$.

If $F$ is a totally real field and $\pi$ is a regular cuspidal automorphic
representation of $\GL_2(\A_F)$, then we fix an isomorphism
$\imath:\Qpbar\to\C$ and associate to $\pi$ the $p$-adic Galois representation
$r_{p,\imath}(\pi):G_F\to\GL_2(\Qpbar)$, normalised as in~\cite[Thm.\
2.1.1]{BLGGT}. (Note that we will use the symbol $\imath$ in other contexts in
the body of the paper, but we will not have to mention the isomorphism
$\Qpbar\to\C$ again, so we have denoted it by $\imath$ here for ease of
reference to~\cite{BLGGT}.)

\subsection{The inertial local Langlands correspondence}\label{subsec:
  inertial local Langlands}Let $L$ be a local field of characteristic
zero, so that $L$ is a finite extension of $\Ql$ for some prime
$l$. An {\em inertial type} for $L$ is a two-dimensional
$E$-representation $\tau$ of $I_L$ with open kernel, which may be
extended to $G_L$. Henniart's appendix to \cite{breuil-mezard}
associates a finite dimensional irreducible $E$-representation $\sigma(\tau)$ of
$\GL_2(\cO_L)$ to $\tau$; we refer to this association as the {\em
  inertial local Langlands correspondence}. To be precise, if $l=p$
then we normalise the correspondence as in Theorem 2.1.3 of
\cite{geekisin}, and if $l\ne p$ we use the version recalled in
Section~5.2 of \cite{geekisin}. (The only difference between the two
cases is that if $l\ne p$ we consider representations with non-trivial
monodromy; the definition only differs in the case of scalar inertial
types, which in the case $l=p$ correspond to twists of the trivial
representation, and in the case $l\ne p$ to twists of the small
Steinberg representation.)

If $l\ne p$, then let $D$ denote the unique nonsplit quaternion
algebra with centre $L$. We say that an inertial type $\tau$ is {\em
  discrete series} if it is either irreducible or scalar. In this
case, a variant of Henniart's construction which is explained in
Section~5.2.2 of \cite{geekisin} associates a finite dimensional
irreducible $E$-representation $\sigma_D(\tau)$ of $\cO_D^\times$ to
$\tau$.

\section{Galois representations and Serre weights}\label{sec:Galois
  reps and weights}
In this section we recall various 
definitions and facts about the weight part of Serre's conjecture, largely following
Section~4 of \cite{breuillatticconj}.  We fix once and for all an
embedding $\embb_0 : k_v \into \F$, and recursively define embeddings
$\embb_i : k_v \into \F$ for
$i \in \Z$ by $\embb_{i+1}^p = \embb_i$.  Let $\omega_f$ denote the
fundamental character of $I_{F_v}$ of niveau $f$ associated to
$\embb_0$; that is, $\omega_f$ is the \emph{reciprocal} of the character
defined by the composite
$$ I_{F_v} \xrightarrow{\Art^{-1}_{F_v}} \cO_{F_v}^{\times}
\longrightarrow k_v^{\times} \xrightarrow{\embb_0} \F^{\times},$$ so
that for example if $f=1$ then $\omega_1=\varepsilonbar^{-1}$.
Let $\omega_{2f}$ be a
fundamental character of niveau $2f$ similarly associated to an embedding 
$\embb'_0 : l_v^{\times} \into \F$ extending $\embb_0$.  Define
$\embb'_i$ for $i \in \Z$ by analogy with the $\embb_i$.

\subsection{Basic notions}
We recall some basic notions related to mod $p$ local Galois 
representations and their associated Serre weights, as well as
the main result of \cite{GLS12}.

\begin{defn}
  \label{defn: generic local Galois representation}We say that a
  continuous representation $\rhobar:G_{F_v}\to\GL_2(\F)$ is
  \emph{generic} if and only if $\rhobar|_{I_{F_v}}$ is isomorphic up
  to twist to a representation of one of the following two forms.
  \begin{enumerate}
  \item  $
    \begin{pmatrix}
      \omega_f^{\sum_{j=0}^{f-1}(r_j+1)p^j}&*\\0&1
    \end{pmatrix}$ with $0\le r_j\le p-3$ for each $j$, and the $r_j$
    not all equal to $0$, and not all equal to $p-3$.
  \item  $
    \begin{pmatrix}
      \omega_{2f}^{\sum_{j=0}^{f-1}(r_j+1)p^j}&0\\0& \omega_{2f}^{q\sum_{j=0}^{f-1}(r_j+1)p^j}
    \end{pmatrix}$ with $1\le r_0\le p-2$, and $0\le r_j\le p-3$ for
    $j>0$.
  \end{enumerate}

\end{defn}

It is easily checked that this definition is independent of
the choice of $\embb_0$ and $\embb'_0$. The
following lemma is also easily checked.
\begin{lem}
  \label{lem:base change of generic is generic}
   If $p>3$ and $\rhobar$ is generic,
  then $\rhobar|_{G_{L_v}}$ is also generic.
\end{lem}
\begin{rem}
  \label{rem: excluding p=3 for generic}If $p=3$ then there are no
  reducible generic representations, and in the above notation the
  only irreducible generic representations have $r_0=1$ and $r_j=0$
  for $j>0$. The failure of Lemma \ref{lem:base change of generic is generic} to hold in this case means that we will not
  consider the case $p=3$ in our main results.
\end{rem}
By definition, a \emph{Serre weight} is an irreducible
$\F$-representation of $\GL_2(k_v)$. Concretely, such a
representation is of the form
\[\sigmabar_{\vec{t},\vec{s}}:=\otimes^{f-1}_{j=0}
(\det{}^{t_j}\Sym^{s_j}k_v^2) \otimes_{k_v,\embb_{-j}} \F,\]
where $0\le s_j,t_j\le p-1$ and not all $t_j$ are equal to
$p-1$.

\begin{defn}
  \label{defn: regular weight}We say that a Serre weight
  $\sigmabar_{\vec{t},\vec{s}}$ is \emph{regular} if no $s_j$ is equal
  to~$p-1$.
\end{defn}
\begin{rem}
  This is the same definition of regular Serre weight as that in
  \cite{geekisin}; it is less restrictive than the definition used in
  \cite{geebdj}.
\end{rem}
We let $\cD(\rhobar)$ be the set of Serre weights associated to
$\rhobar$ by \cite{bdj}; this definition is recalled in Section~4 of
\cite{breuillatticconj}.  While we will not need the full details of
the definition of $\cD(\rhobar)$, it will be convenient to recall
certain of its properties, and in particular its definition in the case that
$\rhobar$ is semisimple; see \cite{bdj} or Section~4 of
\cite{breuillatticconj} for further details. We will also make use of
the characterisation of $\cD(\rhobar)$ of Theorem \ref{thm:GLS main
  thm} below (at least implicitly, via its use in \cite{geekisin}).

We always have $\cD(\rhobar)\subset\cD(\rhobar^{\semis})$. If
$\rhobar$ is reducible and semisimple, then
$\sigmabar_{\vec{t},\vec{s}}\in\cD(\rhobar)$ if and only if there is a
subset $K\subseteq\cS:=\{0,\dots,f-1\}$ such that we can
write
\[
\rhobar|_{I_{F_v}}\cong \omega_f^{\sum_{j=0}^{f-1}t_jp^j} \otimes
    \begin{pmatrix}
      \omega_f^{\sum_{j\in K}(s_j+1)p^j}&0\\0&\omega_f^{\sum_{j\notin K}(s_j+1)p^j}
    \end{pmatrix}.\]

If $\rhobar$ is irreducible, we set $\cS'=\{0,\dots,2f-1\}$, and we
say that a subset $K\subseteq \cS'$ is \emph{symmetric} (respectively
\emph{antisymmetric}) if $i\in K$ if and only if $i+f\in K$
(respectively $i+f\notin K$) for all $i\in\cS'$, numbered
cyclically. Then $\sigmabar_{\vec{t},\vec{s}}\in\cD(\rhobar)$ if and
only if there is an antisymmetric
subset $K\subset\cS'$ such that we can write \[
\rhobar|_{I_{F_v}}\cong \omega_f^{\sum_{j=0}^{f-1}t_jp^j} \otimes
    \begin{pmatrix}
      \omega_{2f}^{\sum_{j\in K}(s_j+1)p^j}&0\\0&\omega_{2f}^{\sum_{j\notin K}(s_j+1)p^j}
    \end{pmatrix}.\]

The following lemma is immediate from the
above description of $\cD(\rhobar)$.
\begin{lem}\label{lem: generic representations have no weight p-1}If $\rhobar$ is generic and
  $\sigmabar\in\cD(\rhobar)$, then $\sigmabar$ is regular; in
  addition, $\sigmabar$ is not of the form $\sigmabar_{\vec{t},\vec{0}}$.
\end{lem}
\begin{defn}
  \label{defn:hodge type}
Let $\emb_i : F_v \into E$ be the unique embedding that lifts
$\embb_i$.   If $V$ is a two-dimensional de Rham representation of
$G_{F_v}$ over $E$ and $\HT_{\emb_{-i}}(V) = \{a_i,b_i\}$
with $a_i \le
b_i$ for $0 \le i \le f-1$, we say that $V$ has \emph{Hodge type
  $(\vec{a},\vec{b})$}. If $V$ has Hodge type $(\vec{0},\vec{1})$, we
will also say that $V$ has \emph{Hodge type $0$}.

If $\sigmabar = \sigmabar_{\vec{t},\vec{s}}$, we also say that $V$ has
\emph{Hodge type $\sigmabar$} if $V$ has Hodge type $(\vec{t},\vec{t}+\vec{s}+\vec{1})$.
\end{defn}
The following theorem is the main result of \cite{GLS12}.

\begin{thm}
  \label{thm:GLS main thm}We have $\sigmabar\in\cD(\rhobar)$ if and only
  if $\rhobar$ has a crystalline lift of Hodge type $\sigmabar$.
\end{thm}

\subsection{Base change for Galois representations and Serre weights}
\label{subsec:base change Serre weights}
In the sequel we will be studying base change from $F_v$ to $L_v$,
and it will be convenient to have standardized notation for this process.

In the case of Galois represenations, base change is simply restriction,
and thus,
if $\rhobar:G_{F_v}\to\GL_2(\F)$, we write $\BC(\rhobar)$ for
$\rhobar|_{G_{L_v}}$.

For Serre weights the definition of base change is slightly more involved.
Suppose given a Serre weight for $\GL_2(k_v)$, say
$\sigmabar_{\vec{t},\vec{s}}:=\otimes^{f-1}_{j=0}
(\det{}^{t_j}\Sym^{s_j}k_v^2) \otimes_{k_v,\embb_{-j}} \F$, and
a Serre weight for $\GL_2(l_v)$, say
$\Sigmabar_{\vec{t'},\vec{s'}}:=\otimes^{2f-1}_{j=0}
(\det{}^{t'_j}\Sym^{s'_j}l_v^2) \otimes_{l_v,\embb'_{-j}} \F$. We say
that $\Sigmabar=\Sigmabar_{\vec{t'},\vec{s'}}$ is the \emph{base
  change} of $\sigmabar=\sigmabar_{\vec{t},\vec{s}}$, and write
$\Sigmabar=\BC(\sigmabar)$, if we have
$t'_j=t_{\overline{j}},s'_j=s_{\overline{j}}$ for all $0\le j\le 2f-1$
(where $\overline{j}$ is $j\pmod{f}$). Note that not every Serre
weight for $\GL_2(l_v)$ is the base change of a Serre weight for
$\GL_2(k_v)$, and that every Serre weight for $\GL_2(k_v)$ has a
unique base change to $\GL_2(l_v)$.

If $\rhobar$ is generic, and  $\sigmabar$ is a
  Serre weight, then $\sigmabar\in\cD(\rhobar)$ if and only if
  $\BC(\sigmabar)\in\cD(\rhobar|_{G_{L_v}})$.  The ``only if''
  direction (which we will use in Section~\ref{sec:base change}) is immediate from the definitions.  It is presumably
  possible to prove the ``if'' direction from the definitions as well; we deduce it
  from the geometric Breuil--M\'ezard conjecture below (see Corollary~\ref{cor: a weight if and only if a weight after base
    change}).

\section{Tame types, Serre weights, and base change}\label{sec:base change}
In this section we describe various results about 
tame types, and we  explain how these results interact with base
change for tame types, Serre weights and Galois representations. We
restrict ourselves to quadratic unramified base change, which is all
that we will need in this paper.

\subsection{Tame types}\label{subsec:tame types}
In this
section we write $K=\GL_2(\cO_{F_v})$, and let $I$ denote the standard Iwahori subgroup of $K$ (i.e.\
the one whose reduction mod $p$ consists of upper-triangular
matrices). We consider tame types for $K$;
that is, irreducible $E$-representations of $K$
that arise by inflation from the irreducible $E$-representations of
$\GL_2(k_v)$. By the results recalled in Section~1 of
\cite{MR2392355}, these representations are either principal series,
cuspidal, one-dimensional, or twists of the Steinberg
representation. In this paper we will only consider the principal
series and cuspidal types, and all tame types occurring will be
assumed to be of this kind without further comment. Under the inertial
local Langlands correspondence of Section~\ref{subsec:
  inertial local Langlands}, this means that we will only consider
non-scalar tame inertial types.

\begin{lem}
  \label{lem:weights in tame types have multiplicity one}
Any tame type $\sigma$ is residually multiplicity free
{\em (}in the sense that the Jordan--H\"older constituents of
$\sigmabar$ each appear with multiplicity one{\em )}
and essentially
self-dual {\em (}in the sense that there is an isomorphism
$\sigma^{\vee} \otimes \chi \cong \sigma$ for some character
$\chi: K \to \cO^{\times}${\em )}. All of the Jordan--H\"older factors
$\sigmabar_i$ of $\sigmabar$ also satisfy $\sigmabar_i^{\vee} \otimes \chibar \cong \sigmabar_i$. Furthermore, $\sigma$ has a model over
an unramified extension of $\Qp$.
\end{lem}
\begin{proof}
  That $\sigma$ is residually multiplicity free follows at once from
  Propositions 1.1 and 1.3 of \cite{MR2392355}, and the essential
  self-duality follows from the character tables for $\sigma$ in
  Section~1 of \emph{ibid.}\ These character tables also show that the
  character of $\sigma$ is defined over an unramified extension of
  $\Qp$, and that $\sigma$ has dimension prime to~$p$; thus the Schur
  index of $\sigma$ is $1$ by Theorem 2(a) of~\cite{MR0466330}, and
  $\sigma$ is defined over the field of definition of its
  character.
\end{proof}

We refer to these representations as tame types because they
correspond to tame inertial types under the inertial local Langlands
correspondence of Henniart's appendix to \cite{breuil-mezard}. It is
easy to see that a tame inertial type is reducible, and is either the
sum of two copies of a single character which extends to $G_{F_v}$ (the {\em scalar} case), the
sum of two distinct characters which extend to $G_{F_v}$ (the {\em principal
  series} case), or a sum $\psi\oplus\psi^{q_v}$ where $\psi$ does
not extend to $G_{F_v}$ but does extend to $G_{L_v}$ (the {\em
  cuspidal} case). In the principal series case, if
$\tau=\eta_1\oplus\eta_2$ then $\sigma(\tau)$ is the representation
denoted $\sigma(\eta_1\circ\Art_{F_v} \otimes\eta_2\circ\Art_{F_v} )$
below, while in the cuspidal case, if $\tau=\psi\oplus\psi^{q_v}$,
then $\sigma(\tau)$ is the representation denoted
$\Theta(\psi\circ\Art_{F_v} )$ below.  For typesetting reasons we will
always write $\sigmabar(\tau)$, $\Thetabar(\psi)$, etc.\ in lieu of
$\overline{\sigma(\tau)}$, $\overline{\Theta(\psi)}$, etc.  

\subsection{Principal series types}\label{ss:principal} We first consider tame principal
series types. We let $\eta,\eta':k_v^\times\to\cO^\times$ be two
multiplicative characters, and we write $\chi = \eta\otimes \eta':
I\to\cO^\times$ for the character \[
\begin{pmatrix}
  a&b\\pc&d
\end{pmatrix}\mapsto \eta(\overline{a})\eta'(\overline{d}).\] We write
$\chi^s:=\eta'\otimes\eta$, and we write $\sigma(\chi) := \Ind_I^K
\chi^s$, an $E$-vector space with an action of $K$.

We will assume from now on that $\eta\ne\eta'$, and we write
$c=\sum_{j=0}^{f-1}c_{j}p^j$, $0\le c_{j}\le p-1$, for
the unique integer $0<c_\chi<p^{f}-1$ such that
$\eta(\overline{x})\eta'^{-1}(\overline{x})=[\overline{x}]^{c_\chi}$,
where $\overline{x}\in k_v^\times$ and $[\overline{x}]$ is the Teichm\"uller
lift of $\overline{x}$. We let
$\cP_{\sigma(\chi)}$ be the collection of subsets of $\cS$ consisting of subsets
$J$
satisfying the conditions:
\begin{itemize}
\item if $j\in J$ and $j-1\notin J$ then $c_{j}\ne p-1$, and
\item if $j\notin J$ and $j-1\in J$ then $c_{j}\ne 0$,
\end{itemize}
where $j-1$ is taken to mean $f-1$ if $j=0$.

The Jordan--H\"older factors of $\sigmabar(\chi)$ are
parameterised by $\cP_{\sigma(\chi)}$ in the following fashion (see
Lemma 2.2 of \cite{Breuil_Paškūnas_2012} or Proposition 1.1 of
\cite{MR2392355}). For any $J\subseteq \cS$, we let $\delta_J$ denote
the characteristic function of $J$, and if $J \in \cP_{\sigma(\chi)}$
we define $s_{J,i}$ by
\[s_{J,i}=\begin{cases} p-1-c_{i}-\delta_{J^c}(i-1)&\text{if }i\in J \\
    c_{i}-\delta_J(i-1)&\text{if }i\notin J, \end{cases}\]
             and we set $t_{J,i}=c_{i}+\delta_{J^c}(i-1)$ if
             $i\in J$ and $0$ otherwise. Then we let
             $\sigmabar(\chi)_J :=\sigmabar_{\vec{t},\vec{s}}\otimes\eta'\circ\det$;
             the $\sigmabar(\chi)_J$ are precisely the
             Jordan--H\"older factors of $\sigmabar(\chi)$.  If $\tau
             = \eta \oplus \eta'$
             corresponds to $\sigma(\chi)$ under the inertial local
             Langlands correspondence, we will often write
             $\sigmabar(\tau)_J$ for $\sigmabar(\chi)_J$.  However,
             we caution that this notation depends implicitly on an
             ordering of the characters $\eta,\eta'$, namely
             $\sigmabar(\tau)_J = \sigmabar(\eta \otimes \eta')_J =
             \sigmabar(\eta' \otimes \eta)_{J^c}$.

\subsection{Cuspidal types}\label{ss:cuspidal} Let $\psi:l_v^\times\to \cO^\times$ be a
multiplicative character which does not factor through the norm
$l_v^\times\to k_v^\times$. We now recall the Jordan--H\"older factors
of the reduction modulo $p$ of the tame cuspidal type $\Theta(\psi)$,
following \cite{MR2392355}. We can write $\psi$ in the form
$[\overline{x}]^{(q+1)b+1+c}$ where $0\le b\le q-2$, $0\le c\le
q-1$. Write $c=\sum_{i=0}^{f-1}c_ip^i$
where
$0\le c_i \le p-1$.  If $J \subseteq \cS$ we set
$J_0=J\triangle\{f-1\}$, and we define $\cP_{\Theta(\psi)}$ to be the collection of
subsets of~$\cS$ consisting of those $J$ satisfying the conditions:
\begin{itemize}
\item if $j\in J$ and $j-1\notin J_0$ then $c_j\ne p-1$, and
\item if $j\notin J$ and $j-1\in J_0$ then $c_j\ne 0$.
\end{itemize}

The Jordan--H\"older factors of $\Thetabar(\psi)$ are
parameterised by $\cP_{\Theta(\psi)}$ in the following fashion.  (See
Proposition 1.3 of \cite{MR2392355}, but note that our
parameterisation is slightly different: the
Jordan--H\"older factor parameterised by $J$ here is parameterised by $J^c$
in  \cite{MR2392355}.)  For any $J \in
\cP_{\Theta(\psi)}$ we define $s_{J,i}$ by
\[s_{J,i}=\begin{cases}
                 p-1-c_i-\delta_{(J_0)^c}(i-1)&\text{if }i\in J \\
                 c_i-\delta_{J_0}(i-1)&\text{if }i\notin J, \end{cases}\]
             and we set $t_{J,i}=c_i+\delta_{J^c}(i-1)$ if
             $i\in J$ and $0$ otherwise. Then we let
             \[\Thetabar(\psi)_J:=\sigmabar_{\vec{t},\vec{s}}\otimes[\cdot]^{(q+1)b+\delta_{J}(0)\delta_{J}(f-1)+\delta_{J^c}(0)\delta_{J^c}(f-1)}\circ\det;\]
             the $\Thetabar(\psi)_J$ are precisely the Jordan--H\"older
             factors of $\Thetabar(\psi)$. 

We say that $\Theta(\psi)$ is \emph{regular} if there is some $i$ with
$0<c_i<p-1$, and \emph{irregular} otherwise. If $\tau$ is a cuspidal
tame inertial type, then we say that $\tau$ is regular if and only if
$\sigma(\tau)$ is regular.
\begin{lem}
  \label{lem: regular cuspidals allow relabelling}The cuspidal type
  $\Theta(\psi)$ is regular if and only if there is some $i\in\cS$
  such that $\{i,\dots,f-1\}\in\cP_{\Theta(\psi)}$ and $\{0,\dots,i-1\}\in\cP_{\Theta(\psi)}$. If
  $\Theta(\psi)$ is irregular, then there is a unique
  $J\in\cP_{\Theta(\psi)}$ such that $\Thetabar(\psi)_J$ is regular.
\end{lem}
\begin{proof}
  If $\Theta(\psi)$ is regular, then it is easy to see that any $i$ with
  $0<c_i<p-1$ works. The converse is also straightforward. If $\Theta(\psi)$ is
  irregular and $\Thetabar(\psi)_J$ is regular, we
  see that if $c_i=0$ then we must have $i-1\not\in J_0$, and if $c_i=p-1$
  then $i-1\in J_0$, so $J$ is uniquely determined.
\end{proof}

\subsection{Base change for tame types}\label{subsec:base change
  tame types}
Suppose
first that $\sigma(\chi)$
is a tame principal series type for
$\GL_2(\cO_{F_v})$, where $\chi=\eta\otimes\eta'$ for multiplicative
characters $\eta,\eta':k_v^\times\to\cO_{F_v}^\times$. Then we let
$\BC(\sigma(\chi))=\sigma(\chi')$, where
$\chi'=\eta\circ N_{l_v/k_v}\otimes\eta'\circ N_{l_v/k_v}$, a tame
principal series type for $\GL_2(\cO_{L_v})$. 

Suppose now that $\sigma(\tau)$ is a tame cuspidal type for
$\GL_2(\cO_{F_v})$, so that $\sigma(\tau)=\Theta(\psi)$, where
$\psi:l_v^\times\to\cO_{L_v}^\times$ is a multiplicative character
which does not factor through $N_{l_v/k_v}$ (see e.g.\  Section~1 of
\cite{MR2392355}). Then we let $\BC(\sigma(\tau))=\sigma(\chi')$, where
$\chi'=\psi\otimes\psi^q$, a tame principal series type for
$\GL_2(\cO_{L_v})$. Note that this definition depends on the $\psi$,
as opposed to only on $\sigma(\tau)$, so that the notation below for the
Jordan--H\"older factors also depends on this choice. However, all of
our main results are independent of this choice.

In either case we say that $\BC(\sigma(\tau))$ is
the \emph{base change} of $\sigma(\tau)$ to $\GL_2(\cO_{L_v})$. Again, not
every tame type for $\GL_2(\cO_{L_v})$ is the base change of a tame type
for $\GL_2(\cO_{F_v})$.

By construction, this definition is
compatible with the inertial
local Langlands correspondence, in the sense that if $\tau$ is 
a tame inertial type for $F_v$, and if $\BC(\tau)$ denotes 
$\tau$ regarded as a tame intertial type for $L_v$, via the
equality $I_{F_v} = I_{L_v}$,
then
$\BC(\sigma(\tau))=\sigma(\BC(\tau))$. 

If we write $\BC(\sigma(\tau))=\sigma(\eta\otimes\eta')$, then an
easy calculation shows that
$\eta(\eta')^{-1}=[\cdot]^{\sum_{i=0}^{2f-1}a_ip^i}$ where in the
principal series case $a_i =
c_{i}$ if $0 \le i \le f-1$  and $a_i = c_{i-f}$ if $f \le i
\le 2f-1$, while in the cuspidal case $a_i=c_i$ if
$0\le i\le f-1$ and $c_i=p-1-c_{i-f}$ if $f\le i\le 2f-1$.

Now, if $J\subseteq\cS$ we write $\BC_{\cusp}(J)$ for the subset of
$\{0,\dots,2f-1\}$ defined by taking $i\in\BC_{\cusp}(J)$ if and only
if either $0\le i\le f-1$ and $i\in J$, or $f\le i\le 2f-1$ and
$i-f\notin J$. Similarly, we write $\BC_{\PS}(J)$ for those
$i\in\{0,\dots,2f-1\}$ with $i\pmod{f}\in J$.

With these definitions, the following lemma is a tedious but elementary consequence of
the above discussion in the cuspidal case, and is
almost trivial in the  principal series case.
\begin{lem}
  \label{lem: base change of JH factors}Let $\sigma(\tau)$ be a tame type. If
  $\tau$ is a principal series {\em (}respectively cuspidal{\em )}
  type, then $\BC_{\PS}(J) \in \cP_{\BC(\sigma(\tau))}$ and 
  $\BC(\sigmabar(\tau)_J)=\sigmabar(\BC(\tau))_{\BC_{\PS}(J)}$
  {\em (}respectively $\BC_{\cusp}(J) \in \cP_{\BC(\tau)}$ and
  $\BC(\sigmabar(\tau)_J)=\sigmabar(\BC(\tau))_{\BC_{\cusp}(J)}${\em )}.
\end{lem}

\subsection{Tame types and Serre weights}The following results will be useful in our later arguments.
\begin{prop}
  \label{prop: types for elimination}Suppose that $\rhobar$ is
  generic, and that $\sigmabar\in\cD(\rhobar)$. Then there is a {\em (}non-scalar{\em })
tame inertial
  type $\tau$ such that $\JH(\sigmabar(\tau))\cap\cD(\rhobar)=\sigmabar$.
\end{prop}
\begin{proof}
  In the case that $\sigmabar=\sigmabar_{\vec{t},\vec{s}}$ with  $1\le s_i\le
  p-3$ for all $i$, then $\sigmabar$ is regular in the sense of~\cite{geebdj},
  and the result is an immediate consequence of
  \cite[Prop.~3.10, Prop.~4.2]{geebdj}.  In the more general setting that $0 \le s_i \le p-2$
  for all $i$, we will see that the proof
  of \cite[Prop.~3.10]{geebdj} (the case where $\rhobar$ reducible) goes
  through with only minor changes, while the irreducible case will
  follow from the reducible case by a base change argument.
For ease of
  reference we will freely use the notation of~\cite{geebdj}, with the minor changes
  that we will omit the letters $\tau$ and $\sigma$ from the notation (so that for example we
  write $c_i$ for $c_{\tau_i}$) and we continue to write $f$ 
   where \cite{geebdj} uses $r$. Our weight $\sigmabar_{\vec{t},\vec{s}}$ is the weight
$\sigmabar_{\vec{a},\vec{b}}$ of~\cite{geebdj}, where $a_i=t_{f-i}$ and
$b_i=s_{f-i}+1$.

Suppose first that $\rhobar$ is reducible, and define $c_i$ as in Section 3.5
of~\cite{geebdj}. Note firstly that the assumption of~\cite{geebdj} that not all
$c_i$ are equal to $0$ and not all $c_i$ are equal to $p-1$ is satisfied, as this could only fail to hold
if $b_i=1$ for all $i$, or equivalently $s_i=0$ for all $i$, which would
contradict Lemma~\ref{lem: generic representations have no weight p-1}.  Then the proof of Proposition 3.10 of~\cite{geebdj} goes through unchanged,
until the sentence beginning ``By the assumption that
$\sigmabar_{\vec{a},\vec{b}}$ is regular''. We may rewrite the displayed
equation before this sentence as \numequation\label{eq:red-chars} \prod_{i\in
  S}\omega_i^{c_i+\delta_J(i+1)}=\prod_{i\in X}\omega_i^{c_i+\delta_{J'\cup
    (K'\cap {J^{'}}^c)}(i+1)}\prod_{i\in X^c}\omega_i^{\delta_{(K'\cap
    {J^{'}}^c)}(i+1)}, \end{equation} where $X:=(J'\cap K')\cup({J^{'}}^c\cap
  {K^{'}}^c)$. 

Let $\psi$ be the character given by either side of
\eqref{eq:red-chars}.  It is easy to
check that the ratio of the two characters occurring in $\rhobar$ is
equal to $\psi \prod_{i \not\in J} \omega_{i-1}^{-1}$.  Since the
genericity of $\rhobar$ implies that this ratio (taken in either order) does
not have the form $\prod_{i \in S} \omega_i^{d_i}$ with $d_i = 0,1$
for all $i$, it follows in particular that $\psi$ is neither
trivial nor equal to $\prod_{i \in S} \omega_i$.

Now, we have $1\le b_i\le p-1$ for all $i$, and so we see from the definition of
$c_i$ that $0\le c_i\le p-1$ and $1\le c_i+\delta_J(i+1)\le p-1$. We
claim moreover that no exponent on
the right-hand side of \eqref{eq:red-chars} is equal to $p$, i.e., that no ``carrying'' can
take place in \eqref{eq:red-chars}.
Indeed the only possibility
for ``carrying'' is when $i\in X$, $c_i=p-1$, and $i+1\in J'\cup (K'\cap
{J^{'}}^c)$. But after carrying out all carries on the right hand side, the
exponent of $\omega_i$ will be at most $1$, whereas it is $p-1$ on the left hand
side, a contradiction since $\psi$ is non-trivial.

In particular (again using the fact that $\psi$ is non-trivial) we can equate
exponents on each side of \eqref{eq:red-chars}.
If  $i\in X^c$ we must have $i+1\in K'\cap{J^{'}}^c\subseteq X^c$, so that if
$X^c$ is nonempty then $X^c=S$, and every exponent on either side is equal to
$1$.  This contradicts our observation that $\psi \neq \prod_{i\in S} \omega_i$.
 Thus $X=S$, so that $J'=K'$, and
equating exponents shows that $J=J'$, as required.

Suppose next that $\rhobar$ is irreducible (and let the characters
$\omega_i$ now be fundamental characters of level $2f$).  Let
$I'_{\vec{a},\vec{b},J}$ be the cuspidal type  constructed in
\cite[\S4.1]{geebdj}.  One checks exactly as in the proof of \cite[Prop.~4.1]{geebdj} that
$\sigmabar_{\vec{a},\vec{b}}$ is a Jordan--H\"older factor of
$\overline{I}'_{\vec{a},\vec{b},J}$. From the discussion of base
change in Sections~\ref{subsec:base change Serre weights} and~\ref{subsec:base change
  tame types} we see that if $\sigmabar \in \calD(\rhobar) \cap
\JH(\overline{I}'_{\vec{a},\vec{b},J})$ then $\BC(\sigmabar)$ lies in
the intersection of
$\calD(\rhobar|_{G_{L_v}})$ with the set of Jordan--H\"older factors
of the reduction mod $p$ of
$\BC(I'_{\vec{a},\vec{b},J})$, and so it suffices to
prove that this latter intersection has size $1$.

Set $\vec{a}' = (a_0,\ldots,a_{f-1},a_0,\ldots,a_{f-1})$ and define $\vec{b}'$
similarly , so that $\BC(\sigmabar_{\vec{a},\vec{b}}) = \sigmabar_{\vec{a}',\vec{b}'}$.
We claim that $\BC(I'_{\vec{a},\vec{b},J})$ is
 the principal series type
$I_{\vec{a}',\vec{b}',J}$ (in the notation of
\cite[\S3.5]{geebdj}, but with $\vec{a}',\vec{b}'$ in place of $\vec{a},\vec{b}$)
constructed in the reducible case for  the representation
$\rhobar|_{G_{L_v}}$. Once we have checked this claim, the irreducible
case will follow immediately from our proof of the reducible case.

To see the claim, one calculates directly from the definitions that the
character $\tilde{\psi}_{\vec{a},\vec{b},J} \tilde{\omega}_f
\prod_{i=1}^f \tilde{\omega}_i^{c_i}$ of \cite[\S4.1]{geebdj} is the lift of
$\chi := \prod_{i=0}^{2f-1} \omega_i^{a_i} \prod_{i \in J}
\omega_i^{b_i-p}$.  Since $J$ is antisymmetric we have $\chi^q = \prod_{i=0}^{2f-1} \omega_i^{a_i} \prod_{i \not\in J}
\omega_i^{b_i-p}$; in other words $\chi^q$ is equal to the character
$\chi_{\vec{a}',\vec{b}',J}$ of \cite[\S3.5]{geebdj}.  Another
straightforward calculation shows that $\chi_{\vec{a}',\vec{b}',J}
\prod_{i=0}^{2f-1} \omega_i^{c_i} = \chi$ (with $c_i$ as in
\cite[\S3.5]{geebdj} now), so
that $$I_{\vec{a}',\vec{b}',J} =
I(\tilde{\chi}_{\vec{a}',\vec{b}',J},\tilde{\chi}_{\vec{a}',\vec{b}',J}\prod_i
\tilde{\omega}_i^{c_i}) = I(\tilde{\chi}^q,\tilde{\chi}) \cong \BC({I}'_{\vec{a},\vec{b},J})$$
as claimed.
\end{proof}

\begin{prop}
  \label{prop: types for existence}Suppose that $\rhobar$ is generic.
  Then there is a {\em (}non-scalar{\em )}
  tame inertial type~$\tau$ such that
  $\cD(\rhobar)\subset\JH(\sigmabar(\tau))$.
\end{prop}
\begin{proof} Since $\cD(\rhobar)\subset\cD(\rhobar^{\mathrm{ss}})$ we may
  assume that $\rhobar$ is semisimple, and the result now follows
  immediately from Theorems 2.1 and 3.2 of \cite{MR2392355}
(since $\rhobar$ is generic,  Lemma \ref{lem: generic
    representations have no weight p-1} shows that there no exceptional weights in the
  terminology of \cite{MR2392355}).
\end{proof}

\section{Lattices in residually multiplicity-free representations}\label{sec:lattices}
In this section we record some general facts about lattices in
(absolutely) irreducible representations whose reductions have all Jordan--H\"older factors
occurring with multiplicity one.

Throughout this section we let $\Gamma$ be a
group, and we let $\sigma$ be a
finite-dimensional $E$-representation of $\Gamma$.

\subsection{Lattices and gauges}\label{subsec:general facts about lattices
  in tame types}
We say that $\sigma$ is \emph{residually multiplicity free} if the
Jordan--H\"older factors of $\sigmabar$ occur with multiplicity
one. The following lemma and its proof are well-known, but for lack of
a convenient reference we give the details here.

\begin{lem}
  \label{lem:uniqueness of lattices with irreducible socle}Suppose
  that $\sigma$ is irreducible and is residually
  multiplicity free. Let $\sigmabar_i$ be a Jordan--H\"older factor of
  $\sigmabar$. Then there is up
to homothety a unique $\Gamma$-stable
  $\cO$-lattice $\sigma^\circ_{i}$ in $\sigma$ such that the socle
  of $\sigmabar^{\circ}_{i}$ is precisely~$\sigmabar_i$.
Similarly, there is up to homothety a unique $\Gamma$-stable
  $\cO$-lattice $\sigma^\circ_{i}$ in $\sigma$ such that the cosocle
  of $\sigmabar^{\circ}_{i}$ is precisely~$\sigmabar_i$.
\end{lem}
\begin{proof}
We begin by proving the result for socles.
  For existence, let $L$ be any $\Gamma$-stable $\cO$-lattice in $\sigma$,
  and consider $\sigma/L$. Let $M$ be the maximal
  $\cO[\Gamma]$-submodule of $\sigma/L$ such that none of
  the Jordan--H\"older factors of $M$ is isomorphic to
  $\sigmabar_i$. If $M$ had infinite length, then the inverse image of
  $M$ in $\sigma$ would contain a non-zero $\Gamma$-stable $E$-subspace,
  so would be equal to $\sigma$, a contradiction. Thus $M$ has finite
  length, and its inverse image in $\sigma$ gives the required
  lattice.

  For the uniqueness, suppose that $L$, $L'$ are two $\Gamma$-stable
  lattices in $\sigma$ with socle $\sigmabar_i$. After scaling, we may,
  without loss of generality, suppose that $L'\subseteq L$ but that
  $L'\not\subseteq\varpi_E L$. Consider the induced map $L'/\varpi_E
  L'\to L/\varpi_E L$; it is non-zero by construction.  
If this is not an isomorphism, then (for length
  reasons) it is not surjective, and thus has non-trivial kernel,
  which necessarily contains the socle $\sigmabar_i$ of $L'/\varpi_E
  L'$. Similarly, its image, being non-zero, contains the socle
of $L/\unif_E L$.  Thus $\sigmabar_i$ occurs at least twice
in $L'/\varpi_E L'$, contradicting our assumption that $\sigma$ is residually
multiplicity free.

The statement for cosocles follows from the statement for
socles via a consideration of contragredients.  More precisely,
the contragredient representation $\sigma^{\vee}$ is also
residually multiplicity free, and the Jordan--H\"older factors
of $\sigmabar^{\vee}$ are the contragredients of the 
Jordan--H\"older factors of $\sigmabar$.  If $L$ is a $\Gamma$-invariant lattice
in $\sigma$, then its $\mathcal O$-dual $L^{\vee}$, equipped with
the contragredient $\Gamma$-representation, is naturally a $\Gamma$-invariant
lattice in $\sigma^{\vee}$.  Furthermore, the cosocle of $L/\unif_E L$
is contragredient to the socle of $L^{\vee}/\unif_E L^{\vee}$.
The case of the lemma for cosocles then follows from the case
for socles, applied to $\sigma^{\vee}$.
\end{proof}

In fact, we will have more cause to apply the preceding lemma in the
case of cosocles than in the case of socles. Note that the formation
of cosocles is right exact, since the cosocle functor is the direct
sum \[L \mapsto\oplus_{\sigmabar}
\Hom_\Gamma(L,\sigmabar)\otimes_{\F}\sigmabar,\] where $\sigmabar$ runs over all
irreducible representations of $\Gamma$ over $\F$.  Then we have, for
example, the following result.

\begin{lem}
  \label{lem: lattice determined by its cosocle lattices}
Suppose that $\sigma$ is irreducible and residually multiplicty free.
Let $\sigma^{\circ}$ be a $\Gamma$-invariant lattice in $\sigma$,
let $\{\sigmabar_i\}$ denote the set of constituents
of the cosocle of $\sigma^{\circ}$,
and for each $\sigmabar_i$, choose a $\Gamma$-invariant
lattice $\sigma_i^{\circ}$ whose cosocle is isomorphic to $\sigmabar_i$,
such that $\sigma_i^{\circ} \subseteq \sigma^{\circ}$,
but such that $\sigma_i^{\circ} \not\subseteq \unif_E \sigma^{\circ}$.
{\em (}Note that $\sigma_i^{\circ}$ is uniquely determined,
by Lemma~{\em \ref{lem:uniqueness of lattices with irreducible socle}.)}
Then $\sigma^{\circ} = \sum_i \sigma_i^{\circ}.$
\end{lem}
\begin{proof}
  Write
  $\sigma^{\circ,\prime}:=\sum_{i=1}^m\sigma_i^\circ\subseteq\sigma^\circ$. This
  inclusion induces an isomorphism on cosocles, from which we see that
  $\sigma^\circ/\sigma^{\circ,\prime}$ has trivial cosocle, 
and thus
  vanishes; so $\sigma^\circ=\sigma^{\circ,\prime}$, as required.
\end{proof}
 
Let us now label the Jordan--H\"older factors of $\sigmabar$ as $\sigmabar_1,
\ldots, \sigmabar_n$, and for each value of $i = 1,\ldots,n$,
fix a choice of lattice $\sigma^{\circ}_i$
(unique up to homothety, by Lemma \ref{lem:uniqueness of lattices with irreducible socle}) such that the cosocle
of $\sigmabar^{\circ}_i$ equals $\sigmabar_i$.

\begin{defn}\label{defn:gauge}
If $\sigma^{\circ}$ is any $\Gamma$-invariant $\mathcal O$-lattice
in $\sigma$, define for each $i$ the fractional ideal $I_i := 
\{ a \in E \, | a \sigma^{\circ}_i \subseteq \sigma^{\circ}\},$ 
and define 
the {\em gauge} 
of $\sigma^{\circ}$ to be 
the $n$-tuple $(I_i)_{i = 1,\ldots,n},$
thought of as an $n$-tuple of fractional ideals in $E$.
We denote the gauge of $\sigma^{\circ}$ by  $\gauge(\sigma^{\circ})$.
\end{defn}

\begin{prop}
\label{prop:the gauge determines the lattice}
Suppose that $\sigma$ is irreducible and residually
multiplicity free.
A $\Gamma$-invariant $\cO$-lattice $\sigma^\circ$ in
$\sigma$ is determined by its gauge $\gauge(\sigma^\circ)$; indeed, if we fix a generator
$\phi_i$ for the fractional ideal $I_i$,
then we have $\sigma^\circ=\sum_{i = 1}^n \phi_i(\sigma^{\circ}_i)$.
\end{prop}
\begin{proof}
Write $\sigma^{\circ, \prime}:= \sum_{i = 1}^n \phi_i(\sigma^{\circ}_i)
\subseteq \sigma^{\circ}$.    Note that evidently $\sigma^{\circ,\prime}$
is well-defined independently of the choice of the generators
$\phi_i$, and hence depends only on the gauge of $\sigma^{\circ}$.
The proposition will follow if we prove that $\sigma^{\circ,\prime}
= \sigma^{\circ};$  however, this follows immediately
from Lemma~\ref{lem: lattice determined by its cosocle lattices}.
(Indeed, that lemma shows that we may restrict the sum defining
$\sigma^{\circ,\prime}$ to those indices $i$ for which
$\sigmabar_i$ is a constituent of the cosocle of $\sigma^{\circ}$.)
\end{proof}

\subsection{Lattices in essentially self-dual representations}

In this subsection we assume that $\Gamma$ is finite,
and that $\Gamma$ admits a character 
$\chi:\Gamma \to \cO^{\times}$ for which there is a $\Gamma$-equivariant
isomorphism $\sigma \cong \sigma^{\vee} \otimes \chi.$
We let $\chibar:\Gamma\rightarrow k_E^{\times}$ denote the reduction mod $\unif_E$
of $\chi$.  

\begin{prop}\label{prop: embedding of dual killed by order of group}
Suppose that $\sigma$ is absolutely irreducible
and that $\sigma^{\circ}$ is a $\Gamma$-invariant lattice in $\sigma$
such that $\cosoc(\sigma^{\circ})$ is absolutely irreducible and appears in
$\JH(\sigmabar)$ with multiplicity one, and such that
there is an isomorphism $\cosoc(\sigma^{\circ}) \cong \cosoc(\sigma^{\circ})^{\vee}
\otimes \chibar.$
Then there is a $\Gamma$-equivariant embedding
$(\sigma^{\circ})^{\vee}\otimes \chi \hookrightarrow \sigma^{\circ}$ whose cokernel
has exponent dividing $|\Gamma|$.
\end{prop}
\begin{proof}
Let $P$ denote a projective envelope of $\sigma^{\circ}$ as an $\cO[\Gamma]$-module.
Then $P^{\vee} \otimes \chi$ is again a projective $\cO[\Gamma]$-module,
and there are $\Gamma$-equivariant isomorphisms 
\begin{multline*}
\cosoc(P^{\vee} \otimes \chi) = \cosoc(\overline{P}^{\vee})
\otimes \chibar \cong
\soc(\overline{P})^{\vee}\otimes \chibar \buildrel \mathrm{(1)}\over \cong
\cosoc(\overline{P})^{\vee}\otimes \chibar
\\
= \cosoc(P)^{\vee} \otimes \chibar = \cosoc(\sigma^{\circ})^{\vee} \otimes \chibar 
\buildrel \mathrm{(2)} \over \cong \cosoc(\sigma^{\circ}) = \cosoc(P),
\end{multline*}
the isomorphism~(1) following from the isomorphism $\soc(\overline{P}) \cong \cosoc(\overline{P})$
(a standard property of projective $k_E[\Gamma]$-modules),
and the isomorphism~(2) holding by assumption.
The uniqueness up to isomorphism of projective envelopes thus implies that there is a 
$\Gamma$-equivariant isomorphism
\numequation
\label{eqn:projective iso}
P^{\vee} \otimes \chi \cong P.
\end{equation}
The $\Gamma$-equivariant surjection $P \to \sigma^{\circ}$
(expressing the fact that $P$ is a projective
envelope of $\sigma^{\circ}$) induces a saturated $\Gamma$-equivariant embedding
$(\sigma^{\circ})^{\vee}\otimes \chi \hookrightarrow P^{\vee}
\otimes \chi$ ({\em saturated} meaning that the cokernel is $\unif_E$-torsion
free).  Composing this with the isomorphism~(\ref{eqn:projective iso}),
we obtain a saturated $\Gamma$-equivariant embedding
\numequation
\label{eqn:dual embedding}
(\sigma^{\circ})^{\vee} \otimes \chi \hookrightarrow P.
\end{equation}

Recall that
there is an idempotent $e_{\sigma} \in \dfrac{1}{|\Gamma|}\cO[\Gamma] \subseteq E[\Gamma]$
which, when applied to any $\Gamma$-representation $V$ over $E$, projects
onto the $\sigma$-isotypic part of $V$.  We claim that $e_{\sigma} (E\otimes_{\cO} P)$
consists of a single copy of $\sigma$.  Indeed,
since $\sigma$ is absolutely irreducible, the number of such copies of $\tau$ is 
equal to the dimension of 
$$\Hom_{E[\Gamma]}(E\otimes_{\cO}P, \sigma)
= E\otimes_{\cO}\Hom_{\cO[\Gamma]}(P,\sigma^{\circ}).$$
It thus suffices to show that $\Hom_{\cO[\Gamma]}(P,\sigma^{\circ})$,
which is a finite rank free $\cO$-module, is of rank one. 
To this end, we compute (using the projectivity of $P$) that
\begin{samepage}
\begin{multline*}
\Hom_{\cO[\Gamma]}(P,\sigma^{\circ})/\unif_E \Hom_{\cO[\Gamma]}(P,\sigma^{\circ}) = 
\Hom_{k_E[\Gamma]}(\overline{P},\sigmabar^{\circ}) 
\\ = \Hom_{k_E[\Gamma]}(\cosoc(P),
\sigmabar^{\ssg}).
\end{multline*}
\end{samepage}
The last of these spaces is one-dimensional, as by assumption
$\cosoc(P) = \cosoc(\sigma^{\circ})$
is absolutely irreducible and appears in $\sigmabar^{\ssg}$ with multiplicity one,
and so our claim is proved.

Since~(\ref{eqn:dual embedding}) is a saturated embedding, it follows
that its image is equal to
$e_{\sigma}(E\otimes_{\cO} P) \cap P,$
and the composite of~(\ref{eqn:dual embedding}) and the projection
$P \to \sigma^{\circ}$ is a $\Gamma$-equivariant embedding
$(\sigma^{\circ})^{\vee} \otimes \chi \hookrightarrow \sigma^{\circ}.$
Since $|\Gamma| e_{\sigma} \in \cO[\Gamma],$ we easily see
that the image of this embedding contains $|\Gamma|\sigma^{\circ}$; indeed, 
$|\Gamma| e_{\sigma} P \subseteq 
e_{\sigma}(E\otimes_{\cO} P) \cap P,$
and the image of $|\Gamma| e_{\sigma} P$
under the projection to $\sigma^{\circ}$ is precisely 
$|\Gamma| \sigma^{\circ}$.  This completes the proof
of the proposition.
\end{proof}

It seems worthwhile to record the following result, which relates the conclusion
of the preceding proposition to the length of the socle and cosocle filtrations of
$\sigmabar^{\circ}$, although in the particular context that we consider in the sequel,
we will need a more precise statement
(Theorem~\ref{thm: the output of gauges.tex} below).

\begin{prop} Continue to assume that $\sigma \cong \sigma^{\vee} \otimes \chi.$
Suppose that $\sigma$ is absolutely irreducible and residually multiplicity free,
and that for each Jordan--H\"older factor $\sigmabar_i$
of $\sigmabar$, there is a $\Gamma$-equivariant isomorphism $\sigmabar_i^{\vee}
\otimes \chibar \cong \sigmabar_i$.
If $\sigma^{\circ}$ is any $\Gamma$-invariant lattice
in $\sigma$ for which there is a $\Gamma$-equivariant embedding 
\numequation
\label{eqn:assumed embedding}
(\sigma^{\circ})^{\vee}\otimes \chi
\hookrightarrow \sigma^\circ
\end{equation}
whose cokernel is annihilated by $\unif_E^m$,
then the socle and cosocle filtrations of $\sigmabar^{\circ}$ are both
of length at most $m+1$.
\end{prop}
\begin{proof}
If $m = 0$ then~(\ref{eqn:assumed embedding}) is an isomorphism, 
and consequently there is a $\Gamma$-equivariant
isomorphism $(\sigmabar^{\circ})^{\vee} \otimes \chi \cong \sigmabar^{\circ}$.
Passing to cosocles, we obtain an isomorphism
$\soc(\sigmabar^{\circ})^{\vee} \otimes \chibar \cong \cosoc(\sigmabar^{\circ}),$
which, by our hypotheses, may be rewritten as an isomorphism
$\soc(\sigmabar^{\circ}) \cong \cosoc(\sigmabar^{\circ}).$
Since each Jordan--H\"older factor of $\sigmabar$ appears with multiplicity one,
this implies a literal equality of the socle and cosocle of $\sigmabar^{\circ}$,
i.e.\ it implies that $\sigmabar^{\circ}$ is semisimple.

We treat the cases $m \geq 1$ by induction on $m$.  
Obviously it is no loss of generality to assume that $m$ is minimal, i.e.\
that the image of the embedding~(\ref{eqn:assumed embedding}) does not
lie in $\unif_E \sigma^{\circ}$.  Since $\sigma$ is absolutely irreducible
by assumption, so that the only endomorphisms of $\sigma$ are scalars,
this condition then determines the embedding~(\ref{eqn:assumed embedding})
up to multiplication by an element of $\cO^{\times}$.
Dualizing~(\ref{eqn:assumed embedding}), and twisting by $\chi$,
we obtain another embedding $(\sigma^{\circ})^{\vee} \otimes \chi \hookrightarrow
\sigma^{\circ}$ with the same property (that its image is not 
contained in $\unif_E \sigma^{\circ}$), which thus coincides with~(\ref{eqn:assumed
embedding}) up to multiplication by an element of $\cO^{\times}$.
Consequently we find that the cokernel of~(\ref{eqn:assumed embedding})
is isomorphic to its own Pontrjagin dual, up to twisting by $\chi$.

In particular, 
if $m = 1$, then the cokernel $\tau$ of~(\ref{eqn:assumed embedding})
is a representation of $\Gamma$ over $k_E$ which satisfies 
$\tau^{\vee} \otimes \chibar \cong \tau.$  Since $\tau$ is a quotient
of $\sigmabar^{\circ}$, we furthermore have that every Jordan--H\"older
factor $\sigmabar_i$ of $\tau$ appears with multiplicity one, and 
satifies $\sigmabar_i^{\vee} \otimes \chi \cong \sigmabar_i.$
Applying the same argument to $\tau$ as we applied to $\sigmabar^{\circ}$
in the case $m = 0$, we conclude that $\tau$ is semisimple.
A similar argument shows that the cokernel of the induced embedding
$\unif_E \sigma^{\circ} \hookrightarrow (\sigma^{\circ})^{\vee} \otimes \chi$
is semisimple.  Thus $\sigmabar^{\circ}$ admits a two-step filtration
whose associated graded pieces are semisimple, proving the proposition
in this case.

Suppose finally that $m > 1$.  Regard $(\sigma^\circ)^{\vee}\otimes \chi$
as a sublattice of $\sigma^{\circ}$ via~(\ref{eqn:assumed embedding}),
and let $\sigma^{\prime\circ} := \unif_E \sigma^{\circ} + (\sigma^{\circ})^{\vee}
\otimes \chi$. Then $(\sigma^{\prime\circ})^{\vee} \otimes \chi =
\bigl(\unif_E^{-1} (\sigma^{\circ})^{\vee}\otimes \chi\bigr) \cap 
\sigma^{\circ},$ and so multiplication by $\unif_E$ induces a $\Gamma$-equivariant
embedding 
$$(\sigma^{\prime \circ})^{\vee} \otimes \chi \hookrightarrow \sigma^{\prime\circ},$$
whose cokernel is easily checked to be annihilated by $\unif_E^{m-1}$.
By induction we conclude that the socle and cosocle filtrations of $\sigmabar^{\prime\circ}$
are of length at most $m$, or equivalently (multiplying by $\unif_E^{-1}$),
that the socle and cosocle filtrations of
$\unif_E^{-1}\sigmabar^{\prime\circ}/\sigmabar^{\prime\circ}$
are of length at most $m$.  Since $\sigma^{\prime\circ} \subseteq
\sigma^{\circ} \subseteq \unif_E^{-1} \sigma^{\prime\circ}$,
we conclude that the socle and cosocle filtrations of $\sigma^{\circ}/\sigma^{\prime\circ}$
are also of length at most $m$. 

Now $\sigma^{\prime\circ}/\unif_E \sigma^{\circ}$ is a $\Gamma$-subrepresentation
of $\sigmabar^{\circ}$ whose cosocle is a quotient of 
$$\cosoc\bigl((\sigma^{\circ})^{\vee}\otimes\chi\bigr) \cong \soc(\sigmabar^{\circ})^{\vee}
\otimes \chibar \cong \soc(\sigmabar^{\circ})$$
(the second isomorphism holding by our hypotheses that all the Jordan--H\"older factors
of $\sigmabar$ are self-dual up to twisting by $\chi$).  Since all the Jordan--H\"older
factors of $\sigmabar^{\circ}$ appear with multiplicity one, in fact 
the cosocle of $\sigma^{\prime\circ}/\unif_E \sigma^{\circ}$ must be a subrepresentation
of the socle of $\sigmabar^{\circ}$, and consequently $\sigma^{\prime\circ}/\unif_E
\sigma^{\circ}$ must be semisimple.
Combining this with the conclusion of the preceding paragraph, we find that
the socle and cosocle filtrations of $\sigmabar^{\circ}$ are of length at most
$m+1$, as required.
\end{proof}

\section{Lattices in tame types}\label{sec:lattices in tame types}

In this section we elaborate on the ideas of the preceding section
in the context of tame types (in the sense of Section~\ref{sec:base change}).
In the specific case of tame principal series types,
related results may be
found in Section~2 of~\cite{breuillatticconj}, and several of our
arguments are generalisations of those found in~\cite{breuillatticconj}.

\subsection{Socle and cosocle filtrations in tame types}If $\tau$ is a tame
             inertial type, assumed to be non-scalar, we write
             $\sigmabar_J(\tau)$ for the Jordan--H\"older factor
             $\sigmabar(\tau)_J$ defined in~\ref{ss:principal}
             and~\ref{ss:cuspidal}, and $\cP_{\tau}$ for
             $\cP_{\sigma(\tau)}$.

For each $J \in \cP_\tau$, 
we fix $\GL_2(\cO_{F_v})$-invariant $\cO$-lattices $\sigma^{\circ,J}(\tau)$ and $\sigma^\circ_J(\tau)$ in $\sigma(\tau)$,
so that the socle of $\sigmabar^{\circ,J}(\tau)$ (resp.\ the cosocle of $\sigmabar^\circ_J(\tau)$)
is equal to $\sigmabar_J(\tau)$.  
(In the case of $\sigma^\circ_J(\tau)$, this can be expressed more simply
by saying that the cosocle of $\sigma^\circ_J(\tau)$ itself 
is equal to $\sigmabar_J(\tau)$.) 
Lemma~\ref{lem:uniqueness of lattices with irreducible socle}
shows that these lattices are uniquely determined up to scaling.
The following theorem describes the cosocle and socle filtrations of
the mod~$\unif_E$ reductions of these lattices. Recall that we number
the layers of the cosocle and socle filtrations starting at $0$.

\begin{thm}
  \label{thm: socle filtration for tame type with irred socle}
The $i$-th layer of the cosocle filtration of
$\sigmabar^\circ_J(\tau)$
  {\em (}respectively, the $i$-th layer of the socle filtration of
  $\sigmabar^{\circ,J}(\tau)${\em )} consists precisely of the Jordan--H\"older factors
  $\sigmabar_{J'}(\tau)$ with $|J\triangle J'|=i$.
  Furthermore, if $J_1, J_2\in\cP_\tau$ with $|J_1\triangle J_2|=1$,
  so that $\sigmabar_{J_1}(\tau)$ and $\sigmabar_{J_2}(\tau)$ occur in
  adjacent layers of the cosocle filtration of $\sigmabar^\circ_J(\tau)$
  {\em (}respectively, the socle filtration of
  $\sigmabar^{\circ,J}(\tau)${\em )}, then the extension between
  $\sigmabar_{J_1}(\tau)$ and $\sigmabar_{J_2}(\tau)$ in this
  filtration is nonsplit.
\end{thm}

Suppose that $J_1, J_2 \in \cP_{\tau}$ with $|J_1 \triangle J_2| =
1$.  It follows straightforwardly from Theorem~\ref{thm: socle filtration for tame type
  with irred socle}, together with Lemma~\ref{lem:weights in tame types have multiplicity one}, that every subquotient of
$\sigmabar^{\circ}_J(\tau)$ (respectively $\sigmabar^{\circ,J}(\tau)$)
with Jordan--H\"older factors precisely $\sigmabar_{J_1}(\tau)$ and
$\sigmabar_{J_2}(\tau)$ must be non-semisimple.  

The remainder of this subsection is devoted to 
proving Theorem~\ref{thm: socle filtration for tame type with irred socle}.
The proof requires a generalization of several
of the results from Sections~3 and~4 of \cite{Breuil_Paškūnas_2012},
and we begin by establishing these generalizations.
To this end, we fix an $f$-tuple of integers $(r_i)_{0 \leq i \leq f-1}$,
such that $0 \leq r_i \leq p-1$ for all $i$. Theorem~\ref{thm: socle
  filtration for tame type with irred socle} is straightforward when
$f=1$, and we will assume that $f>1$ in the below.

In \cite[\S 3]{Breuil_Paškūnas_2012},
Breuil and Pa\v{s}k\={u}nas define a representation $V_r$ of $\GL_2(k_v)$
over $\F$
for each integer $r \geq 0$, as the reduction modulo $\unif_E$
of a certain $\cO$-lattice in $\Sym^r E^2$;
here we recall that $k_v$ is regarded as embedded in $\F$ via $\overline{\kappa}_0.$
(Actually Breuil and Pa\v{s}k\={u}nas work with coefficients in $\Zbar_p$,
$\Qbar_p$, and $\Fbar_p$, rather than $\cO$, $E$, and $\F$, but it is
evident that all their constructions and results descend to these latter
choices of coefficients.)  In particular if $0 \le r \le p-1$ then
$V_r$ is isomorphic to $(\Sym^r k_v^2) \otimes_{k_v,\overline{\kappa}_0} \F$.
By convention we let $V_{-1}=0$.

Suppose that $V$ is an $\F$-representation of $\GL_2(k_v)$.  In this
section we follow the notation of \cite{Breuil_Paškūnas_2012} and
denote by 
$V^{\Fr^i}$ the representation of $\GL_2(k_v)$ obtained by
letting $\GL_2(k_v)$ act on $V$ via twists by powers of the arithmetic
Frobenius; so for instance the Serre weight $\sigmabar_{\vec{t},\vec{s}}$ of Section~\ref{sec:Galois
  reps and weights} can be written
$\bigotimes_{i=0}^{f-1} (V_{s_i} \otimes \det^{t_i})^{\Fr^i}$.  

The tensor product $\bigotimes_{i = 0}^{f-1} V_{2p - 2 - r_i}^{\Fr^i}$ is 
then naturally a representation of $\GL_2(k_v)$, whose socle is equal to
$\bigotimes_{i=0}^{f-1} (V_{r_i}\otimes \det^{p-1-r_i})^{\Fr^i}$,
unless all of the $r_i=0$, in which case the socle also contains
$\bigotimes_{i=0}^{f-1} V_{p-1}$. 
(See the discussion of \cite[~\S 3]{Breuil_Paškūnas_2012}.)
In the case when all $r_i \leq p-2$, the submodule structure of
$\bigotimes_{i = 0}^{f-1} V_{2p - 2 - r_i}^{\Fr^i}$ is analyzed quite
precisely in \cite[\S\S 3,4]{Breuil_Paškūnas_2012}; our goal is
to make a similar analysis in the case when some of the $r_i$ are
equal to $p-1$. In the case that \emph{all} of the $r_i$ are equal to
$p-1$, then the representation is irreducible, and we exclude this
case from our analysis below.

Following Section 3 of~\cite{Breuil_Paškūnas_2012}, we can describe
the Jordan--H\"older factors of $\bigotimes_{i = 0}^{f-1} V_{2p - 2 -
  r_i}^{\Fr^i}$ in the following way. We let $x_0,\dots,x_{f-1}$ be
variables, and we define the set $\cI(x_0,\dots,x_{f-1})$ of all
$f$-tuples $\lambda:=(\lambda_0(x_0),\dots,\lambda_{f-1}(x_{f-1}))$
which satisfy the following conditions for each $i$ (numbered cyclically):
\begin{itemize}
\item $\lambda_i(x_i)\in\{x_i,x_i-1,x_i+1,p-2-x_i,p-3-x_i,p-1-x_i\}$,
\item if $\lambda_i(x_i)\in\{x_i,x_i-1,x_i+1\}$, then
  $\lambda_{i+1}(x_{i+1})\in\{x_{i+1},p-2-x_{i+1}\}$, and
\item if $\lambda_i(x_i)\in\{p-2-x_i,p-3-x_i,p-1-x_i\}$, then
  $\lambda_{i+1}(x_{i+1})\in\{x_{i+1}-1,x_{i+1}+1,p-3-x_{i+1},p-1-x_{i+1}\}$.
\end{itemize}
For $\lambda\in\cI(x_0,\dots,x_{f-1})$, we define
$e(\lambda):=\frac{1}{2}\bigl(\sum_{i=0}^{f-1}p^i(x_i-\lambda_i(x_i))\bigr)$
if $\lambda_{f-1}(x_{f-1})\in\{x_{f-1},x_{f-1}-1,x_{f-1}+1\}$, and
$e(\lambda):=\frac{1}{2}\bigl(p^f-1+\sum_{i=0}^{f-1}p^i(x_i-\lambda_i(x_i))\bigr)$
otherwise. For each $\lambda\in\cI(x_0,\dots,x_{f-1})$, we
define \[\sigmabar_\lambda:=\bigotimes_{i=0}^{f-1}
(V_{\lambda_i(r_i)} \otimes \det{}^{p-1-r_i})^{\Fr^i}\otimes\det{}^{e(\lambda)(r_0,\dots,r_{f-1})}\]
provided that $0 \le \lambda_i(r_i) \le p-1$ for all $i$, and we leave
$\sigmabar_\lambda$ undefined otherwise.
\begin{prop}
  \label{prop: description of JH factors of V_{2p-2-r}}Each
  Jordan--H\"older factor of $\bigotimes_{i = 0}^{f-1} V_{2p - 2 -
  r_i}^{\Fr^i}$  is isomorphic to $\sigmabar_\lambda$ for a
  uniquely determined $\lambda\in\cI(x_0,\dots,x_{f-1})$.
\end{prop}
\begin{proof}
  This follows immediately from Lemma~3.2
  of~\cite{Breuil_Paškūnas_2012}, which shows that every Jordan--H\"older factor of an injective envelope of $(V_{r_i}\otimes
  \det^{p-1-r_i})^{\Fr^i}$   is isomorphic to  $\sigmabar_\lambda$ for a uniquely
determined $\lambda\in\cI(x_0,\dots,x_{f-1})$, and
  Lemma~3.4 of~\cite{Breuil_Paškūnas_2012} (and the discussion
  immediately preceding it), which shows that $\bigotimes_{i = 0}^{f-1} V_{2p - 2 -
  r_i}^{\Fr^i}$ is isomorphic to a subrepresentation of such an
injective envelope.
\end{proof}

For any Jordan--H\"older factor $\sigmabar$ of $\bigotimes_{i = 0}^{f-1} V_{2p - 2 -
  r_i}^{\Fr^i}$, we define an integer $0\le l(\sigmabar)\le f$ as
follows: write $\sigmabar\cong\sigmabar_\lambda$, and let
$l(\sigmabar)$ be the number of $0\le i\le f-1$ with
$\lambda_i(x_i)\in\{p-2-x_i,p-3-x_i,p-1-x_i\}$.

Let $\cE := \{i \, | \, r_i = p-1\},$ and write
$\mathcal S' := \cS \setminus \cE.$  (This is the set denoted $S_{\mathbf r}$
in \cite[\S 3]{Breuil_Paškūnas_2012}.)
We begin our analysis of the submodule structure of $\bigotimes_{i = 0}^{f-1} V_{2p - 2 - r_i}^{\Fr^i}$ by defining subrepresentations $F'_J$ of
$\bigotimes_{i = 0}^{f-1} V_{2p - 2 - r_i}^{\Fr^i}$
indexed by subsets
$J \subseteq  \cS'$.  Note that by
\cite[Lem.~3.5]{Breuil_Paškūnas_2012}, 
there is an embedding
\numequation
\label{eqn:V embedding}
V_{r_i}\otimes \det{}^{p-1-r_i} \hookrightarrow
V_{2p - 2 - r_i},
\end{equation}
uniquely determined up to a non-zero scalar (of course, when
$r_i=p-1$, the embedding is an isomorphism).
For each $i \in \cS,$ define $F'_J(i)$ to be the image of this map if $ i \not\in J,$
and to be $V_{2p - 2 - r_i}$ if $i \in J$, and then define
$$F'_J := \bigotimes_{i = 0}^{f-1} F'_J(i)^{\Fr^i} \subseteq
\bigotimes_{i = 0}^{f-1} V_{2p - 2 - r_i}^{\Fr^i}.$$
These subrepresentations are closely related to the filtration of
$\bigotimes_{i = 0}^{f-1} V_{2p - 2 - r_i}^{\Fr^i}$
considered in \cite{Breuil_Paškūnas_2012} in the discussion
preceding Lemma~3.8.  More precisely, one has that
$$\Fil^j
\bigotimes_{i = 0}^{f-1} V_{2p - 2 - r_i}^{\Fr^i}
= \sum_{|J| = |\cS'| - j} F_J'.$$
The association $J \mapsto F_J'$ is order-preserving,
and, if we write
$$F'_{< J} := \sum_{i \in J} F'_{J\setminus \{i\}} = \sum_{J' \subsetneq J} F'_{J'},$$
then $F_J'/F_{<J}' = W_J$,
where $W_J$ is a certain subquotient of 
$\bigotimes_{i = 0}^{f-1} V_{2p - 2 - r_i}^{\Fr^i}$
defined in \cite[\S 3]{Breuil_Paškūnas_2012}, whose
precise description is as follows. It follows from \cite[Lem.~3.5]{Breuil_Paškūnas_2012} that
the cokernel of~(\ref{eqn:V embedding}) is isomorphic to
$V_{p - 2 - r_i}\otimes V_1^{\Fr}$.
Thus $W_J := \otimes_{i = 0}^{f-1} W_J(i)^{\Fr^i},$
where $W_J(i)$ equals $V_{r_i}\otimes \det^{p-1-r_i}$
if $i \not\in J$, and equals $V_{p - 2 - r_i}\otimes V_1^{\Fr}$
if $i \in J.$  

If $\cE = \emptyset,$ then each $W_J$ is semi-simple,
and the filtration $\Fil^i$ coincides with the cosocle filtration
of $\bigotimes_{i = 0}^{f-1} V_{2p - 2 - r_i}^{\Fr^i}$.
However, if $\cE \neq \emptyset$, these statements are
no longer true; it is this case that we wish to analyze further.

Before continuing our analysis, we need to consider extensions of Serre weights.

\begin{prop}
  \label{prop:extensions of Serre weights}
  Let $\sigmabar_{\vec{s},\vec{t}}$ and
    $\sigmabar_{\vec{s'},\vec{t'}}$ be two Serre weights, and if $f=1$
    then suppose that $s_0,s'_0\ne p-1$. Then there is a non-zero
    extension between $\sigmabar_{\vec{s},\vec{t}}$ and
    $\sigmabar_{\vec{s'},\vec{t'}}$ if and only if the following
    conditions hold:
    \begin{enumerate}
    \item If $f=1$, then $s'_0=p-2-s_0\pm 1$, and $t_0'-t_0\equiv
      s_0+1-p\frac{(1\pm 1)}{2}\pmod{p-1}$.
    \item If $f>1$, then there is some $k$ such that $s_j=s_j'$ if $j\ne
      k, k+1$, $s'_k=p-2-s_k$, $s'_{k+1}=s_{k+1}\pm 1$, and
      $\sum_{j=0}^{p-1}p^j(t'_j-t_j)\equiv
      p^k(s_k+1)-p^{k+1}\frac{(1\pm 1)}{2}\pmod{p^f-1}$.
    \end{enumerate}
In each of these cases both
$\Ext^1_{\F[\GL_2(k_v)]}(\sigmabar_{\vec{s},\vec{t}},\sigmabar_{\vec{s'},\vec{t'}})$
and
$\Ext^1_{\F[\GL_2(k_v)]}(\sigmabar_{\vec{s'},\vec{t'}},\sigmabar_{\vec{s},\vec{t}})$
are one-dimensional.
\end{prop}
\begin{proof}
This is Corollary~5.6(i) of
\cite{Breuil_Paškūnas_2012}. 
\end{proof}

In what follows it
will help to recall the explicit construction of the extensions
that are described by the preceding proposition.
First note that if we set $r_i = s_i,$ then the Serre weight
$\sigmabar_{\vec{s},\vec{t}}$ is a twist of $\bigotimes_{i = 0}^{f-1} (V_{r_i}\otimes
\det{}^{p-1-r_i})^{\Fr^i}$; we thus may, and do, restrict ourselves to considering
extensions of the latter such representations.

As we have already recalled, Lemma~3.5 of \cite{Breuil_Paškūnas_2012} shows that $V_{2p - 2 - r_i}^{\Fr^i}$ sits
in a short exact sequence
\numequation
\label{eqn:V short exact sequence}
0 \to (V_{r_i} \otimes \det{}^{p-1 -r_i})^{\Fr^i}  \to
V_{2 p -2  - r_i}^{\Fr^i} \to
V_{p - r_i - 2}^{\Fr^i}\otimes V_1^{\Fr^{i+1}} \to 0.
\end{equation}
If we fix some $i \in \{0, \ldots, f-1\},$ and tensor this exact sequence
with the representations $(V_{r_{i'}}\otimes \det{}^{p-1 - r_{i'}})^{\Fr^{i'}}$
for $i' \neq i,$ we obtain a short exact
sequence\footnote{Here and below we have to consider
many tensor products indexed by the elements of $\cS$.  It will be
convenient, and save introducing additional notation, to not always
write these tensor products with the indices taken in the
standard order.  This should cause no confusion, as long as it
is understood that for any permutation $\pi$ of $\cS$, 
and any collection of representations $A_i$ indexed by the elements
$i \in \cS$, we always identify $\bigotimes_i A_i$ and $\bigotimes_i A_{\pi(i)}$
via the isomorphism $a_0 \otimes \cdots \otimes a_{f-1} \mapsto
a_{\pi(0)}\otimes \cdots \otimes a_{\pi(f-1)}$; i.e.\ we rearrange the
order of the factors in the tensor product according to the permutation
$\pi^{-1}$.}
\begin{multline*}
0 \to
\bigotimes_{i' = 0}^{f-1} (V_{r_{i'}} \otimes \det{}^{p-1 - r_{i'}})^{\Fr^{i'}}
\to
V_{2 p -2  - r_i}^{\Fr^i} 
\otimes
\bigotimes_{i'\neq i } (V_{r_{i'}} \otimes \det{}^{p-1 - r_{i'}})^{\Fr^{i'}}
\\
\to V_{p  - r_i - 2}^{\Fr^i} \otimes
(V_{r_{i+1}} \otimes V_1 \otimes \det{}^{p-1-r_{i+1}})^{\Fr^{i+1}}
\otimes
\bigotimes_{i'\neq i,i+1 } (V_{r_{i'}} \otimes \det{}^{p-1 - r_{i'}})^{\Fr^{i'}}
\to 0.
\end{multline*}
The middle term of this short exact sequence
is contained in $\bigotimes_{i' = 0}^{f-1} V_{2p - 2 - r_{i'}}^{\Fr^{i'}},$
and since the socle of this latter representation is equal
to 
$\bigotimes_{i' = 0}^{f-1} (V_{r_{i'}} \otimes \det{}^{p-1 -
  r_{i'}})^{\Fr^{i'}}$ (unless all of the $r_{i'}$ are equal to $0$,
but the extra term $\bigotimes_{i' = 0}^{f-1} V_{p-1}$ which is then
present does not affect this argument),
we conclude that the socle of the middle term of this short exact sequence
is equal to 
the left-hand term.

If $r_{i+1} \leq p-2$ then \cite[Lem.~3.8(i)]{Breuil_Paškūnas_2012} 
shows that $V_{r_{i+1}} \otimes V_1 \cong V_{r_{i+1} +1} \oplus (V_{r_{i+1}-1}\otimes \det),$
and thus this short exact sequence gives rise to the non-split extensions
of
$V_{p  - r_i - 2}^{\Fr^i} \otimes (V_{r_{i+1}\pm 1}\otimes
\det{}^{p - \frac{1 \pm 1}{2} - r_{i+1}} )^{\Fr^{i+1}}
\otimes
\bigotimes_{i'\neq i,i+1 } (V_{r_{i'}} \otimes \det{}^{p-1 - r_{i'}})^{\Fr^{i'}}
$
by
$\bigotimes_{i' = 0}^{f-1} (V_{r_{i'}} \otimes \det{}^{p-1 - r_{i'}})^{\Fr^{i'}}$
whose existence is asserted by
Proposition~\ref{prop:extensions of Serre weights}.
If $r_{i+1} = p-1,$ then \cite[Lems.~3.4, 3.5, and 3.8(ii)]{Breuil_Paškūnas_2012} 
show that $V_{r_{i+1}}\otimes V_1 = V_{p-1}\otimes V_1$
contains $V_{p-2}\otimes\det = V_{r_{i+1}-1}\otimes \det$ as a subrepresentation,
and so the preceding short exact sequence gives rise to the non-split extension of
$V_{p  - r_i - 2}^{\Fr^i} \otimes (V_{r_{i+1}-1}\otimes
\det{}^{p - r_{i+1}} )^{\Fr^{i+1}}
\otimes
\bigotimes_{i'\neq i,i+1 } (V_{r_{i'}} \otimes \det{}^{p-1 - r_{i'}})^{\Fr^{i'}}
$
by
$\bigotimes_{i' = 0}^{f-1} (V_{r_{i'}} \otimes \det{}^{p-1 - r_{i'}})^{\Fr^{i'}}$
whose existence is asserted by
Proposition~\ref{prop:extensions of Serre weights}.

From these explicit constructions, we can infer the following lemma.

\begin{lemma}
\label{lem:tensoring extensions}
Suppose that $f > 1$, and that
  $\sigmabar_{\vec{s},\vec{t}}$ and
    $\sigmabar_{\vec{s'},\vec{t'}}$ are two Serre weights
related as in the statement of Proposition~{\em \ref{prop:extensions of Serre weights}~(2)};
namely there is some $k$ such that $s_j=s_j'$ if $j\ne
      k, k+1$, $s'_k=p-2-s_k$, $s'_{k+1}=s_{k+1}\pm 1$, and
      $\sum_{j=0}^{p-1}p^j(t'_j-t_j)\equiv
      p^k(s_k+1)-p^{k+1}\frac{(1\pm 1)}{2}\pmod{p^f-1}$.
Suppose moreover that $s_i = s_i' = 0$ for some particular $i \neq k,k+1$.
If 
$0 \to \sigmabar_{\vec{s},\vec{t}} \to V \to \sigma_{\vec{s'},\vec{t'}} \to 0$
is a non-split extension of $\GL_2(k_v)$-representations over $\F$,
and if $0 \leq r \leq p-1$,
then the tensor product
$0 \to \sigmabar_{\vec{s},\vec{t}}\otimes V_r^{\Fr^i}
\to V\otimes V_r^{\Fr^i} \to \sigma_{\vec{s'},\vec{t'}} \otimes V_r^{\Fr^i} \to 0$
is again a non-split extension of such representations.
\end{lemma}
\begin{proof}
Define $\vec{r}$ by $r_j = s_j$ if $j \neq i$,
and $r_i = r$, and similarly define $\vec{r'}$ by
$r'_j = s'_j $ if $j \neq i$, and $r'_i = r$.
Then
$\sigmabar_{\vec{r},\vec{t}}$
and
$\sigmabar_{\vec{r'},\vec{t'}}$
again satisfy condition~(2) of Proposition~\ref{prop:extensions of Serre weights},
and so, by that proposition, there is a non-trivial extension of 
$\sigmabar_{\vec{r'},\vec{t'}}$
by
$\sigmabar_{\vec{r},\vec{t}}$.  
The preceding discussion shows, though,
that this extension is in fact obtained exactly by tensoring
the extension of 
$\sigmabar_{\vec{s'},\vec{t'}}$
by
$\sigmabar_{\vec{s},\vec{t}}$  
by $V_r^{\Fr^i}$.  This proves the lemma.
\end{proof}

We need a further technical lemma related to extensions.
Before stating it, we introduce a definition.
If $0 \leq r \leq p-1$ and $0 \leq n < f-1$, then we define
$$V_r(n) := 
 V_r \otimes V_1^{\Fr} \otimes \bigotimes_{i = 1}^n
V_{p-1}^{\Fr^i}.$$

\begin{lemma}
\label{lem:tricky socle filtration}
The lattice of subrepresentations
of $V_r(n)$ is totally ordered, and thus $V_r(n)$ admits a unique
Jordan--H\"older filtration, which coincides with both its socle
and cosocle filtration.  
The length of this filtration is $2n+1$,
and the Jordan--H\"older factors of this filtration, taken in
order 
are
\begin{eqnarray*}
& V_{r}\otimes (V_{p-2} \otimes \det) ^{\Fr} \otimes V_{p-1}^{\Fr^{2}}
\otimes \cdots \otimes V_{p-1}^{\Fr^{n}}, \\
& V_{r}^{}\otimes V_0^{\Fr} \otimes (V_{p-2} \otimes \det)^{\Fr^{2}}
\otimes V_{p-1}^{\Fr^{3}} \otimes \cdots \otimes V_{p-1}^{\Fr^{n}},\\
& \cdots,\\
& V_{r}\otimes V_0^{\Fr} \otimes \cdots
\otimes V_0^{\Fr^{n-1}} \otimes (V_{p-2} \otimes \det)^{\Fr^{n}}, \\
& V_{r}^{}\otimes V_0^{\Fr} \otimes \cdots
\otimes V_0^{\Fr^{n}} \otimes V_{1}^{\Fr^{n+1}}, \\
& V_{r}^{}\otimes V_0^{\Fr^{}} \otimes \cdots
\otimes V_0^{\Fr^{n-1}} \otimes (V_{p-2} \otimes \det)^{\Fr^{n}}, \\
& \cdots,\\
& V_{r}^{}\otimes V_0^{\Fr^{}} \otimes (V_{p-2} \otimes \det)^{\Fr^{2}}
\otimes V_{p-1}^{\Fr^{3}} \otimes \cdots \otimes V_{p-1}^{\Fr^{n}},\\
& V_{r}^{}\otimes (V_{p-2} \otimes \det) ^{\Fr^{}} \otimes V_{p-1}^{\Fr^{2}}
\otimes \cdots \otimes V_{p-1}^{\Fr^{n}}. \\
\end{eqnarray*}
\end{lemma}
\begin{rem}
  \label{rem: V_0 are written in, and there is a symmetry}Of course we
  could omit the $V_0$ factors in the expressions, as $V_0$ is the
  trivial representation, but we have left them in place for
  clarity. Note also that the $i$th and $(2n-i)$th layers of the
  filtration are isomorphic.
\end{rem}
\begin{proof}[Proof of Lemma~\ref{lem:tricky socle filtration}]
We proceed by induction on $n$, the case $n = 0$ being trivial.
Suppose now that the result is proved for some value of $n$;
writing $V_r(n+1) = V_r(n) \otimes V_{p-1}^{\Fr^{n+1}},$
we will deduce the claim of the lemma for $n+1$.

To this end, consider one of the extensions
\nummultline
\label{eqn:socle short exact sequences}
0 \to \soc_i V_r(n)/\soc_{i-1} V_r(n) \to \soc_{i+1} V_r(n) /\soc_{i-1} V_r(n)
\\
\to \soc_{i+1}V_r(n)/\soc_i V_r(n) \to 0
\end{multline}
arising from the socle filtration of $V_r(n)$,
where we first suppose that either $i < n-1$ or $i > n$. 
Tensoring with $V_{p-1}^{\Fr^{n+1}}$, it follows from Lemma~\ref{lem:tensoring 
extensions}
that the resulting short exact sequence is again non-split.

If now $i = n-1,$ then (applying our inductive hypothesis)
the short exact sequence~(\ref{eqn:socle short exact sequences})
becomes
\begin{multline*}
0 \to
V_r \otimes V_0^{\Fr} \otimes \ldots \otimes V_0^{\Fr^{n-1}} \otimes
(V_{p-2} \otimes \det)^{\Fr^n} 
\\
\to
\soc_{n} V_r(n)/\soc_{n-2} V_r(n)
\\
\to
V_r \otimes V_0^{\Fr} \otimes \ldots \otimes V_0^{\Fr^{n-1}} \otimes V_{0}^{\Fr^n} 
\otimes V_1^{\Fr^{n+1}} 
\to  0,
\end{multline*}
and tensoring with $V_{p-1}^{\Fr^{n+1}}$ yields the short exact sequence
\begin{multline*}
0 \to
V_r \otimes V_0^{\Fr} \otimes \ldots \otimes V_0^{\Fr^{n-1}} \otimes
(V_{p-2} \otimes \det)^{\Fr^n} 
\otimes V_{p-1}^{\Fr^{n+1}}
\\
\to
\bigl(\soc_{n} V_r(n)/\soc_{n-2} V_r(n) \bigr) \otimes V_{p-1}^{\Fr^{n+1}}
\\
\to
V_r \otimes V_0^{\Fr} \otimes \ldots \otimes V_0^{\Fr^{n-1}} \otimes V_{0}^{\Fr^n} 
\otimes (V_1\otimes V_{p-1})^{\Fr^{n+1}} 
\to  0.
\end{multline*}
Now, as a particular case of the above explicit description of extensions, 
we know that the socle of
$
V_r \otimes V_0^{\Fr} \otimes \ldots \otimes V_0^{\Fr^{n-1}} \otimes V_{0}^{\Fr^n} 
\otimes (V_1\otimes V_{p-1})^{\Fr^{n+1}} $
is equal to
$
V_r \otimes V_0^{\Fr} \otimes \ldots \otimes V_0^{\Fr^{n-1}} \otimes V_{0}^{\Fr^n} 
\otimes (V_{p-2} \otimes \det)^{\Fr^{n+1}} $,
and that the extension
of
$
V_r \otimes V_0^{\Fr} \otimes \ldots \otimes V_0^{\Fr^{n-1}} \otimes V_{0}^{\Fr^n} 
\otimes (V_{p-2} \otimes \det)^{\Fr^{n+1}} $
by
$
V_r \otimes V_0^{\Fr} \otimes \ldots \otimes V_0^{\Fr^{n-1}} \otimes
(V_{p-2} \otimes \det)^{\Fr^n} 
\otimes V_{p-1}^{\Fr^{n+1}}$
induced by the preceding short exact sequence is non-split.

Putting together everything shown so far, and taking into account
the inductive hypothesis,
we conclude that $\soc_i V_r(n+1)$ has the structure claimed
in the statement of the lemma for $i \leq n$.  
A similar analysis applies if we tensor~(\ref{eqn:socle short exact sequences})
with $V_{p-1}^{\Fr^{n+1}}$ in the case when $i = n$, and so we likewise conclude
that $\cosoc_i V_r(n+1)$ has the claimed structure when $i \leq n$.

To complete the proof of the lemma, it suffices to show that the socle filtration
of $\bigl(\soc_n V_r(n)/\soc_{n-1} V_r(n) \bigr)\otimes V_{p-1}^{\Fr^{n+1}}$ has length three,
with Jordan--H\"older consituents (in order) equal to
$
V_{r}\otimes V_0^{\Fr} \otimes \cdots
\otimes V_0^{\Fr^{n}} \otimes (V_{p-2} \otimes \det)^{\Fr^{n+1}},$
$ V_{r}^{}\otimes V_0^{\Fr} \otimes \cdots
\otimes V_0^{\Fr^{n+1}} \otimes V_{1}^{\Fr^{n+2}}, $
and 
$ V_{r}^{}\otimes V_0^{\Fr^{}} \otimes \cdots
\otimes V_0^{\Fr^{n}} \otimes (V_{p-2} \otimes \det)^{\Fr^{n+1}}.$
Now (again using our inductive hypothesis) we have that
\begin{multline*}
\bigl(\soc_n V_r(n)/\soc_{n-1} V_r(n) \bigr)\otimes V_{p-1}^{\Fr^{n+1}}
\\
= 
V_r \otimes V_0^{\Fr} \otimes \ldots \otimes V_0^{\Fr^{n-1}} \otimes V_{0}^{\Fr^n} 
\otimes (V_1\otimes V_{p-1})^{\Fr^{n+1}}.
\end{multline*}

From \cite[Lem.~3.8~(ii)]{Breuil_Paškūnas_2012} we see
that
$V_1\otimes V_{p-1}$ is equal to the representation denoted there by $R_{p-2}$.
From
\cite[Lems.~3.4, 3.5]{Breuil_Paškūnas_2012} and the surrounding
discussion we see
that $R_{p-2}$ has socle equal to $V_{p-2} \otimes \det$, and that
the quotient of $R_{p-2}$ by its socle is equal to an extension
of $V_{p-2} \otimes \det$ by $V_1^{\Fr}$.  Since $V_{p-2}$ does not admit a non-trivial
extension of itself as a $\GL_2(k_v)$-representation
(by Proposition~\ref{prop:extensions of Serre weights}), that
extension of $V_{p-2} \otimes \det$ by $V_1^{\Fr}$  must
be non-split, and so in fact $R_{p-2}$ has a socle
filtration of length three, with constituents (in order)
equal to
$V_{p-2} \otimes \det$, $V_1^{\Fr},$ and $V_{p-2} \otimes \det$.  Our desired conclusion now
follows from this, together with Lemma~\ref{lem:tensoring extensions}.
\end{proof}

Next, we introduce some additional notation.  First, for
any $i \in \cS$, we let $l(i) := \max\{j \in \cS'
\, | \, j \leq i\}.$ (Here the ordering on $\cS$
is understood cyclically, i.e.\ $f-1 < 0$, and recall that $\cS' \neq
\emptyset$.)
If $J \subseteq \cS'$, then we write $\tJ := \{ i \in \cS \, | \,
l(i) \in J\}.$

If $i \in \cS',$ then define
$$X(i) := V_{p-2-r_i}^{\Fr^i}\otimes V_1^{\Fr^{i+1}} 
\otimes \bigotimes_{j \in \cE \text{
s.t. } l(j) = i} V_{p-1}^{\Fr^j}.$$
In terms of this definition (and taking into account
the preceding description of the cokernel of~(\ref{eqn:V embedding})),
we see that for any
$J \subseteq \cS'$, 
there is an isomorphism
\numequation
\label{eqn:X iso}
W_J =
\bigotimes_{i \in J} X(i)\otimes \bigotimes_{i \not\in \tJ} (V_{r_i}\otimes \det{}^{p-1-r_i})^{\Fr^i}.
\end{equation}

\begin{lemma}
\label{lem:X socle filtration}
For any $i \in \cS'$ the lattice of subrepresentations
of $X(i)$ is totally ordered, and thus $X(i)$ admits a unique
Jordan--H\"older filtration, which coincides with both its socle
and cosocle filtration.  If we let $i+n$ denote the maximal element of $\cS$
for which $l(i+n) = i,$ then the length of this filtration is $2n+1$,
and the Jordan--H\"older factors of this filtration, taken in
order
are
\begin{eqnarray*}
& V_{p-2-r_i}^{\Fr^i}\otimes (V_{p-2} \otimes \det)^{\Fr^{i+1}} \otimes V_{p-1}^{\Fr^{i+2}}
\otimes \cdots \otimes V_{p-1}^{\Fr^{i+n}}, \\
& V_{p-2-r_i}^{\Fr^i}\otimes V_0^{\Fr^{i+1}} \otimes (V_{p-2}\otimes \det)^{\Fr^{i+2}}
\otimes V_{p-1}^{\Fr^{i+3}} \otimes \cdots \otimes V_{p-1}^{\Fr^{i+n}},\\
& \cdots,\\
& V_{p-2-r_i}^{\Fr^i}\otimes V_0^{\Fr^{i+1}} \otimes \cdots
\otimes V_0^{\Fr^{i+n-1}} \otimes (V_{p-2}\otimes \det)^{\Fr^{i+n}}, \\
& V_{p-2-r_i}^{\Fr^i}\otimes V_0^{\Fr^{i+1}} \otimes \cdots
\otimes V_0^{\Fr^{i+n}} \otimes V_{1}^{\Fr^{i+n+1}}, \\
& V_{p-2-r_i}^{\Fr^i}\otimes V_0^{\Fr^{i+1}} \otimes \cdots
\otimes V_0^{\Fr^{i+n-1}} \otimes (V_{p-2} \otimes \det)^{\Fr^{i+n}}, \\
& \cdots,\\
& V_{p-2-r_i}^{\Fr^i}\otimes V_0^{\Fr^{i+1}} \otimes (V_{p-2} \otimes \det)^{\Fr^{i+2}}
\otimes V_{p-1}^{\Fr^{i+3}} \otimes \cdots \otimes V_{p-1}^{\Fr^{i+n}},\\
& V_{p-2-r_i}^{\Fr^i}\otimes (V_{p-2} \otimes \det)^{\Fr^{i+1}} \otimes V_{p-1}^{\Fr^{i+2}}
\otimes \cdots \otimes V_{p-1}^{\Fr^{i+n}}. \\
\end{eqnarray*}
\end{lemma}
\begin{proof}
This follows immediately from Lemma~\ref{lem:tricky socle filtration}.
\end{proof}

We let $\cP$ denote the collection of subsets $J \subseteq \cS$
for which $i \in J \cap \cE$ implies that $i -1 \in J$.
The following proposition constructs, for each element $J \in \cP$,
a subrepresentation $F_J$ of
$\bigotimes_{i = 0}^{f-1} V_{2p - 2 - r_i}^{\Fr^i}$.
This construction refines the correspondence $J \mapsto F_J'$
(for $J \subseteq \cS'$) introduced above.

\begin{prop}\label{prop: defn of F_J}
There is a unique way to define an order-preserving map $J \mapsto F_J$
from $\cP$ to the lattice of subrepresentations of 
$\bigotimes_{i = 0}^{f-1} V_{2p - 2 - r_i}^{\Fr^i}$
such that:
\begin{enumerate}
\item $F_J \subset F'_{J \cap \cS'}$.
\item  The inclusion
$$\sum_{J' \in \cP,\ J'\subsetneq J \text{ s.t. } \atop J'\cap \cS' \subsetneq J\cap \cS'}
\!\!\!\!\!\!\!\!\!
F_{J'} \, \subset \, F_J \cap F'_{< J\cap \cS'}$$
{\em (}which holds by {\em (1)} together with the fact that $J'\mapsto F_{J'}$
is order preserving{\em )} is an equality.
\item If for each $i \in J\cap \cS'$ we let $n_i(J)$ denote
the largest integer for which
$\{i+1,\ldots,i+n_i(J)\}\subseteq  J \cap \cE,$
then the image of $F_J/(F_J \cap F'_{< J\cap \cS'}) = (F_J + F'_{< J\cap \cS'})/F'_{< J \cap \cS'}$
in $W_J = F'_{J\cap \cS'}/F'_{< J \cap \cS'}$
is identified 
under the isomorphism~{\em (\ref{eqn:X iso})}
with the tensor product
$$\bigotimes_{i \in J\cap \cS'} \soc_{n_i(J)}X(i) \otimes
\bigotimes_{i \not\in \tJ} (V_{r_i} \otimes \det{}^{p-1-r_i})^{\Fr^i}.$$
{\em (}Here $\soc_{n_i(J)}X(i)$ denotes the $n_i(J)$th stage of the socle filtration on $X(i)$.{\em )}
\end{enumerate}
\end{prop}
\begin{proof}
The given conditions allow us to recursively define $F_J$,
beginning with $F_{\emptyset}~=~\bigotimes_{i=0}^{f-1}
V_{r_i}^{\Fr^i}$. It is obvious that conditions (1)-(3) determine
$F_J\subset F'_{J\cap\cS'}$ uniquely.
\end{proof}

Note that $\cS$ is the (unique) maximal element of $\cP$.
\begin{lemma}
\label{lem:limitation on multiplicity one subobjects}
If $V$ is a subrepresentation of 
$\bigotimes_{i = 0}^{f-1} V_{2p - 2 - r_i}^{\Fr^i}$ with socle equal to
$\bigotimes_{i = 0}^{f-1} (V_{r_i}\otimes \det^{p-1-r_i})^{\Fr^i}$,
then the multiset $\JH(V)$ of Jordan--H\"older factors of $V$ satisfies
multiplicity one if and only if $V \subseteq
F_{\cS}$. Furthermore, any $\F$-representation $W$ of $\GL_2(k_v)$
with socle~$\bigotimes_{i = 0}^{f-1} V_{r_i}^{\Fr^i}$ which satisfies
multiplicity one is isomorphic to some such $V$.
\end{lemma}
\begin{rem}
Provided that some $r_i\ne 0$, every subrepresentation of $V$ has
socle equal to $\bigotimes_{i = 0}^{f-1} (V_{r_i}\otimes \det^{p-1-r_i})^{\Fr^i}$.
\end{rem}
\begin{proof}[Proof of Lemma~\ref{lem:limitation on multiplicity one subobjects}]
Write $n_i:=n_i(\cS)$. An examination of the statement of Lemma~\ref{lem:X socle filtration}
shows that if $j > n_i$ then the multiset of Jordan--H\"older
constituents of $\soc_j X(i)$ does {\em not} satisfy multiplicity one.
From this a straightforward induction shows that if $\JH(V)$ does satisfy multiplicity
one, then $F'_J \cap V = F_{\tJ} \cap V$ for each $J \subset \cS'$.  Taking
$J = \cS'$, so that $\tJ = \cS$,  we conclude that $V \subseteq F_{\cS},$
as claimed.

To show that $F_{\cS}$ satisfies multiplicity one, set $F_{<
  J}:=\sum_{J'\subsetneq J} F_{J'}$. Then by Proposition~\ref{prop:
  defn of F_J}(3), we see that \[F_J/F_{<J}\cong \bigotimes_{i \in
  J \cap \cS'}\left(\soc_{n_i(J)}X(i)/\soc_{n_i(J)-1}X(i)\right) \otimes \bigotimes_{i
  \not\in \tJ} V_{r_i}^{\Fr^i}.\] 
Suppose that $\sigmabar_\lambda$ is a Jordan--H\"older factor of
$F_J/F_{<J}$. From the explicit description of the socle filtration on
$X(i)$, it is easy to check (using
\cite[Lem.~3.8(i)]{Breuil_Paškūnas_2012} whenever $n_i(J) = n_i(\cS)$)  that we have $l(\sigmabar_\lambda)=|J|$,
and it is also clear that $\lambda$ determines $J$. Therefore, in
order to check that $F_{\cS}$ satisfies multiplicity one, it suffices
to show that each $F_J/F_{<J}$ satisfies multiplicity one, which is
immediate from its explicit description.

Finally, note that any $W$ with socle~$\bigotimes_{i = 0}^{f-1} (V_{r_i}\otimes \det^{p-1-r_i})^{\Fr^i}$ which satisfies
multiplicity one is isomorphic to a subrepresentation
of~$\bigotimes_{i = 0}^{f-1} V_{2p - 2 - r_i}^{\Fr^i}$ by Proposition~3.6 of~\cite{Breuil_Paškūnas_2012}.
\end{proof}

\begin{prop}\label{prop: socle filtration on F_S}
The socle filtration of $F_{\cS}$ is given by
$$\soc_n F_{\cS} = \sum_{J \in \cP \atop |J| = n} F_J.$$
Furthermore, all extensions between constituents of successive 
layers of the socle filtration which can be non-zero
according to Proposition~{\em \ref{prop:extensions of Serre weights}},
{\em are} non-zero.
\end{prop}
\begin{proof}
Set $\soc'_nF_{\cS} := \sum_{J \in \cP \atop |J| = n} F_J$. As in
the proof of Lemma~\ref{lem:limitation on multiplicity one
  subobjects}, any Jordan--H\"older factor $\sigmabar$ of
$\soc'_nF_{\cS}/\soc'_{n-1}F_{\cS}$ satisfies $l(\sigmabar)=n$. It
follows immediately from Proposition~\ref{prop:extensions of Serre
  weights} that if $\sigmabar'$ is another Jordan--H\"older factor of
$F_{\cS}$ which admits an extension with $\sigmabar$, then
$l(\sigmabar')=l(\sigmabar)\pm 1$, so in particular we see that
$\soc'_nF_{\cS}/\soc'_{n-1}F_{\cS}$ is semisimple. 

On the other hand, by Lemma~\ref{lem:tensoring extensions} (together
with the discussion preceding it) and the very definition of the $F_J$
(concretely, the description of $F_J/F_{<J}$ given in the proof of
Lemma~\ref{lem:limitation on multiplicity one subobjects}), we see
that each irreducible subrepresentation of the quotient
$\soc'_nF_{\cS}/\soc'_{n-1}F_{\cS}$ extends some irreducible
subrepresentation of the quotient $\soc'_{n-1}F_{\cS}/\soc'_{n-2}F_{\cS}$, and
indeed that any possibly non-zero extension between constituents of these
layers is in fact non-zero. Thus $\soc'_nF_{\cS}=\soc_n F_{\cS}$, and
the result follows.
\end{proof}

\begin{proof}[Proof of Theorem~\ref{thm: socle filtration for tame type
  with irred socle}]The result for the cosocle filtrations of the
$\sigmabar_J^\circ(\tau)$ follows by duality from that for the socle
filtrations of the $\sigmabar^{\circ,J}(\tau)$. In this case, after
twisting we may suppose that for some $0\le r_i\le p-1$ not all equal
to $p-1$, we have $\sigmabar_J(\tau)\cong\bigotimes_{i=0}^{f-1}
(V_{r_i}\otimes \det^{p-1-r_i})^{\Fr^i}$. By Lemma~\ref{lem:weights in
  tame types have multiplicity one}, $\sigmabar^{\circ,J}(\tau)$
satisfies multiplicity one, and it has socle  $\sigmabar_J(\tau)\cong\bigotimes_{i=0}^{f-1}
(V_{r_i}\otimes \det^{p-1-r_i})^{\Fr^i}$ by definition, so we see from
Lemma~\ref{lem:limitation on multiplicity one subobjects} that it is
isomorphic to a subrepresentation of $F_{\cS}$.

The result is then immediate from Proposition~\ref{prop: socle
  filtration on F_S} upon noting that for each Jordan--H\"older factor
$\sigmabar_{J'}(\tau)$ we have $l(\sigmabar_{J'}(\tau))=|J\triangle
J'|$, and that (as already noted in the proof of
Proposition~\ref{prop: socle filtration on F_S}), the irreducible
constituents of the $i$th layer of the socle filtration of $F_{\cS}$
are precisely those $\sigmabar$ with $l(\sigmabar)=i$.
\end{proof}

             \subsection{Specific lattices and their
               gauges}
\label{subsec: specific lattices and gauges}
           Maintain the notation
             and assumptions of the previous sections, so that $\tau$
             is a non-scalar tame inertial type, and for each $J \in \cP_\tau$, 
             we have $\GL_2(\cO_{F_v})$-invariant $\cO$-lattices
             $\sigma^{\circ,J}(\tau)$ and $\sigma^\circ_J(\tau)$ in
             $\sigma(\tau)$, which are uniquely determined up to
             scaling.
 In this section, we choose 
             $E=\cO[\frac 1p]$ to be unramified over $\Qp$; this is
             possible by Lemma~\ref{lem:weights in tame types have multiplicity one}.
 If $\tau$ is cuspidal, assume in addition
             that $\tau$ is regular. The arguments of this section are
             an abstraction and generalisation of the proof of
             Th\'eor\`eme 2.4 of~\cite{breuillatticconj}.

             We can and do fix some
             $\Jbase\in\cP_\tau$ with $\Jbase^c\in\cP_\tau$; if $\tau$ is a
             principal series type, then we take $\Jbase=\emptyset$, and if
             $\tau$ is cuspidal then by Lemma~\ref{lem: regular cuspidals
               allow relabelling} we can and do set
             $\Jbase=\{i,\dots,f-1\}$ for some $i$. We then
             have an involution $\base$ on the subsets of $\cS$ given by
             $\base(J):=J\triangle\Jbase$, so that $\base(\Jbase)=\emptyset$ and $\base(\Jbase^c)=\cS$. Note that if $\tau$ is a principal series
             type, then $\base(J)=J$ for all $J$. We will need the
             following simple combinatorial lemma.
\begin{lem}\label{lem:chains in cP_tau} 
For any $J \subseteq  J'\subseteq\cS$ such
  that $\base(J),\base(J')\in\cP_\tau$, we can find a chain $J = J_{0} \subsetneq J_{1}
  \subsetneq \cdots \subsetneq J_{|J'|-|J|}= J'$ such that $|J_i|=i+|J|$ and
  $\base(J_i) \in \cP_{\tau}$ for each~$i$.
\end{lem}
\begin{proof}

It suffices to consider the case $J' = \cS$.
  (Observe that
  $\base(J_i \cap J') \in \cP_{\tau}$, and so replacing each
  $J_i$ with $J_i \cap J'$ and discarding duplicates will give the
  desired chain.)  The principal series case is part of \cite[Lem.~2.1(iii)]{breuillatticconj},
so we need only consider the (regular) cuspidal case. We must show
that
 if $J\ne\cS$ with $\base(J)\in\cP_\tau$, then
there is some $j\notin J$ such that $\base(J\cup\{j\})\in\cP_\tau$.
We will prove this directly from the definition of $\cP_\tau$.
We use the
notation of Section~\ref{ss:cuspidal}.

Continue to assume that $\tau$ is a regular cuspidal type, so that as above
we have $\Jbase=\{i,\dots,f-1\}$ for some $i$ with $0<c_i<p-1$.  For
the sake of contradiction suppose that we have some 
$J\ne\cS$ with $\base(J)\in\cP_\tau$ and for all $j\not\in J$,
$\base(J\cup\{j\})\notin\cP_\tau$. Then we see that for each $j\notin J$ with $0\le j\le i-1$, either\begin{itemize}
\item $j-1\notin \base(J)_0$ and $c_{j}=p-1$, or
\item  $j+1\notin \base(J)$ and $c_{j+1}=0$,
\end{itemize}
and for each $j\notin J$ with $i\le j\le f-1$, either 
\begin{itemize}
\item $j-1\in \base(J)_0$ and $c_{j}=0$, or
\item $j\ne f-1$, $j+1\notin J$ and $c_{j+1}=p-1$, or
\item $j=f-1$, $0\notin J$, and $c_0=0$.
\end{itemize}
Note that in the final bullet point we have eliminated the possibility
$j = f-1$, $i =
0$, $0 \in J$, and $c_0 =0 $ because this contradicts $c_i \neq 0$.

Suppose that $i\notin J$. Then since $c_i\ne 0$, we see that if $i<f-1$
then $c_{i+1}=p-1$ and $i+1\notin J$. Continuing in this fashion, we see
that $i+1,\dots,f-1\notin J$, and that
$c_{i+1}=\cdots=c_{f-1}=p-1$. Applying the same analysis again, we see
that $0\notin J$ and $c_0=0$.  If $i=0$ this is a contradiction;
otherwise, continuing in the same way, we deduce
that $c_1=0$ and $1\notin \base(J)$, and eventually we conclude that
$c_i=0$, a contradiction.

Suppose now that $i\in J$. Suppose that there is some $0\le j\le
i-1$ with $j\notin J$, and without loss of generality let $j$ be maximal
with this property. Then we see that $c_j=p-1$ and
$j-1\notin\base(J)_0$ (using $c_i \neq 0$ in case $j=i-1$),
so continuing in this fashion, we have in particular $0\notin J$,
$c_0=p-1$, and $f-1\notin J$. Then $c_{f-1}=0$, $i \neq f-1$, and
$f-2\in\base(J)_0$; so $f-2\notin J$, and then $f-3\notin J,\dots,i\notin
J$, a contradiction. So we see that $j\in J$ for all $0\le j\le i$,
and thus we must have $j\notin J$ for some $i+1\le j\le f-1$. Let $j$ be
maximal with this property. Then we see that $c_j=0$ and
$j-1\in\base(J)$ (recalling in case $j = f-1$ that we have $0 \in J$), so $j-1\notin J$, and then $c_{j-1}=0$, and
continuing in this fashion we obtain $j-2\notin J,\dots, i\notin J$, a
contradiction.\end{proof}
If $\sigma^\circ$ is any $\GL_2(\cO_{F_v})$-invariant $\cO$-lattice
in $\sigma(\tau)$ and $\base(J) \in \cP_{\tau}$,
then we let $\varepsilon_{\base(J)}(\sigma^\circ)$ denote the least integer
such that $p^{\varepsilon_{\base(J)}(\sigma^\circ)} \sigma^\circ_{\base(J)}(\tau) \subseteq \sigma^\circ$. 
The tuple $(p^{\varepsilon_{\base(J)}(\sigma^\circ)})_{\base(J)\in \cP_\tau}$ is the
gauge of~$\sigma^\circ$. It is easy to see that the following fundamental property holds.
\begin{property}\label{property: first gauge axiom}
  For any lattice $\sigma^\circ$, and for any $J,J'$ with
  $\base(J),\base(J')\in \cP_\tau$, we
  have
  \[\varepsilon_{\base(J')}(\sigma^\circ_{\base(J)}(\tau)) \geq
  \varepsilon_{\base(J')}(\sigma^\circ) - \varepsilon_{\base(J)}(\sigma^\circ),\] with
  equality if and only if $\sigmabar^\circ$ contains a subquotient
  with socle $\sigmabar_{\base(J')}(\tau)$ and cosocle $\sigmabar_{\base(J)}(\tau)$.
\end{property}

Note that if we take $\sigma^\circ$ to be $\sigma^{\circ,\base(J')}(\tau)$, then for any $\sigmabar_{\base(J)}(\tau)$, 
the quotient $\sigmabar^{\circ,{\base(J')}}(\tau)$ contains a subobject with socle $\sigmabar_{\base(J')}(\tau)$ and cosocle
$\sigmabar_{\base(J)}(\tau)$, and so in particular we deduce that
\numequation
\label{eqn:gauge formula}
\varepsilon_{\base(J')}(\sigma^\circ_{\base(J)}(\tau)) = \varepsilon_{\base(J')}(\sigma^{\circ,\base(J')}(\tau)) - \varepsilon_{\base(J)}(\sigma^{\circ,\base(J')}(\tau)).
\end{equation}

We fix $\sigma^{\circ,{\base(\emptyset)}}(\tau)$ once and for all.  We then
normalize the $\sigma^\circ_{\base(J)}(\tau)$ by requiring that
$\varepsilon_{\base(J)}(\sigma^{\circ,{\base(\emptyset)}}(\tau)) = 0,$ and we normalize
the $\sigma^{\circ,\base(J)}(\tau)$ by requiring that
$\varepsilon_{{\base(\emptyset)}}(\sigma^{\circ,\base(J)}(\tau)) = 0.$ Note that, with
these normalizations, (\ref{eqn:gauge formula}) implies that $
\varepsilon_{{\base(\emptyset)}}(\sigma^\circ_{\base(J)}(\tau))=0$, and
Property~\ref{property: first gauge axiom} then implies that
$$0 = \varepsilon_{{\base(\emptyset)}}(\sigma^\circ_{\base(J)}(\tau))\geq \varepsilon_{{\base(\emptyset)}}(\sigma^{\circ,\base(J')}(\tau))
- \varepsilon_{\base(J)}(\sigma^{\circ,\base(J')}(\tau)) = 0 - \varepsilon_{\base(J)}(\sigma^{\circ,\base(J')}(\tau)),$$
i.e.\ that
\numequation
\label{eqn:non-negativity}
\varepsilon_{\base(J)}(\sigma^{\circ,\base(J')}(\tau)) \geq 0 \quad
\text{ for all } 
 J,J' \text{ with } \base(J),\base(J') \in \cP_{\tau}.
\end{equation}

Our key input is the following.
\begin{prop}\label{prop: facts that drive gauges.tex}
\begin{enumerate}
  \item  We have
$p^f \sigma^{\circ,{\base(\emptyset)}}(\tau) \subseteq
  \sigma^\circ_{\base(\emptyset)}(\tau) \subseteq
  \sigma^{\circ,{\base(\emptyset)}}(\tau),$
and $\sigma^\circ_{\base(\cS)}(\tau)=\sigma^{\circ,{\base(\emptyset)}}(\tau)$.

\item 
The sets $J$ with $\base(J) \in\cP_\tau$ such that $\sigmabar_{\base(J)}(\tau)$ appears in the
$i$th layer of the socle filtration of
$\sigmabar^\circ_{\base(\emptyset)}(\tau)$
are precisely those $J$ with
$|J|=f-i$.

\item
$\sigmabar^{\circ,{\base(\emptyset)}}(\tau)$ contains
a subquotient with socle $\sigmabar_{\base(J')}(\tau)$ and cosocle $\sigmabar_{\base(J)}(\tau)$
if and only if $J' \subseteq J,$ 
and $\sigmabar^\circ_{\base(\emptyset)}(\tau)$ contains
a subquotient with socle $\sigmabar_{\base(J')}(\tau)$ and cosocle $\sigmabar_{\base(J)}(\tau)$ 
if and only if  $J \subseteq J'$. 
\end{enumerate}
\end{prop}
\begin{proof}
  (1) By Theorem~\ref{thm: socle filtration for tame type with irred
    socle}, $\sigmabar^{\circ,\base(\emptyset)}(\tau)$ has cosocle
  $\sigmabar_{\base(\cS)}(\tau)$, so the assumption that
  $\varepsilon_{\base(\cS)}(\sigma^{\circ,{\base(\emptyset)}}(\tau))=0$ implies that
  $\sigma^\circ_{\base(\cS)}(\tau)=\sigma^{\circ,{\base(\emptyset)}}(\tau)$.

  By Lemma~\ref{lem:weights in tame types have multiplicity one},
  there is a character $\chi$ such that
  $\sigma(\tau)^\vee\otimes\chi\cong\sigma(\tau)$, and such that each
  $\sigmabar_{\base(J)}(\tau)$ satisfies
  $\sigmabar_{\base(J)}(\tau)^\vee\otimes\chibar\cong\sigmabar_{\base(J)}(\tau)$.
There is thus an isomorphism  $(\sigma^{\circ,{\base(\emptyset)}}(\tau))^\vee\otimes\chi
  \cong \sigma^\circ_{\base(\emptyset)}(\tau)$, and the result then follows from
  Proposition~\ref{prop: embedding of dual killed by order of group}
  and the assumption that
  $\varepsilon_{{\base(\emptyset)}}(\sigma^{\circ,{\base(\emptyset)}}(\tau))=0$.

(2)  This is immediate from (1) and Theorem~\ref{thm: socle filtration for tame
  type with irred socle}.

(3)  Suppose $J' \subseteq J$.  An easy induction on $|J\setminus J'|$ shows that any subrepresentation of
$\sigmabar^{\circ,\base(\varnothing)}(\tau)$ with
$\sigmabar_{\base(J)}(\tau)$ as a Jordan--H\"older factor also has
$\sigmabar_{\base(J')}(\tau)$  as a Jordan--H\"older factor.  (In the
case $|J \setminus J'| = 1$ consider the quotient of
$\sigmabar^{\circ,\base(\varnothing)}(\tau)$ by the maximal
subrepresentation containing neither $\sigmabar_{\base(J)}(\tau)$ nor
$\sigmabar_{\base(J')}(\tau)$ as a Jordan--H\"older factor, which must have socle
$\sigmabar_{\base(J')}(\tau)$ by the paragraph following Theorem~\ref{thm: socle filtration for
  tame type with irred socle}.  For the induction step use Lemma~\ref{lem:chains in cP_tau}.)
The
claim for $\sigmabar^{\circ,\base(\varnothing)}(\tau)$ follows easily
from this, and the result for $\sigmabar^{\circ}_{\base(\varnothing)}(\tau)$ is similar.
\end{proof}

From Property~\ref{property: first gauge axiom} and
Proposition~\ref{prop: facts that drive gauges.tex}(3), we see that
\numequation
\label{eqn:first relation}
\varepsilon_{\base(J')}(\sigma^\circ_{\base(J)}(\tau)) = 0 \quad \text{ if } \quad J'\subseteq J,
\end{equation}
that
\numequation
\label{eqn:second relation}
\varepsilon_{\base(J')}(\sigma^\circ_{\base(J)}(\tau)) > 0 \quad \text{ if } \quad J'\not\subseteq J,
\end{equation}
and that
\numequation
\label{eqn:third relation}
\varepsilon_{\base(J')}(\sigma^\circ_{\base(J)}(\tau)) = \varepsilon_{\base(J')}(\sigma^\circ_{\base(\emptyset)}(\tau)) -\varepsilon_{\base(J)}(\sigma^\circ_{\base(\emptyset)}(\tau))
\quad \text{ if } \quad J \subseteq J'.
\end{equation}
(There is a fourth relation in the case $J\not\subseteq J'$, but
we will not need this.)

Now combining relations~(\ref{eqn:second relation})
and~(\ref{eqn:third relation}), we see that if we partially order the
subsets of $\base(\cP_\tau)$ by inclusion, then
$\varepsilon_{\base(J)}(\sigma^\circ_{\base(\emptyset)}(\tau))$ is a strictly increasing
function of $J$. By Proposition~\ref{prop: facts that drive
  gauges.tex}(1) we have
$\varepsilon_{\base(\cS)}(\sigma^\circ_{\base(\emptyset)}(\tau))\le f$, so we see that
$0\le \varepsilon_{\base(J)}(\sigma^\circ_{\base(\emptyset)}(\tau))\le f$ for each
$J\in\base(\cP_\tau)$.  By Lemma~\ref{lem:chains in cP_tau} (applied twice,
to the pairs $\varnothing, J$ and $J, \cS$), each $J \in \base(\cP_\tau)$ lies in a chain of
elements of $\base(\cP_\tau)$ of length $f+1$, so we conclude that in fact
$$\varepsilon_{\base(J)}(\sigma^\circ_{\base(\emptyset)}(\tau)) = | J|,$$
and that 
\numequation
\label{eqn:first value}
\varepsilon_{\base(J')}(\sigma^\circ_{\base(J)}(\tau)) = |J' \setminus J| \quad \text{ if } \quad 
J \subseteq J'.
\end{equation}
As an application of this, note that
if we take $J = {\base(\emptyset)}$ in relation~(\ref{eqn:gauge formula}),
we find that
\numequation
\label{eqn:second value}
\varepsilon_{\base(J')}(\sigma^{\circ,\base(J')}(\tau)) = |J'|.
\end{equation}
As another application,
 combining~(\ref{eqn:first relation}) and~(\ref{eqn:first value})
with Property~\ref{property: first gauge axiom}, we find that for any
$J,J'\in\base(\cP_\tau)$ we have
\begin{samepage}
\nummultline
\label{eqn:first inequality}
|J'\setminus J| = \varepsilon_{\base(J')}(\sigma^\circ_{\base(J\cap J')}(\tau))
\\
\geq 
\varepsilon_{\base(J')}(\sigma^\circ_{\base(J)}(\tau)) - \varepsilon_{J\cap J'}(\sigma^\circ_{\base(J)}(\tau)) = \varepsilon_{\base(J')}(\sigma^\circ_{\base(J)}(\tau)).
\end{multline}
\end{samepage}
Fix some $J,J'\in\base(\cP_\tau)$, and consider a chain $J' =
J_0,J_1,\ldots,J_n = J$ joining $\sigmabar_{\base(J')}(\tau)$ to
$\sigmabar_{\base(J)}(\tau)$ in the socle filtration of
$\sigmabar^{\circ,\base(J')}(\tau)$ (so that $\sigmabar_{\base(J_i)}(\tau)$ is in
the $i$-th layer of the socle filtration of
$\sigmabar^{\circ,\base(J')}(\tau)$, and is non-trivially extended by
$\sigmabar_{\base(J_{i-1})}(\tau)$ in the socle filtration). By
Theorem~\ref{thm: socle filtration for tame type with irred socle},
$J_{i+1}$ is obtained from $J_i$ by either adding or removing
an element.
  Property~\ref{property: first gauge axiom} gives that
$$\varepsilon_{\base(J_i)}(\sigma^\circ_{\base(J_{i+1})}(\tau)) = \varepsilon_{\base(J_i)}(\sigma^{\circ,\base(J')}(\tau)) -
\varepsilon_{\base(J_{i+1})}(\sigma^{\circ,\base(J')}(\tau)).$$
From~(\ref{eqn:first relation}) we conclude that 
$$\varepsilon_{\base(J_i)}(\sigma^{\circ,\base(J')}(\tau)) = \varepsilon_{\base(J_{i+1})}(\sigma^{\circ,\base(J')}(\tau))
\quad \text{ if } J_i \subseteq J_{i+1},$$
while from~(\ref{eqn:first value})
we conclude that
$$\varepsilon_{\base(J_i)}(\sigma^{\circ,\base(J')}(\tau)) = \varepsilon_{\base(J_{i+1})}(\sigma^{\circ,\base(J')}(\tau)) + 1
\quad \text{ if } J_{i+1} \subseteq J_{i}.$$
Passing along the whole chain from $J'$ to $J$,
we find that 
\numequation
\label{eqn:intermediate inequality}
\varepsilon_{\base(J)}(\sigma^{\circ,\base(J')}(\tau)) \leq \varepsilon_{\base(J')}(\sigma^{\circ,\base(J')}(\tau)) - |J'\setminus J|.
\end{equation}
Combining this with Property~\ref{property: first gauge axiom} again gives
\numequation
\label{eqn:second inequality}
\varepsilon_{\base(J')}(\sigma^\circ_{\base(J)}(\tau)) = \varepsilon_{\base(J')}(\sigma^{\circ,\base(J')}(\tau)) - \varepsilon_{\base(J)}(\sigma^{\circ,\base(J')}(\tau))
\geq |J'\setminus J|.
\end{equation}
Combining~(\ref{eqn:first inequality})
and~(\ref{eqn:second inequality}), we find that
$$\varepsilon_{\base(J')}(\sigma^\circ_{\base(J)}(\tau)) = |J'\setminus J|.$$
Since we have equality in~(\ref{eqn:second inequality}),
we must also have equality in~(\ref{eqn:intermediate inequality}),
and combining this with~(\ref{eqn:second value}), we
find that
\numequation
\label{eqn:final gauge eqn}
\varepsilon_{\base(J)}(\sigma^{\circ,\base(J')}(\tau)) = |J \cap J'|.\end{equation}

\begin{thm}
  \label{thm: the output of gauges.tex}

(1) We have $\sigma^{\circ,\base(J)}(\tau)=\sum_{J'}p^{|J\cap
  J'|}\sigma^\circ_{\base(J')}(\tau)$.

(2) Suppose that $\base(J)\in\cP_\tau$, that $j\notin J$, and that
$\base(J\cup\{j\})\in\cP_\tau$. Then we have inclusions \[p\sigma^{\circ,\base(J)}(\tau)\subseteq\sigma^{\circ,\base(J\cup\{j\})}(\tau)\subseteq\sigma^{\circ,\base(J)}(\tau).\]
Furthermore the Serre weights occuring as Jordan--H\"older factors of
the cokernel of the inclusion
$\sigma^{\circ,\base(J\cup\{j\})}(\tau)\subseteq\sigma^{\circ,\base(J)}(\tau)$
are the $\sigmabar_{\base(J')}(\tau)$ with $j\in J'$, and the Serre weights occuring in the cokernel of the
inclusion $p\sigma^{\circ,\base(J)}(\tau)\subseteq\sigma^{\circ,\base(J\cup\{j\})}(\tau)$
are the $\sigmabar_{\base(J')}(\tau)$ with $j\notin J'$.

(3) We have $\sigma^{\circ}_{\base(J)}(\tau)=\sum_{J'}p^{|J^c\cap
  {J'}^c|}\sigma^{\circ,\base(J')}(\tau)$.

(4) Suppose that $\base(J)\in\cP_\tau$, that $j\notin J$, and that
$\base(J\cup\{j\})\in\cP_\tau$. Then we have inclusions \[p\sigma^{\circ}_{\base(J\cup\{j\})}(\tau)\subseteq\sigma^{\circ}_{\base(J)}(\tau)\subseteq\sigma^{\circ}_{\base(J\cup\{j\})}(\tau).\]
Furthermore the Serre weights occuring as Jordan--H\"older factors of
the cokernel of the inclusion $\sigma^{\circ}_{\base(J)}(\tau)\subseteq\sigma^{\circ}_{\base(J\cup\{j\})}(\tau)$
are the $\sigmabar_{\base(J')}(\tau)$ with $j\in J'$, and the Serre weights occuring in the cokernel of the
inclusion $p\sigma^{\circ}_{\base(J\cup\{j\})}(\tau)\subseteq\sigma^{\circ}_{\base(J)}(\tau)$
are the $\sigmabar_{\base(J')}(\tau)$ with $j\notin J'$.
\end{thm}
\begin{proof}
(1) This is immediate from~(\ref{eqn:final gauge eqn}) and Proposition~\ref{prop:the gauge
  determines the lattice}.

(2) The inclusions are immediate from (1), and the claims on the Serre
weights follow easily from considering the composite maps
\[p^{\varepsilon_{\base(J')}(\sigma^{\circ,\base(J
  \cup \{j\})}(\tau))} \sigma^\circ_{\base(J')}(\tau) = p^{|J'\cap
  (J\cup\{j\})|}\sigma^\circ_{\base(J')}(\tau)\subseteq\sigma^{\circ,\base(J\cup\{j\})}(\tau)\subseteq\sigma^{\circ,\base(J)}(\tau)\]
and comparing with
$p^{\varepsilon_{\base(J')}(\sigma^{\circ,\base(J)}(\tau))}
\sigma^{\circ}_{\base(J')}(\tau) = p^{|J' \cap J|}
\sigma^{\circ}_{\base(J')}(\tau) \subseteq
\sigma^{\circ,\base(J)}(\tau).$

(3) This follows by an argument analogous to the proof of (1).

(4) This follows by applying (2) to the contragredient of
$\sigma(\tau)$.
\end{proof}

\section{Patching functors}\label{sec: cohomology}

In this section we will introduce the cohomology groups that we will
consider, and we will apply the Taylor--Wiles--Kisin method to them. We
will follow Sections~4 and~5 of \cite{geekisin}, but we will take a more
functorial approach, and we also follow \cite{breuildiamond} in laying
the foundations for the multiplicity one results that we will prove in
Section~\ref{sec: freeness and multiplicity one}. Our functorial
version of the Taylor--Wiles--Kisin method is a refinement of the
approach taken in (\cite{kisinfmc}, \cite{emertongeerefinedBM},
\cite{geekisin}), where the patching argument is applied
simultaneously for various lattices in types, and for certain
filtrations on the reductions of these types modulo $p$; here we take
this construction to its natural level of generality, and patch
simultaneously over all finitely generated representations. (One could
also presumably take a limit over all tame levels simultaneously, but
as we have no need to do so for the applications we have in mind we
have avoided setting up the necessary foundations.) 

Of course, the resulting functor depends on the non-canonical choice
of various 
auxiliary primes, but it seems reasonable to expect that
it in fact enjoys various canonical properties. In fact, under
appropriate genericity hypotheses, the results of Sections~\ref{sec: proof of the
  conjecture} and~\ref{sec: freeness and
  multiplicity one} below show that the restriction of the patching
functor to the subquotients of lattices in a fixed tame type is
independent of the choices of auxilliary primes, and of the particular
global context; see Remark~\ref{rem: tame patching is purely local}. See also \cite{EGP} for a
further discussion of this in the case of $\GL_2/\Qp$. With this in
mind, and with an eye to future applications, we have chosen to set up
a general abstract notion of a patching functor, and then to show that
the various axioms for such a functor are satisfied in our
setting. There are a variety of possible levels of generality in which
one could work, and a variety of axioms one could require; for
clarity, we have chosen to restrict ourselves to the case of
two-dimensional representations, as this is what is used in the rest
of the paper.

\subsection{Abstract patching functors}\label{subsec: abstract
  patching functors}
Fix a
prime $p$. Continue to fix a finite extension $E$ of $\Qp$ with ring
of integers $\cO$ and residue field $\F$. In this subsection, by a local
field we mean a finite extension of some $\Ql$, with $l$ not
necessarily equal to $p$.

Fix a finite set of local fields $L_i$, write $l_i$ for the
residue characteristic of~$L_i$, and $\cO_{L_i}$ for its ring of
integers. If $l_i=p$, then we assume that $L_i/\Qp$ is unramified. For each $i$ we fix a continuous representation
$\rhobar_i:G_{L_i}\to\GL_2(\F)$. (In the patching functors we
construct, the $L_i$ will all be localisations of some global field,
and the $\rhobar_i$ will be localisations of a global Galois
representation, but with an eye to future applications we do not
assume this in the definition of a patching functor.)

For each $i$, we will let $K_i$ be either a compact open subgroup of
$\GL_2(\cO_{L_i})$, or a compact open subgroup of $\cO_{D_i}^\times$,
where $\cO_{D_i}$ is a maximal order in the nonsplit quaternion
algebra $D_i$ with centre $L_i$. If $l_i=p$, we only allow $K_i$ to be
a compact open subgroup of $\GL_2(\cO_{L_i})$.
Write $K=\prod_i K_i$,
and write
$Z_K=Z(K)$. The reduced norm and determinant maps give a natural
homomorphism $Z_K\to\prod_i\cO_{L_i}^\times$.

For each $i$ we have the universal framed deformation ring
$R^{\square}_i$ for $\rhobar_i$ over $\cO$. Write $X_i=\Spf R_i^\square$.
Write $R=\widehat{\otimes}_{\cO}R^{\square}_i$, and write $X=\Spf
R=\times_{\Spf\cO}X_i$.  Finally, fix some $h\ge 0$, and write
$R_\infty=R\llbracket x_1,\dots,x_h\rrbracket$, $X_\infty=\Spf R_\infty$, where the
$x_i$ are formal variables.

Let $\cC$
denote the category of finitely generated $\cO$-modules with a
continuous
action of $K$, and let $\cC'$ be a
Serre subcategory of $\cC$. For each $i$, let $\tau_i$ be an inertial
type for $I_{L_i}$, assumed to be discrete series if $K_i$ is a
subgroup of $\cO_{D_i}^\times$,
and let
$\sigma^\circ(\tau_i)$ denote some $\GL_2(\cO_{L_i})$-stable
$\cO$-lattice in $\sigma(\tau_i)$ (in the case that $K_i$ is a
subgroup of $\GL_2(\cO_{L_i})$), or an $\cO_{D_i}^\times$-stable
$\cO$-lattice in $\sigma_D(\tau_i)$ (in the case that $K_i$ is a
subgroup of $\cO_{D_i}^\times$). Write
$\sigma^\circ(\tau):=\otimes_{\cO}\sigma^\circ(\tau_i)$, an element of
$\cC$.

If $l_i\ne p$, we write $R^{\square,\tau}_i$ for the reduced,
$p$-torsion free quotient of $R^{\square}_i$ corresponding to
deformations of inertial type $\tau_i$, while if $l_i=p$ we write
$R^{\square,\tau}_i$ for the reduced, $p$-torsion free quotient of
$R^{\square}_i$ corresponding to potentially crystalline deformations
of inertial type $\tau_i$ and Hodge type $0$. We write
$R^\tau:=\widehat{\otimes}_{\cO}R^{\square,\tau}_i$,
$R^\tau_\infty:=R^\tau\llbracket x_1,\dots,x_h\rrbracket$, and 
$X_\infty\bigl(\tau\bigr):=\Spf(R^\tau_\infty)$.

Suppose that $\sigmabar\in\cC$ is killed by $\varpi_E$, and has the
property that for each $i$ with $l_i=p$, we have $K_i=\GL_2(\cO_{L_i})$, and the restriction of
$\sigmabar$ to $K_i$ is a direct sum of copies of an irreducible
representation $\sigmabar_i$ of $K_i$.
In this case, for each $i$ with $l_i=p$, we let
$R^{\square,\sigmabar}_i$ be the reduced, $p$-torsion free quotient of
$R^{\square}_i$ corresponding to crystalline deformations of Hodge
type $\sigmabar_i$, and for each $i$ with $l_i\ne p$ we set
$R^{\square,\sigmabar}_i=R^{\square}_i$.  Then we write
$R^{\sigmabar}:=\widehat{\otimes}_{\cO}R^{\square,\sigmabar}_i$, $R^{\sigmabar}_\infty:=R^{\sigmabar}\llbracket x_1,\dots,x_h\rrbracket$, and
$X_\infty\bigl(\sigmabar\bigr):=\Spf(R^{\sigmabar}_\infty)$. We write
$\Xbar_\infty\bigl(\sigmabar\bigr)$ for the special fibre of the space
$X_\infty\bigl(\sigmabar\bigr)$.

\begin{defn}\label{defn: patching functor}
  By a \emph{patching functor} we will mean a  covariant exact functor
  $M_\infty$ from $\cC'$ to the category of coherent sheaves on
  $X_\infty$ which satisfies the following two properties for all
  $\sigma^\circ(\tau)$ and $\sigmabar$ as above.
  \begin{itemize}
  \item $M_\infty(\sigma^\circ(\tau))$ is $p$-torsion free and
    supported on $X_\infty\bigl(\tau\bigr)$, and in fact is maximal
    Cohen--Macaulay over $X_\infty\bigl(\tau\bigr)$.

  \item $M_\infty(\sigmabar)$ is supported on
    $\Xbar_\infty\bigl(\sigmabar\bigr)$, and in fact is maximal Cohen--Macaulay
    over $\Xbar_\infty\bigl(\sigmabar\bigr)$.
  \end{itemize}

\end{defn}

\begin{rem}
  \label{rem: patching functor defined on Serre subcat} Unless stated
  otherwise, all of our patching functors will be defined on all of
  $\cC$, but it will be convenient in some arguments to have the
  additional flexibility of passing to a Serre subcategory.
\end{rem}

We will now define the notion of a \emph{fixed determinant patching
  functor}. For each $i$, in addition to the data above, we also fix a
(lift of a) geometric Frobenius element $\Frob_i\in G_{L_i}$, and we fix an
element $\alpha_i\in\cO_{L_i}$ lifting $\det\rhobar_i(\Frob_i)$. Let
$\cC_Z$ denote the subcategory of $\cC$ consisting of representations
which have a central character, so that the action of $Z_K$
factors through the natural map to $\prod_i\cO_{L_i}^\times$, and
whose central character lifts the character
$\prod_i(\varepsilonbar\det\rhobar_i|_{I_{L_i}})\circ\Art_{L_i}$.
Let $\cC'_Z$
be a Serre subcategory of $\cC_Z$.

Take some $\sigma\in\cC'_Z$. For each $i$, fix a character
$\psi_{\sigma,i}:G_{L_i}\to\cO^\times$ with the properties that
$\psi_{\sigma,i}(\Frob_i)=\alpha_i$, and the composite of $\Art_{L_i}$ and
the restriction of $\psi_{\sigma,i}$ to $I_{L_i}$  gives the
central character of $\sigma$. (If $\sigma$ is not $p$-power torsion, then
$\psi_{\sigma,i}$ is uniquely determined.)
Write
$R^{\square,\psi_\sigma}_i$ for the quotient of $R^{\square}_i$
corresponding to liftings with determinant
$\psi_{\sigma,i}\varepsilon^{-1}$, and $X^{\psi_\sigma}_i = \Spf
R^{\square,\psi_\sigma}_i$. Write
$R^{\psi_\sigma}=\widehat{\otimes}_{\cO}R^{\square,\psi_\sigma}_i$,
and set $X^{\psi_\sigma}=\Spf
R^{\psi_\sigma}=\times_{\Spf\cO}X^{\psi_\sigma}_i$.  Finally, fix some
$h\ge 0$, and write
$R^{\psi_\sigma}_\infty=R^{\psi_\sigma}\llbracket x_1,\dots,x_h\rrbracket$,
$X^{\psi_\sigma}_\infty=\Spf R^{\psi_\sigma}_\infty$, where the $x_i$
are formal variables.

Assuming that the $\tau_i$ above satisfy the condition that
$\det\tau_i$ is a lift of $\varepsilonbar\det\rhobar_i|_{I_{L_i}}$, so
that $\sigma^\circ(\tau)\in\cC_Z$, we have the obvious quotient
$R^{\square,\psi,\tau}_i$ of $R^{\square,\psi}_i$, and we write
$R^{\psi,\tau}:=\widehat{\otimes}_{\cO}R^{\square,\psi,\tau}_i$, $R^{\psi,\tau}_\infty:=R^{\psi,\tau}\llbracket x_1,\dots,x_h\rrbracket$, and
$X^\psi\bigl(\tau\bigr):=\Spf(R^{\psi,\tau}_\infty)$. If
$\sigmabar\in\cC_Z$, then we have the quotient
$R^{\square,\psi,\sigmabar}_i$ of $R^{\square,\sigmabar}_i$, and we
write
$R^{\psi,\sigmabar}:=\widehat{\otimes}_{\cO}R^{\square,\psi,\sigmabar}_i$, $R^{\psi,\sigmabar}_\infty:=R^{\psi,\sigmabar}\llbracket x_1,\dots,x_h\rrbracket$,
and $X^\psi\bigl(\sigmabar\bigr):=\Spf(R^{\psi,\sigmabar}_\infty)$. We
write $\Xbar^\psi\bigl(\sigmabar\bigr)$ for the special fibre of the space
$X^\psi\bigl(\sigmabar\bigr)$.

\begin{defn}\label{defn: fixed det patching functor}
  By a \emph{fixed determinant patching functor} we will mean a covariant exact functor
  $M_\infty$ from $\cC'_Z$ to the category of coherent sheaves on
  $X_\infty$ which satisfies the following properties.
  \begin{itemize}
  \item For all $\sigma\in\cC'_Z$, $M_\infty(\sigma)$ is supported on
    $X_\infty^{\psi_\sigma}$.
  \item $M_\infty(\sigma^\circ(\tau))$ is $p$-torsion free and
    supported on $X^\psi_\infty\bigl(\tau\bigr)$, and in fact is maximal
    Cohen--Macaulay over $X^\psi_\infty\bigl(\tau\bigr)$.

  \item $M_\infty(\sigmabar)$ is supported on
    $\Xbar^\psi_\infty\bigl(\sigmabar\bigr)$, and in fact is maximal Cohen--Macaulay
    over $\Xbar^\psi_\infty\bigl(\sigmabar\bigr)$.
  \end{itemize}

\end{defn}

We have the following consequence of these definitions, an abstract
version of the multiplicity one argument of
\cite{MR1440309}.
\begin{lem}
  \label{lem: abstract version of Diamond freeness argument}If
  $M_\infty$ is a patching functor, and if the deformation space
$X_\infty\bigl(\tau\bigr)$ is regular,
  then $M_\infty(\sigma^\circ(\tau))$ is free over
  $X_\infty\bigl(\tau\bigr)$. Similarly, if the generic fibre of
  $X_\infty\bigl(\tau\bigr)$ is regular, then the restriction of
  $M_\infty(\sigma^\circ(\tau))$ to the generic fibre is
  locally free. The analogous result also holds
  for fixed determinant patching functors.
\end{lem}
\begin{proof}This follows immediately from the Auslander--Buchsbaum
  formula and the maximal Cohen--Macaulay property of
  $M_\infty(\sigma^\circ(\tau))$.
  \end{proof}

  We say that a patching functor (respectively a fixed
  determinant patching functor) $M_\infty$ is \emph{minimal} if for
  any inertial type $\tau$ for which $X_\infty\bigl(\tau\bigr)$ (respectively
  $X^\psi_\infty\bigl(\tau\bigr)$) has regular generic fibre,
  $M_\infty(\sigma^\circ(\tau))$ is locally free of rank one over the
  generic fibre of $X_\infty\bigl(\tau\bigr)$ (respectively
  $X^\psi_\infty\bigl(\tau\bigr)$). Note that by Lemma \ref{lem: abstract version of
    Diamond freeness argument}, $M_\infty(\sigma^\circ(\tau))$ is
  locally free over the generic fibre of $X_\infty\bigl(\tau\bigr)$
  (respectively $X^\psi_\infty\bigl(\tau\bigr)$), so this should be regarded as a
  multiplicity one hypothesis. (The reason we call these patching
  functors minimal is that we will construct minimal patching functors
  by considering the cohomology of Shimura curves at minimal level.)

\begin{defn}
 We will sometimes say that $M_\infty$ is a patching functor \emph{indexed by}
  the $(L_i,\rhobar_i)$, or that $M_\infty$ is a patching functor for
  the $\rhobar_i$. If the coefficient field $E$ is unramified over
  $\Qp$, we will say that $M_\infty$ is a patching functor \emph{with
    unramified coefficients}. We will also use the same terminology for fixed determinant patching
  functors.
\end{defn}

  In the rest of this section we will give examples of patching
  functors arising from the Taylor--Wiles--Kisin method. These
  patching functors will satisfy an additional property which will be
  important for us, but which we have not attempted to axiomatise,
  which is that certain fibres of the patching functors will give
  spaces of automorphic forms.

\subsection{Quaternion algebras}\label{subsec: existence of
  patching functors}

In this subsection we will describe the various spaces of
automorphic forms on quaternion algebras from which we will construct
our patching functors. To this end, let $F$ be a
totally real field, and suppose that the prime $p$ is odd and is
unramified in $F$. Let $\rhobar:G_F\to\GL_2(\F)$ be a continuous
representation. Assume that $\rhobar$ is modular, and that $\rhobar|_{G_{F(\zeta_p)}}$ is
absolutely irreducible. If $p=5$, assume further that the projective
image of $\rhobar|_{G_{F(\zeta_5)}}$ is not isomorphic to~$A_5$. 
Choose a finite
order character $\psi:G_F\to\cO^\times$ such that
$\det\rhobar=\varepsilonbar^{-1}\psibar$.

Let $D$ be a quaternion algebra with centre $F$ which is ramified at a
set $\Sigma$ of finite places of $F$. We assume that $\Sigma$ does not
contain any places lying over $p$. We also assume that either $D$ is
ramified at all infinite places (the \emph{definite case}), or split at
precisely one infinite place (the \emph{indefinite case}). Let $S$ be
a set of finite places of $F$ containing $\Sigma$, the places dividing
$p$, and the places at which $\rhobar$ or $\psi$ is ramified. We will
define a patching functor indexed by the  $(F_w,\rhobar|_{G_{F_w}})$
for $w\in
S$, and from now on we will write the elements of our indexing set as
$w$, rather than as $i$ as in the previous section. We will write $K_S$ for the group $K$ in the definition of a
patching functor.

By Lemma 4.11 of \cite{MR1605752} we can and do choose a finite place
$w_1\notin S$ with the properties that
\begin{itemize}
\item $\mathbf{N}w_1\not\equiv
1\pmod{p}$,
\item  the ratio of the eigenvalues of $\rhobar(\Frob_{w_1})$
  is not equal to $(\mathbf{N}w_1)^{\pm 1}$, and
\item the residue characteristic of $w_1$ is
sufficiently large that for any non-trivial root of unity $\zeta$ in a
quadratic extension of $F$, $w_1$ does not divide $\zeta+\zeta^{-1}-2$.
\end{itemize}
We will consider compact open subgroups $K=\prod_wK_w$ of
$(D\otimes_F\A^\infty_F)^\times$ for which
$K_w\subset(\cO_D)_w^\times$ for all $w$, $K_w=(\cO_D)_w^\times$ if
$w\notin S\cup\{w_1\}$, and $K_{w_1}$ is the subgroup of
$\GL_2(\cO_{F_{w_1}})$ consisting of elements that are
upper-triangular and unipotent modulo $w_1$.  Our assumptions on
$w_1$ imply that $K$ is sufficiently small, in the sense that
condition (2.1.2) of \cite{kisinfmc} holds.

We now define the spaces of modular forms with which we will work. In
fact, the Taylor--Wiles--Kisin method naturally patches homology,
rather than cohomology, so it will be convenient for us to work with
Pontrjagin duals of spaces of modular forms.
Consider some $\sigma\in \cC_Z$. As in the definition of
a fixed determinant patching functor, we fix characters
$\psi_{\sigma,w}:G_{F_w}\to\cO^\times$ for $w \in S$, determined (when
$\sigma$ is not $p$-power torsion) by the central character of
$\sigma$. By the hypothesis on the central characters of elements of
$\cC_Z$, and the definition of $\psi$, we see that
$\psibar_{\sigma,w}^{-1}\psibar|_{G_{F_w}}=1$. 

By Hensel's lemma (and
the assumption that $p>2$), for each $w\in S$ there is a unique character
$\theta_w:G_{F_w}\to\cO^\times$ with the properties that
$\thetabar_w=1$ and
$\theta_w^2=\psi_{\sigma,w}^{-1}\psi|_{G_{F_w}}$. Write
$\theta=\otimes_w\theta_w$, and write $\sigma(\theta)$ for the twist
of $\sigma$ by $\otimes_w (\theta_w\circ\Art_{F_w}\circ\det)$.
Then $\sigma(\theta)$ has an
action of $K$ via the projection onto $K_S$, which we
extend to an action of $K \cdot (\A_F^\infty)^\times$ by letting
$(\A_F^\infty)^\times$ act via the composite
$(\A_F^\infty)^\times\to(\A_F^\infty)^\times/F^\times\to\cO^\times$,
with the second map the one induced by $\psi\circ\Art_F$. (Note that
the actions of $K$ and $(\A_F^\infty)^\times$ agree on their
intersection; since $\psi$ is unramified outside of $S$, this is
automatic at places away from $S$, and follows from the definition of
the $\theta_w$ at the places in $S$.)

Suppose first that we are in the indefinite case. As in Section~3 of
\cite{breuildiamond} there is a smooth projective algebraic curve
$X_K$ over $F$ associated to $K$. Then there is a local system $\cF_{\sigma(\theta)^*}$
on $X_K$ corresponding to $\sigma(\theta)^*$ in the usual way, and we
define \[S(\sigma):=H^1((X_K)_{/\overline{F}},\cF_{\sigma(\theta)^*}).\]

Suppose now that we are in the definite case. Then we let $S(\sigma)$
be the space of continuous functions \[f:D^\times\backslash
(D\otimes_F\A_F^\infty)^\times\to \sigma(\theta)^*\]such that we have
$f(gu)=u^{-1}f(g)$ for all $g\in (D\otimes_F\A_F)^\times$, $u\in K(\A_F^\infty)^\times$.

 We let $\T^{S,\univ}$ be
the commutative polynomial algebra over~$\cO$ generated by formal
variables $T_w,S_w$ for each finite place $w\notin S\cup\{w_1\}$ of
$F$.
For each finite place $w$ of $F$, fix a uniformiser $\varpi_w$ of
$\cO_{F_w}$. Then $\T^{S,\univ}$ acts on $S(\sigma)$ in the usual
way, with the corresponding double cosets
being \[T_w=\left[\GL_2(\cO_{F_w})
\begin{pmatrix}
  \varpi_w& 0\\0&1
\end{pmatrix}\GL_2(\cO_{F_w})\right],\]  \[S_w=\left[\GL_2(\cO_{F_w})
\begin{pmatrix}
  \varpi_w& 0\\0&\varpi_w
\end{pmatrix}\GL_2(\cO_{F_w})\right].\]

Let $\m$ be the maximal ideal of
$\T^{S,\univ}$ with residue field $\F$ with the property that for each
finite place $w\notin S\cup\{w_1\}$ of $F$, the characteristic
polynomial of $\rhobar(\Frob_w)$ is equal to the image of
$X^2-T_wX+(\mathbf{N}w)S_w$ in $\F[X]$. Write $\T(\sigma)$ for the image of
$\T^{S,\univ}$ in $\End_\cO(S(\sigma))$.
The assumption that $\rhobar$
is modular means that we can (and do) assume that we have chosen $S$ and
$\psi$ in such a way that for some $\sigma$, we have
$S(\sigma)_\m\ne 0$. In both the definite and indefinite cases the functor $\sigma\mapsto S(\sigma)_\m$
is exact, because $K$ is sufficiently small and $\m$ is non-Eisenstein.

We will use our characters $\theta_w$ to twist the universal liftings
we consider. The Taylor--Wiles--Kisin method applied to the spaces
$S(\sigma)_\m$ will naturally produce modules over
$R^{\psi}=\widehat{\otimes}_{\cO}R^{\square,\psi|_{G_{F_w}}}_w$,
rather than over
$R^{\psi_\sigma}=\widehat{\otimes}_{\cO}R^{\square,\psi_\sigma}_w$. However,
twisting the universal lifting to $R^{\square,\psi|_{G_{F_w}}}_w$ by
$\theta_w^{-1}\circ\det$ gives a canonical isomorphism
$R^{\square,\psi_\sigma}_w\isoto R^{\square,\psi|_{G_{F_w}}}_w$ for
each $w$, and thus a canonical isomorphism $R^{\psi_\sigma}\isoto
R^{\psi}$, so from now on we can and do think of any $R^{\psi}$-module
as an $R^{\psi_\sigma}$-module.

Let $R_S^\univ$ be the universal deformation ring for deformations of
$\rhobar$ which are unramified outside of $S$, and let
$\rho^\univ:G_{F,S}\to\GL_2(R_S^\univ)$ denote a choice of the
universal deformation.
There is a
natural homomorphism $R_S^\univ\to\T(\sigma)_\m$ which for each finite
place $w\notin S\cup\{w_1\}$ sends the characteristic polynomial of
$\rho^\univ(\Frob_w)$ to $X^2-T_wX+(\mathbf{N}w)S_w$. We write
$\rho(\sigma)_\m$ for the composite
$G_{F,S}
\buildrel \rho^{\univ}\over \longrightarrow
\GL_2(R_S^\univ)
\to
\GL_2(\T(\sigma)_\m)$.

In the definite case, write $M(\sigma)$ for $S(\sigma)^*_\m$. In the
indefinite case, we ``factor out'' the action of $G_{F,S}$ by
defining \[M(\sigma):=(\Hom_{\T(\sigma)_\m[G_{F,S}]}(\rho(\sigma)_\m,S(\sigma)_\m))^*.\]
By a standard argument, the evaluation
map \[M(\sigma)^*\otimes_{\T(\sigma)_\m}\rho(\sigma)_\m\to
S(\sigma)_\m\] is an isomorphism of
$\T(\sigma)_\m[G_{F,S}]$-modules. (See, for example, Th\'eor\`eme 4 of
\cite{MR1279611} or Proposition 5.5.3 of
\cite{emerton2010local}.)

\subsection{Infinite level}\label{subsec: limit over p power level}With an eye ahead
to Section~\ref{sec: proof of the conjecture}, we also note a related
construction, of spaces of modular forms of infinite $p$-power
level. (Note that we will not patch these spaces via the Taylor--Wiles
method, although it is possible to do so, \emph{cf.}\ \cite{EGPatching}.) Let $v|p$ be a fixed place of $F$, and write the group
$K_S$ above as $K_S=K_vK^v$, with $K^v=\prod_{w \in S \setminus
  \{v\}}K_w$. Fix a continuous finitely generated representation
$\sigma^v$ of $K^v$, and let $\sigma$ be the representation of $K_S$
given by twisting the action of $K^v$ on $\sigma^v$ by the character
$\psi\circ\Art_{F_v}\circ\det$ of $K_v$. We assume that
$\sigma\in\cC_Z$, and write $S(K_v,\sigma^v)_\m$ for $S(\sigma)_\m$.

Then we define \[S^v(\sigma^v):=\varinjlim_{K_v}
S(K_v,\sigma^v),\] where the direct limit is taken over the directed
system of compact open subgroups of $\GL_2(\cO_{F_v})$.
We let $\T(\sigma^v)$ denote the image of
$\T^{S,\univ}$ in $\End_\cO(S^v(\sigma^v))$, and we write
$\rho(\sigma^v)_\m$ for the composite $G_{F,S}\buildrel \rho^{\univ} \over \longrightarrow
\GL_2(R_S^\univ)\to\GL_2(\T(\sigma^v)_\m)$.  If we are in the definite case then we set
$M^v(\sigma^v):=S^v(\sigma^v)^*_\m$, and if we are in the indefinite
case we
set \[M^v(\sigma^v):=(\Hom_{\T(\sigma^v)_\m[G_{F,S}]}(\rho(\sigma^v)_\m,S^v(\sigma^v)_\m))^*.\]
Again, we have an isomorphism \[M^v(\sigma^v)^*\otimes_{\T(\sigma^v)_\m}\rho(\sigma^v)_\m\to
S^v(\sigma^v)_\m.\] Since $\m$ is non-Eisenstein, in either the define
or indefinite cases if we fix some
$K_v$ and a representation $\sigma_v$ of $K_v$ such that
$\sigma:=\sigma_v\otimes_\cO\sigma^v$ is an element of $\cC_Z$, then
we have a natural isomorphism
\numequation\label{eqn:U comparison patching section}
M(\sigma)\isoto\Hom_{K_v}(M^v(\sigma^v),\sigma_v^*)^*.
\end{equation}

\subsection{Functorial Taylor--Wiles--Kisin patching}\label{sec: abstract patching}

Since there are only countably many isomorphism classes of finite
length $\cO$-modules with a smooth action of $K_S$, applying the
usual Taylor--Wiles--Kisin patching argument to the $M(\sigma)$ as
in Section~5 of \cite{geekisin} gives us the following data.
\begin{itemize}
\item positive integers $g,q$,
\item $S_\infty=\cO\llbracket x_1,\dots,x_{4\#S+q-1}\rrbracket$, a power series ring in $4\#S+q-1$ variables over $\cO$,
\item $R_\infty$, a power series ring in $g$ variables over $R$,
\item an $\cO$-algebra homomorphism $S_\infty\to R_\infty$,
\item a covariant exact functor $\sigma\mapsto M_\infty(\sigma)$ from
  the subcategory of finite length objects of $\cC_Z$  to the category of $R_\infty$-modules.
\end{itemize} We extend $M_\infty$ to the whole of $\cC_Z$ by defining $M_\infty(\sigma)=\varprojlim_nM_\infty(\sigma/\varpi_E^n\sigma)$.
This data satisfies the following further properties.
\begin{itemize}
\item There is an isomorphism $M_\infty(\sigma)/(x_1,\dots,x_{4\#S+q-1})\cong  M(\sigma)$,
  and a compatible isomorphism
  $R_\infty/(x_1,\dots,x_{4\#S+q-1})\cong R_S^\univ$.
\item The action of $R$ on $M_\infty(\sigma)$ factors through $R^{\psi_\sigma}$.
\item If $R^{\psi,\tau}_\infty=R_\infty\otimes_R R^{\psi,\tau}$, then
  $\dim R^{\psi,\tau}_\infty=\dim S_\infty$.

\item If $\sigma$ is a finite free $\cO$-module, then
  $M_\infty(\sigma)$ is a finite free $S_\infty$-module (where the
  $S_\infty$-module structure comes from the homomorphism $S_\infty\to R_\infty$).
\item If $\sigma$ is a finite free $\F$-module, then
  $M_\infty(\sigma)$ is a finite free $S_\infty\otimes_\cO\F$-module.
\end{itemize}

Regard the $R_\infty$-module $M_\infty(\sigma)$ as a coherent sheaf on
$X_\infty=\Spf R_\infty$. We now verify the hypotheses of Definition
\ref{defn: fixed det patching functor}; the first hypothesis is clear
from the construction.

In order to check the hypothesis on $M_\infty(\sigma^\circ(\tau))$, we
need the action of $R$ on $M_\infty(\sigma^\circ(\tau))$ to factor through
$R^{\psi,\tau}$, and we need $M_\infty(\sigma^\circ(\tau))$ to be maximal
Cohen--Macaulay over $R^{\psi,\tau}_\infty$. The first property is an
immediate consequence of local-global compatibility and the
construction of $M_\infty$, and the claim that
$M_\infty(\sigma^\circ(\tau))$ is maximal Cohen--Macaulay over
$R^{\psi,\tau}_\infty$ is immediate from the Auslander--Buchsbaum
formula and the fact that $M_\infty(\sigma^\circ(\tau))$ is free over the
regular local ring $S_\infty$.

It remains to check the hypothesis on $M_\infty(\sigmabar)$. It is
enough to check that the action of $R$ on $M_\infty(\sigmabar)$
factors through $R^{\psi,\sigmabar}$, as the claim that it is maximal
Cohen--Macaulay then follows exactly as in the previous
paragraph. This is immediate from  Corollaries 5.6.4 and 4.5.7 of
\cite{geekisin} and the construction of $M_\infty$.

\subsection{Minimal level}\label{subsec:minimal level}Under some
additional hypotheses (which will be satisfied in the cases of
ultimate interest to us in this paper), we now refine the
constructions of the previous section to produce a minimal fixed
determinant patching functor with unramified coefficients. We will use
this functor in Section~\ref{sec: freeness and
  multiplicity one}.

We continue to use the notation and assumptions introduced in the
previous section, so that in particular $p$ is unramified in $F$. For
the remainder of this section, we make the following additional
assumptions:
\begin{itemize}
\item $p>3$, 
 \item $\rhobar|_{G_{F_w}}$ is generic for all places $w|p$, and
\item if $w\in\Sigma$, then $\rhobar|_{G_{F_w}}$ is not scalar.
\end{itemize} 
Since $p>3$ is
unramified in $F$, it follows from the proof of Lemma 4.11 of
\cite{MR1262939} 
that we can choose the place $w_1$ so that $\rhobar(\Frob_{w_1})$ has distinct
eigenvalues, and we choose $w_1$ in this manner. If necessary, we replace $\F$
with a quadratic extension so that these eigenvalues are in $\F$. Following
\cite{breuildiamond}, we now explain a slight refinement of the above
constructions in the case of minimal level.

Fix a place $v|p$. Set $E=W(\F)[1/p]$, let $\psi$ be the Teichm\"uller lift of
$\varepsilonbar\det\rhobar$, and let $S$ be the union of $\Sigma$,
the places dividing $p$, and the places where $\rhobar$ is
ramified. We will ultimately construct a minimal patching functor for
which the corresponding index set is just the place $v$, but we will
begin by constructing a patching functor for which the index set is
the set of places $S$.

Following Section~3.3 of \cite{breuildiamond} very closely (which
defines, under slightly different hypotheses, the corresponding spaces
in characteristic $p$), we begin by defining the spaces of modular
forms that we will use. Set $K_v=\GL_2(\cO_{F_v})$. For each place
$w\in S\setminus \{v\}$, we will now define a compact open subgroup
$K_w$ of $(\cO_D)^\times_w$ and a continuous representation of $K_w$
on a finite free $\cO$-module $L_w$. Let $S'$ be the union of the set
of places $w|p,w\ne v$ for which $\rhobar|_{G_{F_w}}$ is reducible and
the set of places $w$ for which $w\in\Sigma$ and $\rhobar|_{G_{F_w}}$
is reducible, and the place $w_1$; then at the places $w\in S'$, we
will also define Hecke operators $T_w$ and scalars
$\beta_w\in\F^\times$.
  \begin{itemize}
  \item If $w|p,w\ne v$, and $\rhobar|_{G_{F_w}}$ is irreducible, then
    we let $K_w=\GL_2(\cO_{F_w})$, we choose a Serre weight
    $\sigmabar_w\in\cD(\rhobar|_{G_{F_w}})$, we let $\tau_w$ be the
    corresponding inertial type from Proposition \ref{prop: types for
      elimination}, and we let $L_w$ be a $\GL_2(\cO_{F_w})$-stable
    $\cO$-lattice in $\sigma(\tau_w)$.
    \item If $w|p,w\ne v$, and $\rhobar|_{G_{F_w}}$ is reducible, then
      the assumption that $\rhobar|_{G_{F_w}}$ is generic implies in
      particular that it is not scalar, so we are in the situation
      considered in Section~3.3 of \cite{breuildiamond}. Write $\rhobar|_{G_{F_w}}\cong
      \begin{pmatrix}
        \xibar_w&*\\0&\xibar_w'\omega^{-1}
      \end{pmatrix}$. Note that the assumption that
      $\rhobar|_{G_{F_w}}$ is generic also implies that
      $\xibar_w|_{I_{F_w}}\ne\xibar_w'|_{I_{F_w}}$.

Let $K_w$ be the Iwahori subgroup of $\GL_2(\cO_{F_w})$ consisting of
matrices which are upper-triangular modulo $w$, and let
$\gammabar_w:K_w\to\F^\times$ be the character which sends a matrix
congruent to $
\begin{pmatrix}
  a&b\\0&d
\end{pmatrix}$ modulo $w$ to
$(\xibar\circ\Art_{F_w})(a)(\xibar'\circ\Art_{F_w})(d)$. Let
$\gamma_w$ be the Teichm\"uller lift of $\gammabar_w$, and let $L_w=\cO(\gamma_w)$.

Choose a uniformiser $\varpi_w$ of $\cO_{F_w}$, let
$V_w:=\ker(\gamma_w)$, and let $T_w$ be the Hecke operator defined by
the double coset $V_w
\begin{pmatrix}
  \varpi_w&0\\0&1
\end{pmatrix}V_w$. Let $\beta_w=(\xibar_w\circ\Art_{F_w})(\varpi_w)$.

\item If  $w\in \Sigma\cap S'$ (equivalently: if $w\in\Sigma$ and
  $\rhobar|_{G_{F_w}}$ is reducible), then the assumption that some
  $S(\sigma)_\m\ne 0$
  implies that
  $\rhobar|_{G_{F_w}}$ is a twist of an extension of the trivial
  character by the mod $p$ cyclotomic character; furthermore, if
  $\mathbf{N}w\equiv 1\pmod{p}$, then the assumption that
  $\rhobar|_{G_{F_w}}$ is not scalar implies that this extension is
  non-trivial. We let $K_w=(\cO_D)^\times_w$, and we let $L_w$ be the
  rank one $\cO$-module on which $(\cO_D)^\times_w$ acts via
  $\gamma_w:=\psi|_{I_{F_w}}\circ\Art_{F_w}\circ\det$. We choose a
  uniformiser $\Pi_w$ of $(\cO_D)_w$, we let $V_w:=\ker(\gamma_w)$, we
  let $T_w$ be the Hecke operator defined by the double coset
  $V_w\Pi_wV_w$, and we set
  $\beta_w:=(\gammabar_w\circ\Art_{F_w})(\det\Pi_w)$.
  \item If $w\notin\Sigma$, $w\nmid p$ and $\rhobar|_{G_{F_w}}$ is
    reducible, we choose a character $\xibar_w:G_{F_w}\to\F^\times$
    such that $\rhobar'_w:=\xibar_w^{-1}\rhobar|_{G_{F_w}}$ has minimal
    conductor among the twists of $\rhobar|_{G_{F_w}}$ by
    characters. Write $\xi_w$ for the Teichm\"uller lift of
    $\xibar_w$. We let $n_w$ be the (exponent of the) conductor of $\rhobar'_w$, we let
    $\mu_w$ be the Teichm\"uller lift of $\det\rhobar'_w$, and we
    let $K_w$ be the subgroup of $\GL_2(\cO_{F_w})$ consisting of
    matrices which are upper-triangular modulo $w^{n_w}$.

    We let $L_w$ be the rank one $\cO$-module on which $K_w$ acts via
    a character $\gamma_w$, which is defined as follows. If $n_w=0$,
    then $\gamma_w=\xi_w\circ\Art_{F_w}\circ\det$. If $n_w>1$, we let
    $\gamma_w(g)=(\xi_w\circ\Art_{F_w})(\det(g))(\mu_w\circ\Art_{F_w})(d)$,
    where $g\equiv
    \begin{pmatrix}
      a&b\\0&d
    \end{pmatrix}\pmod{w^{n_w}}$.

  \item If $w\nmid p$ and $\rhobar|_{G_{F_w}}$ is irreducible, then we
    let $\tau_w$ be any inertial type for which $\rhobar|_{G_{F_w}}$
    has a lift of type $\tau_w$ and determinant
    $\psi|_{G_{F_w}}\varepsilon^{-1}$, and we let $L_w$ be an
    $(\cO_D)^\times_w$-invariant lattice in $\sigma(\tau_w)$
    (respectively $\sigma_D(\tau_w)$ if $w\in\Sigma$). (It is easy to
    check that we can choose $\tau_w$ and thus $L_w$ to be defined
    over $\cO$, either by a consideration of the rationality
    properties of the local Langlands correspondence, or by explicitly
    choosing $\tau_w$ via the classification of the possible
    $\rhobar|_{G_{F_w}}$.)

  \item If $w=w_1$, we let $K_{w_1}$ be the subgroup of $\GL_2(\cO_{F_{w_1}})$
    consisting of those matrices which are upper-triangular modulo $w_1$. We
    also define a Hecke operator $T_{w_1}$ by
\[T_{w_1}:=\left[K_{w_1} \begin{pmatrix}
  \varpi_{w_1}& 0\\0&1
\end{pmatrix}K_{w_1}\right].\] Note that $T_{w_1}$ depends on the
choice of $\varpi_{w_1}$. Fix from now on a choice $\beta_{w_1}$ of
eigenvalue of $\rhobar(\Frob_{w_1})$, where
$\Frob_{w_1}=\Art_{F_{w_1}}(\varpi_{w_1})$.
  \end{itemize}

Write $L:=\otimes_{w\in S, w\ne v,\cO}L_w$. Let $\cC^v_Z$ be the
category of finitely generated $\cO$-modules $\sigma_v$ with a
continuous action of $K_v=\GL_2(\cO_{F_v})$, having the property that
the central character of $\sigma_v$ lifts the character
$(\varepsilonbar\det\rhobar|_{I_{F_v}})\circ\Art_{F_v}$. Let $\cC'_Z$ be the
subcategory of $\cC_Z$ consisting of modules of the form
$L\otimes_{\cO}\sigma_v$, where $\sigma_v$ is an object of $\cC^v_Z$.

For each object $\sigma$ of $\cC'_Z$, we define $S(\sigma)$ as in
Section~\ref{sec: abstract patching}. We let
$\T(\sigma)'=\T(\sigma)[T_w]_{w\in S'\cup\{w_1\}}$, and we denote the
maximal ideal of $\T(\sigma)'$ generated by $\m$ and the $T_w-\beta_w$
for $w\in S'$ by $\m'$. By Lemme 3.3.1 of \cite{breuildiamond}, there
is a natural action of $\T(\sigma)'$ on $S(\sigma)$; we define
$S^{\min}(\sigma)_\m:=S(\sigma)_{\m'}$, and construct
$M^{\min}(\sigma)$ from $S^{\min}(\sigma)_\m$ in the same way that we
constructed $M(\sigma)$ from $S(\sigma)_\m$.

For each place $w\in S$, set $\alpha_w=\psi(\Frob_w)$.  Then the
patching arguments of Section~5 of \cite{geekisin} go through as in
Section~\ref{sec: abstract patching} to produce a fixed determinant
patching functor $M^{\min}_\infty$ defined on $\cC'_Z$, such that
\[M^{\min}_\infty(\sigma)/(x_1,\dots,x_{4\#S+q-1})\cong
M^{\min}(\sigma).\]

We will now use the obvious functor from $\cC^v_Z$ to $\cC'_Z$ to
construct a minimal fixed determinant patching functor on $\cC^v_Z$
from $M^{\min}_\infty$. For each place $w\in S\setminus \{v\}$ we
define a certain universal lifting ring $R^{\min}_w$ as
follows.
\begin{itemize}
\item If $\rhobar|_{G_{F_w}}$ is irreducible,
  then we let $R^{\min}_w=R_{w}^{\square,\psi|_{G_{F_w}},\tau_w}$ for
  the above choice of $\tau_w$.
  \item If $w\nmid p$, and $\rhobar|_{G_{F_w}}$ is reducible but is
    not a twist of an extension of the trivial character by the mod
    $p$ cyclotomic character, then we let $R^{\min}_w$ be the complete
    local noetherian $\cO$-algebra which prorepresents the functor
    which assigns to an Artinian $\cO$-algebra $A$ the set of lifts of
    $\rhobar|_{G_{F_w}}$ to representations
    $\rho_A:G_{F_w}\to\GL_2(A)$ of determinant
    $\psi|_{G_{F_w}}\varepsilon^{-1}$, for which
    $\rho_A(I_{F_w})\isoto\rhobar(I_{F_w})$.
  \item If $w|p$ and $\rhobar|_{G_{F_w}}$ is reducible, in which case
    we write $\rhobar|_{G_{F_w}}\cong
    \begin{pmatrix}
      \etabar_w&*\\0&\etabar'_w\omega^{-1}
    \end{pmatrix}$ for some characters $\etabar_w,\etabar'_w$, or if  $w\nmid p$ and $\rhobar|_{G_{F_w}}\cong
    \begin{pmatrix}
      \etabar_w&*\\0&\etabar_w\omega^{-1}
    \end{pmatrix}$ for some character $\etabar_w$, then we let
    $R^{\min}_w$ be the complete local noetherian $\cO$-algebra which
    prorepresents the functor which assigns to an Artinian
    $\cO$-algebra $A$ the set of pairs $(\rho_A,L_A)$ as follows.  The
    representation $\rho_A:G_{F_w}\to\GL_2(A)$ is a lift of
    $\rhobar|_{G_{F_w}}$ of determinant $\psi\varepsilon^{-1}$, and
    $L_A$ is a direct factor of $A^2$ (the module on which
    $\rho_A$ acts), such that $G_{F_w}$ acts on $L_A$ by a character
    of the form $\eta_w$, where $\eta_w$ is a lift of
    $\etabar_w$ such that $\eta_w(I_{F_w})\isoto\etabar_w(I_{F_w})$.
\end{itemize}
In each case the ring $R^{\min}_w$ is formally smooth over $\cO$; in
the case that $w|p$ and $\rhobar|_{G_{F_w}}$ is irreducible, this
follows from Theorem \ref{thm:deformation rings principal series} below, and
in the other cases it follows from Lemma 3.4.1 of
\cite{breuildiamond}. Write
$R_{S\setminus\{v\}}^{\min}=\widehat{\otimes}_{w\in S,w\ne
  v,\cO}R^{\min}_w$, and let
$R^{\min,\psi_\sigma}=R_{S\setminus\{v\}}^{\min}\widehat{\otimes}R_v^{\square,\psi_\sigma}$. Then
$R_{S\setminus\{v\}}^{\min}$ is formally smooth over $\cO$, and it follows
exactly as in Proposition 3.5.1 of \cite{breuildiamond} that the
action of $R$ on $M^{\min}_\infty(\sigma)$ factors through
$R^{\min,\psi_\sigma}$.

Now, since $R_{S\setminus\{v\}}^{\min}$ is formally smooth over $\cO$,
the ring $R_\infty\otimes_{\cO}R^{\min,\psi_\sigma}$ is formally
smooth over $R_v^{\square,\psi_\sigma}$; so the functor which assigns
$M_\infty^{\min}(\sigma_v\otimes_{\cO} L)$ to each object $\sigma_v$ of
$\cC^v_Z$ is a fixed determinant patching functor with unramified coefficients.

It only remains to check that $M_\infty^{\min}$ is minimal. In
order to see this, it suffices to check that if $\tau_v$ is an
inertial type for $I_{F_v}$ with
$\det\taubar_v=\varepsilonbar\det\rhobar|_{I_{F_v}}$, then
$M(\sigma^\circ(\tau_v)\otimes_{\cO}L)[1/p]$ is locally free of rank one over
$\T(\sigma^\circ(\tau_v)\otimes_{\cO}L)'[1/p]$; but this follows from the definition of $L$,
\emph{cf.}\ Proposition 3.5.1 of \cite{breuildiamond}.

\subsection{Unitary groups}\label{subsec: unitary patching
  functors}We now give a concrete construction of a patching functor
(without fixed determinant), by applying the Taylor--Wiles--Kisin
method to forms of $U(2)$. The construction is extremely similar to
the arguments above with quaternion algebras, and we content ourselves
with sketching the details, following Section~4 of~\cite{geekisin}.

Let $F$ be an imaginary CM field with maximal totally real subfield
$F^+$, and suppose that the prime $p$ is odd and is unramified in $F$,
and that all places of $F^+$ lying over $p$ split in $F$. Assume
further that $\zeta_p\notin F$, that $F/F^+$ is unramified at all
finite places, and that $[F^+:\Q]$ is even. Then there is an algebraic
group $G_{/\cO_{F^+}}$ as in Section 3.1 of~\cite{geekisin}.  This
group  is
a quasisplit form of $U(2)$, and is compact mod centre at the infinite
places. If $v$ is a finite place of $F^+$ which splits at $ww^c$ in
$F$, then there is an isomorphism
$\imath_w:G(\cO_{F^+_v})\isoto\GL_2(\cO_{F_w})$ which extends to an
isomorphism $G(F^+_v)\isoto\GL_2(F_w)$. 

Fix an absolutely irreducible representation $\rbar:G_F\to\GL_2(\F)$
with $\rbar^c\cong\rbar^\vee\varepsilonbar^{-1}$. We assume that $\rbar$
is automorphic in the sense of Section 3.2 of~\cite{geekisin}, that
$\rbar(G_{F(\zeta_p)})$ is adequate in the sense of~\cite{jack}, and
that if $w$ is a finite place of $F$ for which $\rbar|_{G_{F_w}}$ is
ramified, then $w|_{F^+}$ splits in $F$.

Let $S$ be a set of finite places of $F^+$ which split in $F$, and
assume that $S$ contains the places dividing $p$ and the places at
which $\rhobar$ is ramified. For each place $v\in S$, choose a place
$\tv$ of $F$ lying over $v$, and let $\tS$ denote the set of these
places. The local fields $L_i$ in the definition of our patching
functor will be the $F_\tv$ for $v\in S$, the representations will be the
$\rhobar|_{G_{F_\tv}}$, and in this section for consistency
with~\cite{geekisin} we will write the various compact groups that occur as
$U$ rather than $K$, and in particular we will write $U_S$ for the
group $K$ in the definition of a patching
functor.

By Lemma 4.11 of \cite{MR1605752} we can and do choose a finite place
$v_1\notin S$ which splits as $w_1w_1^c$ in $F$, with the properties that
\begin{itemize}
\item $\mathbf{N}w_1\not\equiv
1\pmod{p}$,
\item  the ratio of the eigenvalues of $\rbar(\Frob_{w_1})$
  is not equal to $(\mathbf{N}w_1)^{\pm 1}$, and
\item the residue characteristic of $w_1$ is
sufficiently large that for any non-trivial root of unity $\zeta$ in a
quadratic extension of $F$, $w_1$ does not divide $\zeta+\zeta^{-1}-2$.
\end{itemize}
We consider compact open subgroups $U=\prod_vU_v$ of
$G(\A_{F^+}^\infty)$ with the properties that \begin{itemize}
\item $U_v\subset G(\bigO_{F^+_v})$ for all $v$ which split in $F$;
  \item $U_v$ is a hyperspecial maximal compact subgroup of $G(F_v^+)$
    if $v$ is inert in~$F$;
   \item $U_v=G(\bigO_{F^+_v})$ if $v|p$ or $v\notin S\cup\{v_1\}$;
   \item $U_{v_1}$ is the preimage of the upper triangular unipotent matrices under \[ G(\cO_{F^+_{v_1}}) \buildrel {\iota_{w_1} \atop \sim}
\over \longrightarrow \GL_2(\cO_{F_{w_1}}) \to \GL_2(k_{v_1}).\]
\end{itemize}
In particular, the assumptions on $v_1$ and $U_{v_1}$ imply that $U$
is sufficiently small. Set $U_S=\prod_{v|p}U_v$. Then for any
$\sigma\in \cC$, we have a space of algebraic modular forms
$S(U,\sigma^*)^*$ defined as in Section 3.1 of~\cite{geekisin}. We
have a Hecke algebra $\T^{S\cup\{v_1\},\univ}$ defined as in Section
3.2 of~\cite{geekisin}, with a maximal ideal $\m$ corresponding to
$\rbar$, and we set $M(\sigma):=S(U,\sigma^*)_{\m}^*$.

Then in the same way as above, the Taylor--Wiles--Kisin patching
method (as explained in Section
4.3 of~\cite{geekisin}) allows us to patch the $M(\sigma)$ to obtain a
covariant exact functor $\sigma\to M_\infty(\sigma)$ from the category
of finite length objects of $\cC$ to the category of $R_\infty$-modules, where $R_\infty$ is a power series ring over $R$. We extend
this to a functor on all of $\cC$ by setting
$M_\infty(\sigma)=\varprojlim_n M_\infty(\sigma/\varpi_E^n\sigma)$,
and regard $M_\infty(\sigma)$ as a coherent sheaf on $X_\infty=\Spf
R_\infty$. 

As in Section~\ref{sec: abstract patching}, it is easy to verify that
this is a patching functor; in particular, the hypothesis on
$M_\infty(\sigma^\circ(\tau))$ follows from local-global compatibility
in the same way as in Section~\ref{sec: abstract patching}, as does that on $M_\infty(\sigmabar)$ (\emph{cf.}\ the proof
of Lemma 4.1.3 of~\cite{blggU2}).
\section{Tame deformation spaces}\label{sec: deformation
  spaces}

The goal of this section is to describe the spaces 
of tamely potentially Barsotti--Tate lifts of a given generic
Galois representation
$\rhobar: G_{F_v} \to \GL_2(\F)$,
where a lift $\rho$ of $\rhobar$
is called {\em tamely potentially Barsotti--Tate} if $\rhobar$
becomes Barsotti--Tate (that is, crystalline of Hodge type $0$) over a tame extension
of $F_v$.  For such a lift, the inertial type $\tau$
underlying the potentially crystalline Dieudonn\'e module
of $\rho$ will correspond via inertial local Langlands
to a tame type $\sigma(\tau)$, i.e.\ an
irreducible representation of $\GL_2(\F_q)$, and we restrict our 
attention to the case when $\sigma(\tau)$ is principal series or cuspidal.
(Thus we omit the case when $\sigma(\tau)$ is one-dimensional, which is the case
when a tame twist of $\rho$ is actually Barsotti--Tate over $F_v$
itself.) 
For each
such~$\tau$, our goal then is to describe the framed deformation space
$X\bigl(\tau\bigr)$ 
of potentially Barsotti--Tate liftings of type $\tau$.

In the case when $\tau$ is a principal series and $\rhobar$ has only
scalar automorphisms, a precise calculation of $X\bigl(\tau\bigr)$ is given
in \cite[Thm.~5.2.1]{BreuilMezardRaffinee}.  We will extend this
calculation to the case of arbitrary generic $\rhobar$,
and also to the case of cuspidal~$\tau$. 
Furthermore, we will 
give an identification of the various
components of the special fibre $\overline{X}\bigl(\tau\bigr)$ with the mod $p$
reductions of crystalline deformation rings
(thus providing a concrete geometric interpretation of the matching
between components and Serre weights of $\rhobar$ established in
\cite[Corollaire 5.2.2]{BreuilMezardRaffinee}).

Our proof will be in several stages. 
We will make use of
some of the geometric Breuil--M\'ezard results of \cite{emertongeerefinedBM}
(encapsulated in Theorem~\ref{thm:geometric BM for types} below), 
and we will use base-change arguments to deduce our result for 
cuspidal types from the corresponding result for principal series types
with a minimum of additional calculation.
A key ingredient of our computation (as in the computation of
\cite{BreuilMezardRaffinee}) is the fact that a potentially Barsotti--Tate
lift of $\rhobar$ arises from a uniquely determined strongly 
divisible module, and in fact we will follow~\cite{MR2137952}
in working with deformations of strongly divisible modules
rather than directly with deformations of Galois representations.
Thus, in addition to considering framed deformations (i.e.\ liftings)
on the Galois side, we also consider framings of strongly
divisible modules; it is not difficult to transfer information
obtained in terms of framings of strongly divisible modules
back to the usual lifting spaces of Galois representations.

\subsection{Deformation spaces of Galois representations}
We fix a generic continuous representation $\rhobar:G_{F_v} \to
\GL_2(\F)$ and a non-scalar tame inertial type $\tau$. As in Section~\ref{subsec:
  abstract patching functors}, we let $X\bigl(\tau\bigr)$ denote the
deformation space parameterising liftings of~$\rhobar$ that are
potentially Barsotti--Tate of type $\tau$. Let
$R^\tau$ be the universal lifting ring for such deformations, so that
$X\bigl(\tau\bigr)$ is the formal spectrum of $R^\tau$. Fix a
character $\psi:G_{F_v}\to\cO^\times$ lifting
$\varepsilonbar\det\rhobar$, with $\psi|_{I_{F_v}}=\det\tau$, and
consider the subspace $X^\psi\bigl(\tau\bigr)=\Spf R^{\psi,\tau}$
corresponding to lifts with determinant $\psi\varepsilon^{-1}$. For each Serre weight
$\sigmabar$ we also have the spaces $X\bigl(\sigmabar\bigr)$ and
$X^\psi\bigl(\sigmabar\bigr)$ for crystalline lifts of Hodge type
$\sigmabar$.

\begin{theorem}
  \label{thm:geometric BM for types}
Assume $p > 2$.  Let $\rhobar$ be generic, and let
  $\tau$ be a non-scalar tame inertial type. Then the mod~$\varpi_E$ fibre of the
  deformation space $X\bigl(\tau\bigr)$ is the union of the mod $\varpi_E$
  fibres $\overline{X}\bigl(\sigmabar\bigr)$ where $\sigmabar$ runs over the
  Jordan--H\"older factors of $\sigmabar(\tau)$. Furthermore,
  $\overline{X}\bigl(\sigmabar\bigr)$ is non-empty if and only if
  $\sigmabar\in\cD(\rhobar)$.

The analogous statements also hold for $X^\psi\bigl(\tau\bigr)$.
\end{theorem}
\begin{proof}
  The first part is a special case of
  \cite[Thm.\ 5.5.4]{emertongeerefinedBM} (note that the assumption
  that all predicted weights of $\rhobar$ are regular follows from
  Lemma~\ref{lem: generic representations have no weight p-1}, and the
  multiplicities $n_a$ occurring in~\cite{emertongeerefinedBM} are all
  zero or one by Lemma~\ref{lem:weights in tame types have
    multiplicity one}). The analogous results for
  $X^\psi\bigl(\tau\bigr)$ follow from those for $X\bigl(\tau\bigr)$,
  because twisting by the universal unramified deformation of the trivial character gives an
  isomorphism $R^\tau\cong R^{\psi,\tau}\llbracket X\rrbracket$, by Lemma 4.3.1 of
  \cite{emertongeerefinedBM}.
\end{proof}
The rest of this section is devoted to proving a refinement
of this
result (Theorem~\ref{thm:deformation rings principal series} below)  
which gives explicit equations for the deformation rings
involved; Theorem~\ref{thm:geometric BM for types} will be a key
ingredient in the proof of this refinement. Before we state this
refinement, we need to introduce some terminology and notation.

\subsection{The structure of potentially Barsotti--Tate deformation spaces}
\label{sec:diedonne-modules}
Let $L$ be an unramified extension of
$F_v$, of degree $f'$ over $\Qp$.
Suppose that $L$ embeds into the coefficient field 
$E$. (In applications, $L$
will be either $F_v$ or its quadratic unramified extension $L_v$, so
this supposition will hold; recall that our $E$ is always assumed to
contain an embedding of~$L_v$.)   Fix a root $\uni$ of the polynomial
$E(u) = u^{p^{f'}-1} + p$.

A
\emph{Dieudonn\'e $\cO$-module} is a free
$\cO_L\otimes_{\Zp}\cO$-module of finite rank $M$ together with
an injective endomorphism $\varphi:M\to M$ which is $\cO$-linear and
$\cO_L$-semilinear, and which satisfies $pM\subset\varphi(M)$.  We say that $M$ is a \emph{Dieudonn\'e
  $\cO$-module with descent data} if it is equipped with an
$\cO_L\otimes_{\Zp}\cO$-linear action of
$\Gal(L(\uni)/L)$ which commutes with $\varphi$.

Given a $p$-divisible group $\calH$ over $\cO_L[\uni]$ with
an action of $\cO$ and descent data on the generic fibre to $L$, the
contravariant Dieudonn\'e module corresponding to $\calH \times_{\cO} \F$
is a Dieudonn\'e
$\cO$-module with descent data. In particular, an $\cO$-point of
$X\bigl(\tau\bigr)$ gives a Dieudonn\'e $\cO$-module with
descent data of type $\tau$.

The usual isomorphism
$\cO_L\otimes_{\Zp}\cO \isoto\cO \times\cdots\times\cO $ induces a
decomposition $M=M^0\times\cdots\times M^{f'-1}$ such that
$\varphi(M^i)\subseteq M^{i+1}$.  (We remark that we
  have used the notation of \cite{breuillatticconj} instead of the
  notation of   \cite{BreuilMezardRaffinee}, so that the index
  set $\cS$ is a set of integers rather than the set of embeddings of
  $k_v$ into $\F$; for a Dieudonn\'e $\cO$-module $M = M^0 \times
  \cdots \times M^{f'-1}$, the piece
  that we denote $M^j$ here corresponds to the piece that would be denoted $M^{\emb_j}$
  in \cite{BreuilMezardRaffinee}.)

 Since
$\Gal(L(\uni)/L)$ has order prime to $p$, it follows
that the $\cO$-linear action of $\Gal(L(\uni)/L)$ on each $M^j$ is
via a sum of characters, which is independent of $j$. We refer to this
sum of characters as the \emph{type} of $M$, and indeed in the case 
that $M$ arises from a $p$-divisible group $\calH$
over $\cO_L[\pi]$ with descent data on the generic fibre, it
has type~$\tau$ if and only if the potentially Barsotti--Tate
representation associated to $\calH$ has type~$\tau$ (regarded as a
representation of $I_{F_v}$).

Let $\iota$ denote the involution $\iota(j) = f'-1-j$ on $\{0,\ldots,f'-1\}$.
Given a Dieudonn\'e $\cO$-module~$M$ with descent data of type
$\eta\oplus\eta'$ with $\eta \neq \eta'$, we define the \emph{gauge} (more
specifically, the $\eta$-gauge and the $\eta'$-gauge) of $M$ as
follows: fix a basis $e_j,e'_j$ of $M^j$ on which $\Gal(L(\uni)/L)$
acts by $\eta,\eta'$ respectively 
 (such a basis is
well-defined up to $\cO^\times$-scalars). Then we have
$\varphi(e_j)=x_je_{j+1},\varphi(e'_j)=x'_je'_{j+1}$ for some $x_j,x'_j$ that
are well-defined up to $\cO^{\times}$, and
we say that the $j$-part of the $\eta$-gauge (respectively the
$j$-part of the $\eta'$-gauge) of $M$ is $x_{\iota(j)} \cO^{\times}$ (respectively
$x'_{\iota(j)} \cO^{\times}$);
\textit{cf}.~\cite[Eq.~(13)]{breuillatticconj}.
We will also
refer to $\{x_{\iota(j)} \cO^{\times} \}_j$ (respectively
$\{x'_{\iota(j)} \cO^{\times}\}_j$) as
the
$\eta$-gauge of $M$  (respectively, the $\eta'$-gauge).

For the rest of this section we will freely identify characters of
$\cO_L^\times$ and $I_L$ via $\Art_L$. The main theorem of this section is the following.

\begin{theorem}
  \label{thm:deformation rings principal series}
Assume that $p>3$.
  Continue to assume that $\rhobar$ is generic and let $\tau$ be a
  non-scalar tame inertial type. If $\sigma(\tau)$ is a principal series type, then we write
  $\sigma(\tau)=\sigma(\eta\otimes\eta')$, and let $L=F_v$ and
  $f'=f$ in the above discussion. If $\sigma(\tau)=\Theta(\eta)$
  is a
  cuspidal type, then we write $\eta'=\eta^q$, so that
  $\BC(\sigma(\tau))=\sigma(\eta\otimes\eta')$, and we let $L=L_v$ (the
  quadratic unramified extension of $F_v$) and $f'=2f$ in the above
  discussion.  Recall that $\cS = \{0,\ldots,f-1\}$ in either case.
\begin{enumerate}
\item The deformation space $X\bigl(\tau\bigr)$ is non-empty
  if and only if $\sigmabar_J(\tau)\in\cD(\rhobar)$ for at least one
  $J \in\cP_{\tau}$.  If this is the case, then there are subsets
  $\Jmin \subseteq \Jmax \subseteq \cS$ with $\Jmin,\Jmax\in\cP_\tau$ such that
  $\sigmabar_J(\tau)\in\cD(\rhobar)$ if and only if
  $\Jmin \subseteq J\subseteq \Jmax$.  Moreover, all $J \subseteq
  \cS$ with $\Jmin \subseteq J\subseteq \Jmax$ are
  in $\cP_{\tau}$.
\item
Assume that $X\bigl(\tau\bigr)$ is non-empty.  If $\tau$ is cuspidal,
assume further that $\F$ is large enough for 
$X\bigl(\tau\bigr)$ to have a point over the ring of integers of some
finite extension of $E$ with residue field $\F$.
Then the deformation space $X\bigl(\tau\bigr)$ is equal to the
formal spectrum of a power series ring over
$$\cO\llbracket (X_j,Y_j)_{j \in \Jmax\setminus \Jmin}\rrbracket/(X_j Y_j - p)_{j \in
  \Jmax\setminus \Jmin}. $$

\item If $\lambda : R^{\tau} \to \Zpbar$ is a point on $X(\tau)$, then
  for $j \in \Jmax \setminus \Jmin$ the $j$-part of the $\eta'$- and
  $\eta$-gauges of the Dieudonn\'e module of the 
  potentially Barsotti--Tate representation associated to $\lambda$
  are represented by
  $\lambda(X_j)$ and $\lambda(Y_j)$ respectively.

\item
In the context of~{\em (2)},
the mod $\varpi_E$ fibre $\overline{X}\bigl(\tau\bigr)$
is the union of the mod~$\varpi_E$ fibres $\overline{X}\bigl(\sigmabar_J\bigr)$,
for each $J$ such that $\Jmin \subseteq J \subseteq \Jmax$.
Precisely,
the component $\overline{X}\bigl(\sigmabar_J\bigr)$ of $\overline{X}\bigl(\tau\bigr)$
is cut out by the equations $X_j = 0$ {\em (}for $j \in J\setminus \Jmin${\em )} and
$Y_j = 0$ {\em (}for $j \in \Jmax \setminus J${\em )}.
 \item The analogous results also hold for $X^\psi\bigl(\tau\bigr)$.
\end{enumerate}
\end{theorem}

\begin{rem}
  Since twisting by the universal unramified deformation of the trivial
character gives an
  isomorphism $R^\tau\cong R^{\psi,\tau}\llbracket X\rrbracket$, part (5) of
  Theorem~\ref{thm:deformation rings principal series} follows from
  the previous four parts, and we will ignore it in the below.
\end{rem}

\begin{rem}
  \label{rem:sufficiently-large}
  We comment briefly on the extra hypothesis on the residue field $\F$ in part (2) of
  Theorem~\ref{thm:deformation rings principal series}.  In the
  principal series case, the explicit calculations of
  \cite{breuillatticconj} show that our running hypotheses on $\F$ 
  already guarantee the existence of a point on 
  $X\bigl(\tau\bigr)$ as in \ref{thm:deformation rings principal series}(2), and so no additional hypothesis on $\F$ is
  necessary in this case.  We expect that a similar explicit
  calculation in the cusipdal case would eliminate the extra
  hypothesis here. We note that this extra hypothesis is unimportant
  in our applications of Theorem~\ref{thm:deformation rings principal
    series} in Sections~\ref{sec: proof of the
  conjecture} and~\ref{sec: freeness and
  multiplicity one}, where we are free to choose $\F$ to be
arbitrarily large.
\end{rem}

\subsection{Strongly divisible
modules in the principal series case}\label{ss:strongly-divisible} 
We begin the proof of Theorem~\ref{thm:deformation rings principal series} by briefly recalling the notion of a strongly divisible
module with coefficients and tame descent data; for the full details
of this theory, see Sections~3 and~4 of~\cite{MR2137952}. Set
$e=p^{f'}-1$, and let
$S_{\cO_L,\Zp}$ be the $p$-adic completion of $\cO_L[u,u^{ie}/i!]_{i\ge 0}$. Let
$R$ be a complete local noetherian $\cO_L$-algebra, and let $S_{\cO_L,R}$
be the $\m_R$-adic completion of $S_{\cO_L} \otimes_{\Zp} R$.
When $L$ is clear from context we will abbreviate $S_{\cO_L,\Zp}$ and
$S_{\cO_L,R}$ as $S$ and $S_R$, respectively.  Define $\Fil^1 S_R$ to
be the completion of the ideal of $S_R$ generated by the elements $E(u)^i/i!$ for
all $i \ge 1$.

Let $\varphi$
denote the endomorphism of $S_R$ that acts $R$-linearly and
$\cO_L$-Frobenius-semilinearly, with $\varphi(u) = u^p$.   Write $c =
p^{-1} \varphi(E(u))$.  For each 
$g\in \Gal(L(\pi)/L)$ we let $\widehat{g}$ be the linear endomorphism of $S_R$ with $g(u) = (g(\pi)/\pi) u$.

\begin{defn}\label{def:stronglydivisible}  A \emph{strongly divisible module with tame descent
    data from $L(\pi)$ to $L$ and $R$-coefficients}  is a finitely generated free $S_R$-module $\M$, together with
a sub-$S_R$-module $\Fil^{1} \M$, a $\varphi$-semilinear map $\varphi_1 : \Fil^1 \M
\rightarrow \M$, and additive bijections $\ghat : \M \rightarrow
\M$ for each $g \in \Gal(L(\pi)/L)$, satisfying the following
conditions:
\begin{enumerate}
\item $\Fil^{1} \M$ contains $(\Fil^{1}S_R)\M$,

\item $\Fil^{1}\M \cap I\M = I \Fil^{1}\M$ for all ideals $I$
in $R$,

\item $\varphi_1(sx) = \varphi_1(s)\varphi(x)$ for $s \in \Fil^1 S_R$
  and $x \in \M$, with $\varphi(x) := c^{-1} \varphi_1(E(u)x),$

\item $\varphi_1(\Fil^{1}\M)$ is contained in $\M$ and
generates it over $S_R$,

\item $\ghat(sx) = \ghat(s)\ghat(x)$ for all $s \in S_R$, $x \in
\M$, $g \in \Gal(L(\pi)/L)$,

\item $\ghat_1 \circ \ghat_2 = \widehat{g_1 \circ g_2}$ for all
$g_1,g_2 \in \Gal(L(\pi)/L)$,

\item $\ghat(\Fil^{1}\M) \subset \Fil^{1}\M$ for all $g \in
\Gal(L(\pi)/L)$, and

\item $\varphi$ commutes with $\ghat$ for all $g \in
\Gal(L(\pi)/L)$.

\end{enumerate}
\end{defn}

This differs from the definition of strongly divisible modules with
descent data and coefficients in Definition 4.1 of~\cite{MR2137952} in
the following respects.  First, we have set $k=2$ and $F=L(\pi)$, $F'=L$ (in
the notation of~\cite{MR2137952}), and we have not required $R$ to be
flat over $\cO$.  Second, we have equipped each strongly divisible
module with a map $\varphi_1 : \Fil^1 \M \to \M$ rather than a map
$\varphi : \M \to \M$; when $R$ is $p$-torsion-free (as
in~\cite{MR2137952}) these notions are equivalent.  Finally,  we have ignored the monodromy
operator $N$: we can do this because the operator $N$ will exist and be unique for all strongly divisible modules that we
  consider.  Existence follows by the same argument as in
  \cite[Prop~5.1.3(1)]{MR1804530} and \cite[Lem.~2.1.1.9]{MR1804530}
  after noting that $\Fil^1 \M/(\Fil^1 S_R) \M$ is free over $R
  \otimes_{\Zp} \cO_L[\pi]$ for the strongly
divisible modules $\M$ that we consider; uniqueness is a consequence
of the usual argument as in  \cite[Prop~5.1.3(1)]{MR1804530}.

Proposition 4.13 of~\cite{MR2137952} and the remarks immediately
preceding it (together with the remarks in the previous paragraph) show that if $R=\cO$, the category of strongly
divisible modules with tame descent data and $R$-coefficients is
equivalent to the category of $G_L$-stable $\cO$-lattices in
$E$-representations of $G_L$ which become
Barsotti--Tate over $L(\pi)$. 

In particular, if $\calH$ is a $p$-divisible group over $\cO_L[\uni]$
with an action of $\cO$ and descent data on the generic fibre to $L$,
then there is a strongly divisible $\cO$-module with descent data
$\cM$  corresponding to~$\calH$,
and  thus to the
(descended) generic
fibre~$\rho$ of $\calH$.  We recall that $M \cong \cM \otimes_{S} \cO$
where the $\cO$-algebra map $S \to \cO$ sends the variable $u$ and its divided
powers to $0$, and $M$ is the contravariant Dieudonn\'e module
corresponding to $\calH \times_{\cO} \F$ discussed above.
The usual isomorphism $\cO_L\otimes_{\Zp}\cO \isoto\cO
\times\cdots\times\cO $ induces a decomposition
$\cM=\cM^0\times\cdots\times\cM^{f'-1}$, compatible with the
decomposition of $M$ defined above, and since $\Gal(L(\uni)/L)$ has
order prime to $p$, it follows that the $\cO$-linear action of
$\Gal(L(\uni)/L)$ on each $\cM^j$ again is via a sum of
characters, which is independent of~$j$.

The strongly divisible modules that arise in the two-dimensional
principal series case are studied in detail in \cite[\S \S
5--8]{breuillatticconj} and \cite[\S5]{BreuilMezardRaffinee}, and we
now recall some of these results. 
 We note that
  although \cite{breuillatticconj} and \cite{BreuilMezardRaffinee}
  assume that $E$ contains a splitting field of the polynomial
  $u^{p^{2f}-1} + p$, this assumption is never used in their study of
  strongly divisible modules, and so our weaker hypothesis on $E$ is sufficient.
(In particular it is possible to take $E$ to be unramified.) 

Let $\eta\ne\eta'$ be characters of $\Gal(L(\pi)/L)$, and let
$\tau=\eta\oplus\eta'$, a tame inertial principal series type. As
mentioned above, we identify $\eta,\eta'$ with characters of
$\cO_L^\times$, and we use the notation of Section~\ref{ss:principal},
so that in particular we have an integer
$c=\sum_{i=0}^{f-1}c_ip^i$. Write $c^{(j)} := \sum_{i=0}^{f-1} (p-1-c_{i-j})
p^i$.  (Recall that $\omega_f$ agrees with the \emph{reciprocal} of the
composition of $g \mapsto g(\pi)/\pi$ with $\kappabar_0$; this is the
reason for writing $p-1-c_{i-j}$ rather than $c_{i-j}$ in the preceding definition.)
 We say that a strongly divisible module $\cM$ with tame descent
data and $R$-coefficients has \emph{shape} $\tau$ if there is a
decomposition $\cS= I_{\eta}\coprod I_{\eta'}\coprod II$ such that
$\cM$ can be written in the following form.

\begin{eqnarray*} j \in I_{\eta} : & & \begin{cases}
\widetilde{\Fil}^1 \M^j & = \langle e_{\eta}^j + a_j u^{c^{(j)}}
e^j_{\eta'}, (u^e+p) e^j_{\eta'} \rangle \\
\varphi_1( e_{\eta}^j + a_j u^{c^{(j)}} e^j_{\eta'}) & = e_{\eta}^{j+1} \\
\mathrlap{\varphi_1( (u^e+p) e^j_{\eta'} )} 
\hphantom{\varphi_1 ( -\frac{p}{a_j} e^j_{\eta'} + u^{e-c^{(j)}} e_{\eta}^j) } & = e_{\eta'}^{j+1}
\end{cases} \\
j \in I_{\eta'} : & &  \begin{cases} 
\widetilde{\Fil}^1 \M^j & = \langle (u^{e}+p) e^j_{\eta},  e_{\eta'}^j +
a_j u^{e - c^{(j)}}
e^j_{\eta} \rangle \\
\varphi_1( (u^e+p) e^j_{\eta} ) & = e_{\eta}^{j+1}\\
\mathrlap{\varphi_1( e_{\eta'}^j + a_j u^{e-c^{(j)}} e^j_{\eta})}
\hphantom{\varphi_1 ( -b_j e^j_{\eta'} + u^{e-c^{(j)}} e_{\eta}^j) }
& = e_{\eta'}^{j+1} 
\end{cases} \\
j \in II : &  & \begin{cases}
\widetilde{\Fil}^1 \M^j & = \langle a_j e^j_{\eta} + u^{c^{(j)}}
e^j_{\eta'}, -b_j e^j_{\eta'} + u^{e-c^{(j)}} e_{\eta}^j \rangle \\
\varphi_1( a_j e^j_{\eta} + u^{c^{(j)}}
e^j_{\eta'}) &= e_{\eta}^{j+1} \\
\mathrlap{\varphi_1 ( -b_j e^j_{\eta'} + u^{e-c^{(j)}} e_{\eta}^j)}  
\hphantom{\varphi_1 ( -b_j e^j_{\eta'} + u^{e-c^{(j)}} e_{\eta}^j) } &= e_{\eta'}^{j+1}
\end{cases}  \end{eqnarray*}
with $a_j \in R$ if $j\in I_\eta, I_{\eta'}$, and $a_j,b_j\in\m_R$ if $j \in
II$, with $a_jb_j=p$. When
$j=f-1$ the right-hand side of the defining expressions for $\varphi_1$ must be
modified to $\alpha e_{\eta}^0$ and $\alpha' e_{\eta'}^0$ for some
$\alpha,\alpha' \in R^\times$.   
Here $\widetilde{\Fil}^1 \M^j := \Fil^1 \M^j /
(\Fil^p S) \M^j$, and the descent data acts on $e^j_{\eta}$,
$e^j_{\eta'}$ via $\eta, \eta'$ respectively.
It is not difficult to check that an object of this form is indeed a strongly
divisible module in the sense of
Definition~\ref{def:stronglydivisible} (in particular that the condition
(2) is satisfied).  We refer to the elements $e_{\eta}^j$ and $e_{\eta'}^j$ in the above shape as a
\emph{gauge basis} of $\M^j$.

Fix a generic representation $\rhobar$.  If $\rho$ is a
potentially Barsotti--Tate lift of $\rhobar$ of type~$\tau$, and $\cM$
is the corresponding strongly divisible module with $\cO$-coefficients
and descent data, then by
\cite[Prop.~5.2, Prop.~7.1]{breuillatticconj} we see that $\M$ has
shape~$\tau$. 

\begin{rem}
  \label{rem:conventions-functors}
  Recall that our conventions for Hodge--Tate weights differ from
  those of \cite{breuillatticconj}.  In particular where
  \cite{breuillatticconj} associates to $\M$ the Galois representation
  $T^{L}_{\st,2}(\M)$ (where $T_{\st,2}^L$ is the functor
defined between Remark~4.8 and Lemma~4.9 of~\cite{MR2137952}), we instead have $\rho = T^{L}_{\st,2}(\M)(-1)$.
  That is, our representations $\rho$ and $\rhobar$ differ from those
  of \cite{breuillatticconj} by a twist by $\varepsilon^{-1}$. 

 Note
  that this twist preserves genericity, even recalling that our
  fundamental characters are the reciprocals of the ones in
  \cite{breuillatticconj}.  In particular, items from
  \emph{op.~cit.}~such as \cite[Prop.~7.1]{breuillatticconj}, which concern intrinsic statements
  about the shape of strongly divisible modules of type $\tau$ under
  the hypothesis that $\rhobar$ is generic, will remain true in our
  setting, without any changes.
\end{rem}

By \cite[Thm.~5.1.1]{BreuilMezardRaffinee} the sets
$I_{\eta}$, $I_{\eta'}$, $II$, and the indices~$j$ such that $a_j \in
\calO^{\times}$ are completely determined by $\rhobar$ and the type
$\tau$; note that in light of Remark~\ref{rem:conventions-functors}
this statement remains true in our conventions.
Similarly we remark that \cite[Thm.~8.1]{breuillatticconj} shows that we have
$\rhobar$ irreducible if and only if $|II|$ is odd and 
$\val_p(a_j)>0$ for all $j$, non-split reducible if and only if some $a_j$ is a
unit, and split if and only if $|II|$ is even and $\val_p(a_j)>0$ for
all $j$.

\subsection{Deformation spaces of strongly divisible
  modules}\label{subsec:deformation of strongly divisible}  We
continue to assume that $\tau$ is a principal series type.
If $\M$ is a strongly divisible module with $R$-coefficients of shape
$\tau$, then the gauge bases $e^j_{\eta},
e^j_{\eta'}$ are determined by $\M$ up to scalar multiplication.
(This is \cite[Rem.~5.5(iv)]{breuillatticconj} and \cite[Prop.~5.4]{breuillatticconj} in the case $R=\cO$, but
the argument goes over unchanged to the general case.)
Once a choice of gauge basis $e^0_\eta, e^0_{\eta'}$ is fixed, then
the gauge bases
$e^j_{\eta}, e^j_{\eta'}$ and the $a_j, b_j$
and $\alpha,\alpha'$ are also uniquely determined.

We now
consider a further three deformation problems. Fix $\rhobar$ and $\tau$
such that $\rhobar$ has a potentially Barsotti--Tate lift of type
$\tau$.  There is a unique strongly divisible module $\Mbar$ (with
$\F$-coefficients, of shape $\tau$) that occurs as $\M/\m \M$
as $\M$ varies over all strongly divisible modules with
$\cO$-coefficients of shape $\tau$ corresponding to potentially
Barsotti--Tate lifts of $\rhobar$ of type $\tau$.
 (This is clear when $\rhobar$ is reducible and split, because
$\overline{a}_j = 0$ for all $j$, while the parameters
$\overline{\alpha},\overline{\alpha}'$ are determined by the
unramified parts of the characters comprising $\rhobar$ \cite[Ex.~3.7]{MR2457845}; 
when $\rhobar$ has scalar endomorphisms, the claim follows from the proof of \cite[Prop.~5.1.2]{BreuilMezardRaffinee}.)
Fix a choice of the gauge
basis elements $\bar{e}^0_\eta, \bar{e}^0_{\eta'}$ for $\Mbar^0$. 
We then let $R^\tau_{\Mbar}$ be the complete local noetherian
$\cO$-algebra representing the functor which assigns to a complete
local noetherian $\cO$-algebra $R$ the set of isomorphism classes of
tuples $(\M,e^0_\eta,e^0_{\eta'})$
where $\M$ is a strongly divisible module with $R$-coefficients
of shape $\tau$ lifting $\Mbar$, and $e^0_\eta,e^0_{\eta'}$ are a
choice of gauge basis lifting $\bar{e}^0_\eta,\bar{e}^0_{\eta'}$. We
write $X_{\Mbar}\bigl(\tau\bigr):=\Spf R^\tau_{\Mbar}$.

Given a strongly divisible module $\cM$ with $R$-coefficients of shape
$\tau$, we obtain a Galois representation $\rho_{\M}:G_L\to\GL_2(R)$
as follows. If $R$ is Artinian, then we set
$\rho_{\M}:=T_{\st,2}^L(\M)(-1)$, with $T_{\st,2}^L$ as in Remark~\ref{rem:conventions-functors}. In
general, we set $\rho_{\M}:=\varprojlim_n \rho_{\M/\m_R^n}$.

We let $R^{\square,\tau}_{\Mbar}$ denote the
complete local noetherian $\cO$-algebra representing the functor which
assigns to a complete local noetherian $\cO$-algebra $R$ the set of
isomorphism classes of tuples $(\M,e^0_\eta,e^0_{\eta'},\rho)$, where
$(\M,e^0_\eta,e^0_{\eta'})$ is as above, and 
$\rho$ is a lifting of $\rhobar$ such that
$\rho\cong\rho_{\cM}$ (i.e.\ the
same data as that parameterised by $R^\tau_{\Mbar}$, together with a
framing of $\rho_{\cM}$). 

We also consider $R^{\square}_{\Mbar}$, the complete local noetherian
$\cO$-algebra representing the functor which assigns to a complete
local noetherian $\cO$-algebra $R$ the set of isomorphism classes of
pairs $(\M,\rho)$, with $\M$ and $\rho$ as in the previous paragraph.

Now, note that $R^{\square,\tau}_{\Mbar}$ is
formally smooth over both $R^\tau_{\Mbar}$ and $R^{\square}_{\Mbar}$ (the
additional data being a choice of framing of $\rhobar_{\M}$ in the
first case, and a choice of gauge basis $e^0_\eta,e^0_{\eta'}$ in the second
case). In addition, we claim that there is a natural isomorphism
$R^\tau\isoto R^{\square}_{\Mbar}$, induced by the forgetful functor
$(\M,\rho)\mapsto\rho$. To see this, we argue as in the proof of
Th\'eor\`eme~5.2.1 of~\cite{BreuilMezardRaffinee}. This morphism
certainly induces an isomorphism on $\cO_{E'}$-points for each finite
extension $E'/E$ (this is just the statement recalled above, that the
strongly divisible module associated to a potentially Barsotti--Tate
lifting of $\rhobar$ of type $\tau$ necessarily has shape $\tau$), so
it suffices to show that it induces a surjection on reduced tangent
spaces. This reduces to a straightforward explicit calculation with filtered
$\varphi_1$-modules, exactly as in the last paragraph of the proof of
Th\'eor\`eme~5.2.1 of~\cite{BreuilMezardRaffinee}.

Thus parts (1) and (2) of Theorem~\ref{thm:deformation rings principal series} hold in the
principal series case if and
only if the analogous statements hold with $X\bigl(\tau\bigr)$ replaced with the
space $X_{\Mbar}\bigl(\tau\bigr)$.  Similarly we will be able to
deduce parts (3) and (4)  of
Theorem~\ref{thm:deformation rings principal series} in the principal
series case from calculations on
$X_{\Mbar}\bigl(\tau\bigr)$ together with the description of the
relationship between $X\bigl(\tau\bigr)$ and  $X_{\Mbar}\bigl(\tau\bigr)$.
For these reasons, in the principal series case we will work with
$X_{\Mbar}\bigl(\tau\bigr)$ in place of $X\bigl(\tau\bigr)$ from now on.

We will prove Theorem \ref{thm:deformation rings principal series} via
a series of lemmas.

\begin{lem}\label{lem:first-steps} If $\sigma(\tau)$ is a principal series type, then parts
  (1), (2), and (3) of Theorem~\ref{thm:deformation rings principal
    series} hold, with $\Jmin = \iota(I_{\eta'})$ and $\Jmax \setminus
  \Jmin = \iota(II)$ for the sets $I_{\eta'}, II$ determined by $\Mbar$.
  \label{lem: principal series def ring from BM}
\end{lem}
\begin{proof}  Parts~(2) and (3)
  are straightforward from the form of a strongly divisible module of
  shape $\tau$.  Define 
  $\Jmin=\iota(I_{\eta'})$ and $\Jmax\setminus\Jmin=\iota(II)$, so we may
  set $X_j=a_{\iota(j)}, Y_j=b_{\iota(j)}$ for each $j \in \Jmax
  \setminus \Jmin$, and the unnamed formal
  variables correspond to the other values of $a_j$, and to
  $\alpha,\alpha'$ (or more precisely to $a_j - [\overline{a}_j]$,
  where the brackets denote Teichm\"uller lift, and similarly for $\alpha,\alpha'$). Recall
  we have adopted the indexing from \cite{breuillatticconj} rather
  than the indexing from \cite{BreuilMezardRaffinee}; this explains
  the presence of the involution $\iota$.  The argument in part (3) is
  completed by noting that if we begin with a $\Zpbar$-point $\lambda
  \in X\bigl(\tau\bigr)$ and pass to a point 
  $\lambda' \in X_{\Mbar}\bigl(\tau\bigr)$ by lifting to
  $\Spf(R_{\Mbar}^{\square,\tau})$ via formal smoothness and then
  projecting, then $\lambda'$ corresponds to the strongly divisible
  module attached to the potentially Barsotti--Tate representation
  associated to $\lambda$.

Part~(1) will follow from  Proposition~4.3 of \cite{breuillatticconj}
but requires a comparison of the conventions of our paper with the
conventions of \cite{breuillatticconj}.
Note that from \cite[Eq.~(26)]{breuillatticconj} the sets
$\Jmin, \Jmax$ defined above are the \emph{complements} of the sets
$\Jmax, \Jmin$
from \cite{breuillatticconj}.  By \cite[Thm.~8.1]{breuillatticconj},
the Serre weights of $\rhobar(1)$ that are Jordan--H\"older factors of
the representations that Breuil calls
$\sigmabar(\chi^s)$ (and using the definition of being modular of some
weight from
\cite{breuillatticconj}) are exactly the weights that Breuil calls $\sigmabar_J$, for
$\Jmax^c \subseteq J \subseteq \Jmin^c$ (in our conventions for these sets).  We translate to our
conventions.  Note that because our normalization of the inertial local Langlands
correspondence is opposite to that of \cite{breuillatticconj}, the
representation denoted $\sigmabar(\chi^s)$ in \emph{op.~cit.}~is our
representation $\sigmabar(\eta^{-1} \otimes (\eta')^{-1})$, or in
other words our $\sigmabar(\eta' \otimes \eta) \otimes
\det(\tau)^{-1}$.
In
particular Breuil's $\sigmabar_J$ is our $\sigmabar(\tau)_{J^c}
\otimes \det(\tau)^{-1}$.   

On the other hand $\rhobar(1) \cong \rhobar^{\vee} \otimes
\det(\tau)$.  Now $\rhobar$ has Serre weight $\sigma$ (in our
conventions for being modular of some weight) if and only if
$\rhobar^{\vee}$ has Serre weight $\sigma$ (in Breuil's conventions
for being modular of some weight), if and only if $\rhobar(1)$ has
Serre weight $\sigma \otimes \det(\tau)^{-1}$ (in Breuil's
conventions for being modular of some weight, but our conventions for
local Langlands).
It follows
that the Serre weights of $\rhobar$ that are Jordan--H\"older factors
of $\sigmabar(\tau)$ are precisely the $\sigmabar(\tau)_{J^c}$ for
$\Jmax^c \subseteq J \subseteq \Jmin^c$, or in other words the
$\sigmabar(\tau)_J$ for $\Jmin \subseteq J \subseteq \Jmax$.
\end{proof}

Note that the proof of Lemma~\ref{lem:first-steps} gives a
construction of the universal strongly divisible module on
$X_{\Mbar}\bigl(\tau\bigr)$ in the same spirit as the construction (for
$\rhobar$ not split-reducible) of the
universal strongly divisible module on $X\bigl(\tau\bigr)$ in the proof
of \cite[Thm.~5.2.1]{BreuilMezardRaffinee}.

\begin{lem}If $\sigma(\tau)$ is a principal series type, and $\rhobar$
  is reducible,
  then
 part (4) of Theorem~\ref{thm:deformation
    rings principal series} holds.
  \label{lem: identification of components of the special fibre for PS
    with rhobar reducible}
\end{lem}
\begin{proof}
Fix a set $\Jmin \subseteq J \subseteq \Jmax$.
Write $\m^{\tau}_{\Mbar}$ for the maximal ideal of $R^{\tau}_{\Mbar}$, and let
$I_J \subseteq R^{\tau}_{\Mbar}$ be the ideal generated by $\varpi_E$ and
$(\m^{\tau}_{\Mbar})^2$, together with the variables
$X_j$ for $j \in J \setminus \Jmin$, $Y_j$ for
$j\in \Jmax \setminus J$, and all of the unnamed extra power series
variables of Theorem~\ref{thm:deformation
    rings principal series}(2).
Let $\F_J$ be the quotient $R^{\tau}_{\Mbar}/I_J$, so that
$\F_J$ is an
Artinian local ring, and write $\rhobar_J$ for the deformation of
$\rhobar$ corresponding to the natural quotient map $R^{\tau}_{\Mbar} \to \F_J$.
The deformation $\rhobar_J$ lies on the  component of
$\overline{X}_{\Mbar}\bigl(\tau\bigr)$ 
cut out by the equations $X_j = 0$ (for $j \in 
J\setminus \Jmin$)  and
$Y_j = 0$ (for $j \in \Jmax \setminus J$), and only that
component, and so by Theorem~\ref{thm:geometric BM for types} (and
from the relationship between $X\bigl(\tau\bigr)$ and
$X_{\Mbar}\bigl(\tau\bigr)$) it suffices to show that $\rhobar_J$ also lies on $\overline{X}\bigl(\sigmabar_J\bigr)$.

We will carry this out by showing that a twist of $\rhobar_J$ is the
generic fibre of a suitable Fontaine--Laffaille module with
$\F_J$-coefficients.  We refer the reader to
Appendix~\ref{sec:unipotent-FL} for a discussion of what we need about
(unipotent) Fontaine--Laffaille theory, and in particular for some of the notation
used in this argument.

Fix a compatible system of $p^n$th roots of $-p$ in $\overline{F}_v$
and let $F_{v,\infty}$ be the extension of $F_v$ that they generate.
Since $\rhobar$ is generic, the following lemma shows that both $\rhobar_J$ and $\rhobar$ are
uniquely determined by their restrictions to
$\Gal(\overline{F}_{v}/F_{v,\infty})$, and so in the remainder of the
proof it suffices to consider
these restrictions. (We thank Fred Diamond for simplifying our
original proof of the following lemma.)

\begin{lem}
  \label{lem:cohomology-lemma}
  Suppose that $A$ is a local Artin $\Fp$-algebra with maximal ideal
  $\m_A$, and $\chi : G_{F_v}
  \to A^{\times}$ is a character such that $\chi \pmod{\m_A}$ is
  neither trivial nor cyclotomic.  Then the restriction map $
  \mathrm{res}: H^1(G_{F_v},\chi) \to H^1(G_{F_{v,\infty}},\chi)$ is injective.
\end{lem}

\begin{proof}
  We note that our proof will work for any finite extension $F_v/\Qp$
  (not just unramified $F_v$).  Write $\chibar$ for $\chi \pmod{\m_A}$, and
  let $\widehat{F}_v$ be the Galois closure of $F_{v,\infty}$.  By
  inflation-restriction, it suffices to show that
  $H^1(\Gal(\widehat{F}_v/F_v),\chi^{G_{\widehat{F}_v}}) = 0$.  Since
  $\chi^{G_{\widehat{F}_v}}$ is a successive extension of copies
    of $\chibar^{G_{\widehat{F}_v}}$, by d\'evissage it suffices
    to show that
    $H^1(\Gal(\widehat{F}_v/F_v),\chibar^{G_{\widehat{F}_v}}) = 0$.

 If $\chibar^{{G_{\widehat{F}_v}}} \neq 0$, then $\chibar$ is trivial on $G_{\widehat{F}_v}$ (so $\chibar^{G_{\widehat{F}_v}} = \chibar$), and so also on
 $G = \Gal(\overline{F}_v/F_v(\mu_{p^{\infty}}))$.
 Another
 application of inflation-restriction gives
$$ 1 \to
H^1(\Gal(F_v(\mu_{p^{\infty}})/F_v),\chibar)
\to   H^1(\Gal(\widehat{F}_v/F_v),\chibar) \to
H^1(G ,\chibar)^{\Gal(F_v(\mu_{p^{\infty}})/F_v)}.$$
Since $G \cong \Zp(1)$, the rightmost term is
$\Hom_{\Gal(F_v(\mu_{p^{\infty}})/F_v)}(\Zp(1),\chibar)$, and this is
nontrivial if and only if $\chibar$ is cyclotomic.  As for the
leftmost term, one more
application of inflation-restriction
shows that $H^1(\Gal(F_v(\mu_{p^{\infty}})/F_v),\chibar)$ injects into
$
H^1(\Gal(F_v(\mu_{p^{\infty}})/F_v(\mu_p)),\chibar)^{\Gal(F_v(\mu_p)/F_v)}
\cong \Hom_{\Gal(F_v(\mu_p)/F_v)}(\Zp,\chibar)$, and this is non-zero if
and only if $\chibar$ is trivial.  The lemma follows.
\end{proof}

We resume the proof of Lemma~\ref{lem: identification of components of the special fibre for PS
    with rhobar reducible}. From
the construction of the universal strongly divisible module on
$X_{\Mbar}\bigl(\tau\bigr)$ described above, the representation $\rhobar_J$ is the
generic fibre of a strongly divisible module $\M_J =
\M_J^0 \times \cdots \times \M_J^{f-1}$  of shape $\tau$ with
$\F_J$-coefficients and descent data.  In particular the structure of $\M_J$ is described
as follows. 
\begin{eqnarray*} j \in I_{\eta} : & & \begin{cases}
\widetilde{\Fil}^1 \M_J^j & = \langle e_{\eta}^j + a_j u^{c^{(j)}}
e^j_{\eta'}, u^e e^j_{\eta'} \rangle \\
\varphi_1( e_{\eta}^j + a_j u^{c^{(j)}} e^j_{\eta'}) & = e_{\eta}^{j+1} \\
\mathrlap{\varphi_1( u^e e^j_{\eta'} )} 
\hphantom{\varphi_1 (-Y_{\iota(j)} e^j_{\eta'} + u^{e-c^{(j)}}
e^j_{\eta}) } & = e_{\eta'}^{j+1}
\end{cases} \\
j \in I_{\eta'} : & &  \begin{cases} 
\widetilde{\Fil}^1 \M_J^j & = \langle u^{e} e^j_{\eta},  e_{\eta'}^j +
a_j u^{e - c^{(j)}}
e^j_{\eta} \rangle \\
\varphi_1( u^e e^j_{\eta} ) & = e_{\eta}^{j+1}\\
\mathrlap{\varphi_1( e_{\eta'}^j + a_j u^{e-c^{(j)}} e^j_{\eta})}
\hphantom{\varphi_1 (-Y_{\iota(j)} e^j_{\eta'} + u^{e-c^{(j)}}
e^j_{\eta}) }
& = e_{\eta'}^{j+1} 
\end{cases} \\
j \in II_X : &  & \begin{cases}
\widetilde{\Fil}^1 \M_J^j & = \langle X_{\iota(j)} e^j_{\eta} + u^{c^{(j)}}
e^j_{\eta'}, u^{e-c^{(j)}} e_{\eta}^j \rangle \\
\varphi_1( X_{\iota(j)} e^j_{\eta} + u^{c^{(j)}}
e^j_{\eta'}) &= e_{\eta}^{j+1} \\
\mathrlap{\varphi_1 ( u^{e-c^{(j)}} e_{\eta}^j)}  
\hphantom{\varphi_1 (-Y_{\iota(j)} e^j_{\eta'} + u^{e-c^{(j)}}
e^j_{\eta}) } &= e_{\eta'}^{j+1}
\end{cases} \\
j \in II_Y : &  &  \begin{cases}
\widetilde{\Fil}^1 \M_J^j& = \langle  u^{c^{(j)}} e_{\eta'}^j , -Y_{\iota(j)} e^j_{\eta'} + u^{e-c^{(j)}}
e^j_{\eta} \rangle \\
\varphi_1(u^{c^{(j)}} e_{\eta'}^j) &= e_{\eta}^{j+1} \\
\mathrlap{\varphi_1 (-Y_{\iota(j)} e^j_{\eta'} + u^{e-c^{(j)}}
e^j_{\eta}) }
\hphantom{\varphi_1 (-Y_{\iota(j)} e^j_{\eta'} + u^{e-c^{(j)}}
e^j_{\eta}) }
&= e_{\eta'}^{j+1}
\end{cases}  \end{eqnarray*}
Here our notation is as in the definition of ``shape $\tau$'' in
\ref{def:stronglydivisible}, except that we have decomposed $II$ as $II_X
\coprod II_Y$ with $II_X = \iota(\Jmax \setminus
J)$ and $II_Y = \iota(J \setminus \Jmin)$.  Recall also
 that by Lemma~\ref{lem: principal series def ring from BM} we have $\Jmin =
\iota(I_{\eta'})$  and
$\cS \setminus \Jmax = \iota(I_{\eta})$.  
As usual, when
$j=f-1$ the right-hand side of the defining expressions for $\varphi_1$ should be
modified to $\alpha e_{\eta}^0$ and $\alpha' e_{\eta'}^0$.
For the remainder of this argument
it is convenient for us to define
$e_{\eta}^f := \alpha e_{\eta}^0$ and $e_{\eta'}^f := \alpha'
e_{\eta'}^0$ so that from now on we can ignore this last
complication.

The generic fibre of the strongly divisible module $\M_J$ is
also the generic fibre of a certain \'etale $\varphi$-module; we show
this following the method in the proof of Proposition~7.3 of
\cite{breuillatticconj}.
 The first step is to note
that by \cite[Prop.~A.2(i)]{breuillatticconj} the object $\M_J$ comes
from a unique $\varphi$-module $\mathfrak{M}_J$ of type
$\overline{\chi}$ over $k_v \otimes_{\Fp} \F \llbracket u \rrbracket$
(see the discussion preceding \emph{loc. cit.}, which in particular gives the recipe
for recovering $\M_J$ from $\mathfrak{M}$).
By functoriality $\mathfrak{M}_J$ has an action of~$\F_J$.
Set $\mathfrak{D}_J = \mathfrak{M}_J[1/u]$; this is a $\varphi$-module of
type $\chi$ in the sense of \cite[D\'ef.~A1]{breuillatticconj}.
There is a decomposition $\mathfrak{D}_J =
\mathfrak{D}^0_J \times \cdots \times \mathfrak{D}^{f-1}_J$ where
$\mathfrak{D}^j_J = \F_J((u))\mathfrak{e}^j_{\eta} \oplus
\F_J((u))\mathfrak{e}^j_{\eta'}$ and
$\Gal(F_{v,\infty}(\sqrt[e]{-p})/F_{v,\infty})$ acts on
$\mathfrak{e}^j_{\eta},\mathfrak{e}^j_{\eta'}$ via
$\overline{\eta},\overline{\eta}'$ respectively, and one
calculates that
\begin{eqnarray*}
j \in I_{\eta} : & &
\begin{cases}
\varphi(\mathfrak{e}^{j-1}_{\eta}) & =  u^e \mathfrak{e}^j_{\eta}  - a_j u^{c^{(j)}} \mathfrak{e}^j_{\eta'} \\
\varphi(\mathfrak{e}^{j-1}_{\eta'}) & = \mathfrak{e}^j_{\eta'} 
\end{cases} \\
j \in I_{\eta'} : & &
\begin{cases}
\varphi(\mathfrak{e}^{j-1}_{\eta}) & = \mathfrak{e}^j_{\eta}  \\
\varphi(\mathfrak{e}^{j-1}_{\eta'}) & = u^e \mathfrak{e}^j_{\eta'} -
a_j u^{e - c^{(j)}} \mathfrak{e}^j_{\eta}  
\end{cases} \\
j \in II_{X} :& &
\begin{cases}
\varphi(\mathfrak{e}^{j-1}_{\eta}) & = u^{c^{(j)}} \mathfrak{e}^j_{\eta'}  \\
\varphi(\mathfrak{e}^{j-1}_{\eta'}) & = -X_j \mathfrak{e}^j_{\eta'} +
u^{e - c^{(j)}} \mathfrak{e}^j_{\eta} 
\end{cases} \\
j \in II_{Y} : & &
\begin{cases}
\varphi(\mathfrak{e}^{j-1}_{\eta}) & = u^{c^{(j)}}
\mathfrak{e}^j_{\eta'}  + Y_j \mathfrak{e}^j_{\eta}  \\
\varphi(\mathfrak{e}^{j-1}_{\eta'}) & = u^{e - c^{(j)}}
\mathfrak{e}^j_{\eta}  .
\end{cases} 
\end{eqnarray*}
with a suitable modification when $j = f-1$.  (To check this, it
suffices to define $\mathfrak{M}_J$ by the same formulas, and then
follow the recipe for recovering $\M_J$ from $\mathfrak{M}_J$.)

By \cite[Prop.~A.2(ii)]{breuillatticconj}
  the generic fibre of $\mathfrak{D}_J$ is $\rhobar_J^{\vee}
  |_{\Gal(\overline{F}_v/F_{v,\infty})}$.  Set $\mathfrak{e}^j =
  \mathfrak{e}^j_{\eta}$ and $\mathfrak{f}^j = u^{c^{(j)}}
  \mathfrak{e}^j_{\eta'}$; one checks that with respect to the basis
  $\mathfrak{e}^j,\mathfrak{f}^j$, the action of $\varphi$ involves
  only powers of $u^e$; it follows from the isomorphism
  \cite[Eq.~(47)]{breuillatticconj} that replacing $u^e$ with $u$
  gives a $\varphi$-module $\prod_j \F_J((u)) \mathfrak{e}^j \oplus
  \F_J((u)) \mathfrak{f}^j$ without descent data whose generic fibre
  is $(\rhobar_J^{\vee} \otimes
  \overline{\eta})|_{\Gal(\overline{F}_v/F_{v,\infty})}$; the action
  of $\varphi$ is given by
\begin{eqnarray*}
j \in I_{\eta} : & &
\begin{cases}
\varphi(\mathfrak{e}^{j-1}) & = u \mathfrak{e}^j - a_j \mathfrak{f}^j \\
\varphi(\mathfrak{f}^{j-1}) & = u^{p-1-c_{f-j}} \mathfrak{f}^j 
\end{cases} \\
j \in I_{\eta'} : & &
\begin{cases}
\varphi(\mathfrak{e}^{j-1}) & = \mathfrak{e}^j \\
\varphi(\mathfrak{f}^{j-1}) & = u^{p-c_{f-j}} (\mathfrak{f}^j  -
a_j \mathfrak{e}^j )
\end{cases} \\
j \in II_{X} : & &
\begin{cases}
\varphi(\mathfrak{e}^{j-1}) & = \mathfrak{f}^j  \\
\varphi(\mathfrak{f}^{j-1}) & = - X_j u^{p-1-c_{f-j}} \mathfrak{f}^j
+ u^{p-c_{f-j}} \mathfrak{e}^j 
\end{cases} \\
j \in II_{Y} : & &
\begin{cases}
\varphi(\mathfrak{e}^{j-1}) & = \mathfrak{f}^j + Y_j \mathfrak{e}^j  \\
\varphi(\mathfrak{f}^{j-1}) & = u^{p-c_{f-j}} \mathfrak{e}^j .
\end{cases} 
\end{eqnarray*}

Define
$$ v_{f-j} = \begin{cases}
p-1 & \text{if } f-j \in J \\
c_{f-j} & \text{if } f-j \not\in J,\, f-j-1 \not\in J \\
c_{f-j}-1 & \text{if } f-j \not\in J,\, f-j-1 \in J. \end{cases}$$
Note that $v_{f-j}$ has been chosen so that $v_i = s_{J,i} + t_{J,i}$
with $s_{J,i}, t_{J,i}$ as in Section~\ref{ss:principal}. Write $\widetilde{\omega}_f$ for the extension of $\omega_f$ to $\Gal(\overline{F}_v/F_v)$
given by composing $g \mapsto (g(\uni)/\uni)^{-1}$ with~$\embb_0$.  
We twist by
$\widetilde{\omega}_f^{-\sum_{j=0}^{f-1}  v_j p^j}$;
as in the proof of \cite[Prop.~7.3]{breuillatticconj}, this
 has the effect of multiplying
$\varphi(\mathfrak{e}^{j-1})$ and $\varphi(\mathfrak{f}^{j-1})$ by
$u^{v_{f-j}}$ for all $j$.  Make the change of 
basis $(\frac 1u \mathfrak{e}^{j-1}, 
\frac 1u \mathfrak{f}^{j-1})$ when $f-j \in J$ and
$(\mathfrak{e}^{j-1}, \frac 1u \mathfrak{f}^{j-1})$ when $f-j \not\in J$,
and let $\mathfrak{M}$ be the resulting $\varphi$-module.

Now define a Fontaine--Laffaille module
$M = M^0 \times \cdots \times M^{f-1}$ with $M^j = \F_J e^j \oplus
\F_J f^j$ and 
\begin{eqnarray*}
 j \in I_{\eta}, \ f-j \in J : & & \begin{cases}
 \varphi(e^j) & = e^{j+1} - a_j f^{j+1} \\
\mathrlap{\varphi_{p-c_{f-j}-1}(f^j)}
\hphantom{ \varphi_{p-c_{f-j}-1}(e^j) } 
& = f^{j+1}
\end{cases} \\
 j \in I_{\eta}, \ f-j \not\in J  : & & \begin{cases}
 \varphi_{c_{f-j}+1}(e^j) & = e^{j+1} - a_j f^{j+1} \\
\mathrlap{\varphi(f^j)}
\hphantom{ \varphi_{p-c_{f-j}-1}(e^j) } 
 & = f^{j+1} 
\end{cases} \\
 j \in I_{\eta'}, \ f-j \in J :  & & \begin{cases}
 \varphi(e^j) & =  e^{j+1} \\
\mathrlap{\varphi_{p-c_{f-j}}(f^j) }
\hphantom{ \varphi_{p-c_{f-j}-1}(e^j) } 
& = f^{j+1} - a_j e^{j+1}
\end{cases} \\
 j \in I_{\eta'}, \ f-j \not\in J: & & \begin{cases}
 \varphi_{c_{f-j}}(e^j) & = e^{j+1} \\
\mathrlap{\varphi(f^j) }
\hphantom{ \varphi_{p-c_{f-j}-1}(e^j) } 
& = f^{j+1} - a_j e^{j+1}
\end{cases} \\
 j \in II_X, \ f-j \in J: & & \begin{cases}
\varphi(e^j) & = f^{j+1} \\
\mathrlap{\varphi_{p-c_{f-j}-1 } (f^j) }
\hphantom{ \varphi_{p-c_{f-j}-1}(e^j) } 
& = -X_j f^{j+1} + e^{j+1} 
\end{cases} \\
 j \in II_X, \ f-j \not\in J  : & & \begin{cases}
\varphi_{c_{f-j}+1} (e^j) & = f^{j+1} \\
\mathrlap{\varphi(f^j) }
\hphantom{ \varphi_{p-c_{f-j}-1}(e^j) } 
& = -X_j f^{j+1} + e^{j+1}  
\end{cases} \\
 j \in II_Y, \ f-j \in J  : & & \begin{cases}
\varphi(e^j) & = f^{j+1} + Y_j e^{j+1} \\
\mathrlap{\varphi_{p-c_{f-j}}(f^j) }
\hphantom{ \varphi_{p-c_{f-j}-1}(e^j) } 
& = e^{j+1}
\end{cases} \\
 j \in II_Y, \ f-j \not\in J : & & \begin{cases}
\varphi_{c_{f-j}}(e^j) & = f^{j+1} + Y_j e^{j+1} \\
\mathrlap{\varphi(f^j)}
\hphantom{ \varphi_{p-c_{f-j}-1}(e^j) } 
& = e^{j+1} .
\end{cases}
\end{eqnarray*}
It is easy to check that $M$ and $\mathfrak{M}$ are unipotent in the sense of
Definitions~\ref{defn:unipotent-FL-modules}
and~\ref{defn:etale-phi-unip} respectively, 
and that $\Theta_{p-1}(\mathfrak{M}) \cong \EF_{p-1}(M)$ (see
Appendix~\ref{sec:unipotent-FL} for the functors
$\Theta_{p-1},\EF_{p-1})$.   By 
Theorem~\ref{thm:equivalence-A} we see that
$\Theta_{p-1}(\mathfrak{M}) \cong \EF_{p-1}(M)$ is unipotent in the
sense of Definition~\ref{defn:unipotent-str-div}, and so combining
Propositions~\ref{prop:isomA} and~\ref{prop:isomB} shows that
$$T(M)|_{\Gal(\overline{F}_v/F_{v,\infty})}\cong T(\mathfrak{M}) \cong (\rhobar_J^{\vee} \otimes \etabar \widetilde{\omega}_f^{-\sum_{j=0}^{f-1}  v_j p^j})  |_{\Gal(\overline{F}_v/F_{v,\infty})}.$$

Recalling that $f-j-1 \in J$ if and only if $j \in I_{\eta'} \cup
II_Y$, we see that $M$ corresponds to a point on the
deformation space of crystalline liftings of $\rhobar^{\vee} \otimes \etabar \widetilde{\omega}_f^{-\sum_{j=0}^{f-1} v_j
  p^j}$ of Hodge type $(-\vec{s}-\vec{1},\vec{0})$ where 
$$s_j = s_{J,j} = \begin{cases} p-1-c_{j} - \delta_{J^c}(j-1) &
  \text{if } j \in J\\
c_{j} - \delta_J(j-1) & \text{if } j \not\in J.\end{cases}$$

Note that $\rhobar^{\vee} \otimes \etabar \etabar'\varepsilonbar^{-1} \simeq
\rhobar$.  Similarly, since  $\rhobar_J$ lies on $X_{\Mbar}\bigl(\tau\bigr)$ we
must have
$\rhobar_J^{\vee} \otimes \etabar \etabar'\varepsilonbar^{-1} \nu \simeq
\rhobar_J$ for some unramified character $\nu$.
We
conclude that 
$$\rhobar_J \simeq \rhobar_J^{\vee} \otimes \etabar \etabar' \varepsilonbar^{-1} \nu \simeq (\rhobar_J^{\vee}
\otimes \etabar \widetilde{\omega}_f^{-\sum_{j=0}^{f-1} v_j p^j}) \otimes
\widetilde{\omega}_f^{\sum_{j=0}^{f-1} v_j p^j } \etabar'
\varepsilonbar^{-1} \nu $$ comes from a point on
the deformation space of crystalline liftings of $\rhobar$ of Hodge
type $\sigmabar_{\vec{v}-\vec{s},\vec{s}} \otimes (\etabar'
\circ \det) = \sigmabar_{\vec{t},\vec{s}} \otimes (\etabar' \circ
\det) = \sigmabar_J(\tau)$.
\end{proof}

\begin{proof}[Proof of Theorem~\ref{thm:deformation rings principal series}]

  We use the theory of base change explained in Section
  \ref{sec:base change}.  The first sentence of Theorem~\ref{thm:deformation rings principal
   series}(1) is a consequence of Theorem~\ref{thm:geometric BM for
   types}, so we may assume throughout the argument that $X\bigl(\tau\bigr)$ is
 non-empty.     Recall that $L_v$ denotes the quadratic
  unramified extension of $F_v$. Then $\tau':=\BC(\tau)$ is a
  principal series representation, and $\rhobar':=\rhobar|_{G_{L_v}}$
  is reducible, so Theorem \ref{thm:deformation rings principal
    series} holds for $\rhobar'$ and $\tau'$ by Lemmas \ref{lem:
    principal series def ring from BM} and \ref{lem: identification of
    components of the special fibre for PS with rhobar
    reducible}. 

  We take $L=L_v$, $f'=2f$ from now on. Given a strongly divisible
  module (with tame descent data and coefficients) $\M'$  such that $\M'$ has shape $\tau'$, we
  define another strongly divisible module $c(\M')$ of shape $\tau'$
  in the following way.  Let $\varphi^f$ denote the Frobenius 
  on $L/F_v$.  Then we take $(c(\M'),\Fil^1 c(\M')) = (\M',\Fil^1
  \M')$ as sets and as strongly divisible $S_{\cO_{F_v},R}$-modules; we let $\cO_L$ act on $c(\M')$
  through~$\varphi^f$;  and the additive bijection $\widehat{g} : c(\M') \to
  c(\M')$ is given by $\widehat{\varphi^f g \varphi^f} : \M' \to \M'$ for each $g \in
  \Gal(L(\pi)/L)$.  (In particular the roles of $\eta,\eta'$ are
  exchanged if $\tau$ is cuspidal, and preserved if $\tau$ is
  principal series.)  The identity map on underlying sets gives a
  $\varphi^f$-semilinear map $c_{\M'}: \M' \to c(\M')$ and a
  natural identification $(\M')^j\isoto c(\M')^{j+f} $.  Furthermore, there is a natural isomorphism
  $\rho_{c(\M')}\cong\rho_{\M'}^{\varphi^f}$, where
  $\rho_{\M'}^{\varphi^f}$ is the conjugate of $\rho_{\M'}$ by the generator
  $\varphi^f$ of the Galois group $\Gal(L_v/F_v)$. (To see this last
  statement, it suffices to consider the case that $R$ is
  Artinian. In this case
  $\rho_{\M'}=T_{\st,2}^L(\M')(-1)=\Hom_{\phi,\Fil^1}(\M',\hat{A}_{\st,\infty})^{\vee}$,
  and it is easily checked that the
  map \[\Hom_{\phi,\Fil^1}(\M',\hat{A}_{\st,\infty})\to
  \Hom_{\phi,\Fil^1}(c(\M'),\hat{A}_{\st,\infty})\] sending $h\mapsto
  \varphi^f\circ h\circ c_{\M'}^{-1}$ (where $\phi$ is the Frobenius on
  $\hat{A}_{\st,\infty}$ as in Section~2.2.2 of~\cite{MR1681105}) is
  Galois equivariant after conjugating the left hand side by $\varphi^f$.)

Write $\Mbarp$ for the strongly divisible module (with
  $\F$-coefficients and descent data) of shape $\tau'$ corresponding
  to $\rhobar'$.  In the principal series case, note that $\Mbarp = l_v \otimes_{k_v} \Mbar$, with
  the action of $\Gal(F_v(\pi)/F_v)$ on $\Mbar$ extended
  $l_v$-linearly to give the action of $\Gal(L(\pi)/L)$ on $\Mbarp$.  In
  particular there is a canonical $S_{\cO_L,R}$-linear isomorphism $\gamma : \Mbarp \isoto c(\Mbarp)$
  given by the generator of $\Gal(l_v/k_v)$ on the first factor of
  $\Mbarp = l_v \otimes_{k_v} \Mbar$ and
  the identity on the second factor.  In the cuspidal case we no
  longer \emph{a priori} have an $\Mbar$, but there still exists a canonical
  isomorphism  $\gamma : \Mbarp \isoto c(\Mbarp)$ as above.  By
  hypothesis there exists an $\cO'$-point $x \in X\bigl(\tau\bigr)$
  for the ring of integers $\cO'$ in some finite extension $E'/E$ with
  residue field $\F$.  Let $x'
  \in X\bigl(\tau'\bigr)$ be the point arising from $x$ by restriction to
  $G_{L_v}$.  If $\cM,\cM'$ are the strongly divisible modules
  corresponding to $x,x'$, then $\cM'$ is obtained from $\cM$ by
  restricting the descent data.  In particular there is an involutive
  isomorphism $\tilde{\gamma} : \M' \isoto c(\M')$ coming from the
  action on $\cM$ of the element of $\Gal(L_v(\pi)/F_v)$
  lifting $\varphi^f$ and fixing $\pi$.  Then $\gamma$ is the
  reduction of this map modulo $\m_{E'}$.  In any case we see that $c_{\Mbarp}  \gamma^{-1}
   c_{\Mbarp} = \gamma$.  For the remainder of this argument we write $c$ for $c_{\M'}$ and
   $\bar{c}$ for $c_{\Mbarp}$.

 We have the universal deformation rings $R^{\tau'}_{\Mbarp}$ and
 $R^{\square,\tau'}_{\Mbarp}$ considered
  in Section~\ref{subsec:deformation of strongly
    divisible}, which depended on our choice of gauge
basis elements $\bar{e}^0_\eta, \bar{e}^0_{\eta'}$ for $(\Mbar\vphantom{\M}')^0$. 
Then we let $R^{\tau',f}_{\Mbarp}$ be the complete local noetherian
$\cO$-algebra representing the functor which assigns to a complete
local noetherian $\cO$-algebra $R$ the set of isomorphism classes of
tuples $(\M',e^0_\eta,e^0_{\eta'},\epsilon_\eta^f,\epsilon_{\eta'}^f)$,
where $\M'$, $e^0_\eta,e^0_{\eta'}$ are as above, and
$\epsilon_\eta^f,\epsilon_{\eta'}^f$ are a 
choice of gauge basis for $(\M')^f$ lifting $\{(\gamma^{-1} 
\bar{c})(\bar{e}^0_\eta),(\gamma^{-1}  \bar{c})(\bar{e}^0_{\eta'})\}$. We
write $X_{\Mbarp,f}\bigl(\tau'\bigr):=\Spf R^{\tau',f}_{\Mbar}$. Similarly,
$R^{\square,\tau',f}_{\Mbarp}$ represents the functor assigning to $R$
the set of isomorphism classes of
tuples
$(\M',e^0_\eta,e^0_{\eta'},\epsilon_\eta^f,\epsilon_{\eta'}^f,\rho')$
with $\rho'$ a lifting of $\rhobar'$ such that $\rho'\cong\rho_{\M'}$.

There is a natural involutive action of $\varphi^f$  on each of $R^{\tau',f}_{\Mbarp}$,
$R^{\square,\tau',f}_{\Mbarp}$, $R^\square_{\Mbarp}$, and $R^{\tau'}$ by sending $(\M',\{e^0_\eta,e^0_{\eta'}\},\{\epsilon^f_\eta,\epsilon_{\eta'}^f\})$ to $(c(\M'),\{c(\epsilon^f_\eta),c(\epsilon_{\eta'}^f)\},\{c(e^0_\eta),c(e^0_{\eta'})\})$ and conjugating $\rho'$ by
$\varphi^f$.  (Note that the latter will indeed define a point of each
of the deformation spaces we are considering, after  identifying
$(c(\Mbar'),\{\gamma(\bar{e}^0_{\eta}),\gamma(\bar{e}^0_{\eta'})\})$
with $(\Mbar',\{\bar{e}^0_{\eta},\bar{e}^0_{\eta'}\})$ via $\gamma^{-1}$.)

 Similarly if $X,X'$ are the universal framed deformation spaces
 of $\rhobar$ and $\rhobar|_{G_{L_v}}$, then we have a natural map
  $X \to (X')^{\varphi^f}$
 which is easily seen to be an isomorphism
  since $p>2$. Since the extension $L_v/F_v$ is
  unramified, we see that a lift $\rho$ of $\rhobar$ is potentially
  Barsotti--Tate of type $\tau$ if and only if $\rho':=\rho|_{G_{L_v}}$ is
  potentially Barsotti--Tate of type $\BC(\tau)$, so that furthermore
  $X\bigl(\tau\bigr)=X\bigl(\tau'\bigr)^{\varphi^f}$. 
 It is easy to see that $(\Spf
  R^{\square,\tau',f}_{\Mbarp})^{\varphi^f}$ is formally smooth over
  both $X\bigl(\tau'\bigr)^{\varphi^f}$ and
  $X_{\Mbarp,f}\bigl(\tau'\bigr)^{\varphi^f}$.
 Furthermore, in
  the principal series case the Dieudonn\'e
  module with descent data $M'$ and strongly divisible module with
  descent data $\cM'$ corresponding to $\rho'$ are obtained from
  those corresponding to $\rho$ by tensoring with $\cO_{L_v}$ over
  $\cO_{F_v}$ and extending $\varphi$ semilinearly; in the cuspidal
  case, the Dieudonn\'e module for $\rho'$ is equal to the one for
  $\rho$, and the strongly divisible module is obtained by restricting
  the descent data from $\Gal(L_v(\pi)/F_v)$ to $\Gal(L_v(\pi)/L_v)$.
  Consequently, we see that in order to establish parts (1) and (2) of
  Theorem~\ref{thm:deformation rings principal series}, it is enough
  to prove the analogous statements for
  $X_{\Mbarp,f}\bigl(\tau'\bigr)^{\varphi^f}$; parts (3) and (4) will
  also follow from considerations on this latter space. 

Choose a gauge basis $\bar{e}^0_{\eta},\bar{e}^0_{\eta'}$ for
$(\Mbarp)^0$, and recursively define 
$\bar{e}^j_{\eta},\bar{e}^j_{\eta'}$ for $0 < j < 2f$ by the formulas of
Section~\ref{ss:strongly-divisible}.  Define
$\bar{\epsilon}_{\eta}^{j+f}$ and $\bar{\epsilon}_{\eta'}^{j+f}$ to be 
$(\gamma^{-1}
\bar{c})(\bar{e}^j_{\eta})$ and $ (\gamma^{-1}
\bar{c})(\bar{e}^j_{\eta'})$ respectively if $\tau$ is principal series, and
non-respectively if $\tau$ is cuspidal.  

For $0 \le j < f$ we write $(\Mbarp)^j$ in
terms of the gauge basis $\bar{e}^j_{\eta},\bar{e}^j_{\eta'}$ as in
Section~\ref{ss:strongly-divisible}  (denoting the
structure constants in the type $I$ filtrations as $\bar{a}_j$; note
that the structure constants in type $II$ filtrations are $0$ here).
For $f \le j < 2f$ we instead write $(\Mbarp)^{j}$ in terms of the
gauge basis $\bar{\epsilon}^j_{\eta},\bar{\epsilon}^j_{\eta'}$ (again
denoting the structure constants as $\bar{a}_j$).
In view of the fact
that $\gamma : \Mbarp \to c(\Mbarp)$ is an isomorphism, we see that
the sets $I_{\eta}, I_{\eta'}, II$ for $\Mbarp$
are invariant under translation by $f$ if $\tau$ is principal series,
and $II$ is similarly invariant under translation by $f$ if $\tau$ is cuspidal while
$I_{\eta},I_{\eta'}$ are interchanged; that $\bar{a}_{j+f} =
\bar{a}_j$ for all $j$; and that in this basis, if the right-hand side
of  the defining expressions for $\varphi_1$ for $j = 2f-1$
are written as
 $\bar{\beta} \bar{e}_{\eta}^0$ and $\bar{\beta'}
\bar{e}_{\eta'}^0$, then the right-hand side of the defining
expressions for $\varphi_1$ for $j = f-1$ can be written as
$\bar{\beta} \bar{\epsilon}_{\eta}^f$ and $\bar{\beta'} \bar{\epsilon}_{\eta'}^f$.

Consider an arbitrary point
$x = (\M',\{e^0_{\eta},e^0_{\eta'}\},\{\epsilon_{\eta}^f,\epsilon_{\eta'}^f\})$
of $X_{\Mbarp,f}\bigl(\tau'\bigr)$.  Recursively define $e^j_{\eta},e^j_{\eta'}$
for $0 \le j < f$ and $\epsilon^j_{\eta},\epsilon^j_{\eta'}$ for $f
\le j < 2f$ by the formulas of Section~\ref{ss:strongly-divisible}.
As we did with $\Mbarp$, we write $\M'$ in terms of the gauge basis
$e^j_{\eta},e^j_{\eta'}$ for $0 \le j < f$ (denoting the structure
constants in the filtration as $a_j, b_j$), and in terms of the gauge
basis $\epsilon^j_{\eta},\epsilon^j_{\eta'}$ for $f \le j < 2f$ (here
denoting the structure constants as $a'_j,b'_j$).   In this basis, let
the right-hand side of the defining expressions for $\varphi_1$ for $j
= f-1$ and $2f-1$ be $\beta \epsilon_{\eta}^f, \beta'
\epsilon_{\eta'}^f$ and $\tilde{\beta} e_{\eta}^0, \tilde{\beta}'
e_{\eta'}^0$  respectively.   In particular $\beta,\tilde{\beta}$ are both
lifts of $\bar{\beta}$, and similarly $\beta',\tilde{\beta}'$ are both
lifts of $\bar{\beta'}$.

Then $\varphi^f(x) = x$ if and only if there is an isomorphism $c(\M')
\isoto \M'$ lifting $\gamma^{-1}$ and sending $\{c(\epsilon^f_{\eta}),
c(\epsilon^f_{\eta'})\}$, $\{c(e^0_{\eta}),c(e^0_{\eta'})\}$ 
to $\{e^0_{\eta},e^0_{\eta'}\}$,
$\{\epsilon^f_{\eta},\epsilon^f_{\eta'}\}$ respectively.    It follows
immediately from the above description that this is the case precisely
when $\beta=\tilde{\beta}$, $\beta' = \tilde{\beta}'$, $a_j =
a'_{j+f}$ for all $0 \le j < f$, and $b_j = b'_{j+f}$ for all $j \in
II \cap \{0,\ldots,f-1\}$, except that if $\tau$ is cuspidal we have
$a_j = -b'_{j+f}$ and $b_j = -a'_{j+f}$ for all $j \in II\cap \{0,\ldots,f-1\}$.

Now rewrite $\M'$ entirely in terms of the gauge bases
$e^j_{\eta},e^j_{\eta'}$ for $0 \le j < 2f$, as in
Section~\ref{ss:strongly-divisible}.  By the change-of-basis formulae
in the proof of \cite[Prop.~5.4]{breuillatticconj} we obtain $a_{j+f}
= a_j (\beta/\beta')^{\gamma_j}$ for some $\gamma_j \in \{ \pm 1\}$ for
all $0 \le j < f$, except that again  if $\tau$ is cuspidal and $j \in II \cap \{0,\ldots,f-1\}$ we have
$a_j = -b_{j+f} (\beta/\beta')^{\gamma_j}$ and $b_j = -a_{j+f}
(\beta/\beta')^{-\gamma_j}$.   (The exponents $\gamma_j$ are
determined by the sets $I_{\eta},I_{\eta'},II$, and so by $\rhobar$
and $\tau$.)  Furthermore we have $\alpha,\alpha'$
equal to $\beta^2, (\beta')^2$ if $II \cap \{0,\ldots,f-1\}$ is even, and
$\alpha=\alpha' = \beta\beta'$ if $II \cap \{0,\ldots,f-1\}$ is odd.  

For the remainder of this proof we will use $\iota$ for the
involution $\iota(j) = f-1-j$ on $\{0,\ldots,f-1\}$, and $\iota'$ for
the analogous involution on $\{0,\ldots,f'-1\}$.     
Write $J_0 = \iota(II \cap \{0,\ldots,f-1\})$.
It follows easily from the previous paragraph and 
from the description of $X_{{\Mbarp},f}\bigl(\tau'\bigr)$ given by Lemma~\ref{lem: principal series def ring from BM} that
$X_{{\Mbarp},f}\bigl(\tau'\bigr)^{\varphi^f}$ is the formal spectrum
of a ring $R$ equal to a power
series ring over $\cO\llbracket (X_j,Y_j)_{j \in J_0}
\rrbracket/(X_j Y_j - p)_{j \in J_0}$, where the variables $X_j, Y_j$
for $j \in J_0$ are taken to be the restrictions to
$X_{{\Mbarp},f}\bigl(\tau'\bigr)^{\varphi^f}$ of the 
variables $X_j, Y_j$ on $X_{{\Mbarp},f}\bigl(\tau'\bigr)$.   (Note
that when $|J_0|$ is even we use the hypothesis that $p
\neq 2$, so that $(1 + Z)^{1/2} \in \cO\llbracket Z
\rrbracket$.)
This gives part (2) of the Theorem.

Part (3) now follows by an argument similar to that in the proof of
Lemma~\ref{lem:first-steps} (if a $\Zpbar$-point $\lambda \in X\bigl(\tau\bigr)$
leads to a point $\lambda' \in
X_{{\Mbarp},f}\bigl(\tau'\bigr)^{\varphi^f}$, and if $\lambda$
corresponds to the potentially Barsotti--Tate representation $\rho$, then $\lambda'$ corresponds
to the strongly divisible module associated to $\rho |_{G_L}$) together with the behavior of $\eta$- and
$\eta'$-gauges under base change; for the latter one considers the
principal series and cuspidal cases separately.

Now let $\Jmin'$ and $\Jmax'$ be the sets from part (1) of the Theorem
applied to
$\rhobar'$ and~$\tau'$, so that $\Jmin' = \iota'(I_{\eta'})$ and
$\Jmax' = \iota'(I_{\eta'} \cup II)$.  Define $\Jmin = \Jmin' \cap
\{0,\ldots,f-1\}$ and $\Jmax = \Jmax' \cap \{0,\ldots,f-1\}$.  Observe
that the set $J_0$ from the previous paragraph is precisely $\Jmax \setminus \Jmin$.

 Write $J'$ for
  $\BC_{\cusp}(J)$ or $\BC_{\PS}(J)$ depending on whether $\tau$ is a
  cuspidal or principal series inertial type; it is easy to verify  that if $J
  \subseteq \cS$ then $\Jmin \subseteq J \subseteq \Jmax$ if and only if
  $\Jmin' \subseteq J' \subseteq \Jmax'$.

  Now, by Lemma \ref{lem: base change of JH factors} and the
  definition of $\cD(\rhobar)$, we see that if
  $\sigmabar_J(\tau)\in\cD(\rhobar)$ then
  $\sigmabar_{J'}(\tau')\in\cD(\rhobar')$.  In particular, we have
  $\Jmin'\subseteq J'\subseteq\Jmax'$, or equivalently $\Jmin
  \subseteq J \subseteq \Jmax$.   
Conversely, we claim that if $J$
  is such that $\Jmin'\subseteq J'\subseteq\Jmax'$, then
  $\sigmabar_J(\tau)\in\cD(\rhobar)$. In order to see this, we need to
  check that $\rhobar$ has a crystalline lift of Hodge type
  $\sigmabar_J(\tau)$. In addition, we need to verify the
  assertion of Theorem \ref{thm:deformation rings principal
    series}(4), by identifying explicit equations for
  $\Xbar\bigl(\sigmabar_J(\tau)\bigr)$.

  In fact, we can verify these conditions simultaneously, as follows:
  let $\Xbar_J$ be the component of $\Xbar\bigl(\taubar\bigr)$ cut out by the
  equations $X_j = 0$  (for $j \in J\setminus
  \Jmin$) and $Y_j = 0$ (for $j \in \Jmax
  \setminus J$). From Lemma~\ref{lem: identification of components of the special fibre for PS
    with rhobar reducible}, from the relationship between
  $X\bigl(\tau'\bigr)$ and $X_{\Mbarp}\bigl(\tau'\bigr)$, and by the
  discussion above, we see that
  $\Xbar_J=\Xbar\bigl(\sigmabar_{J'}(\tau')\bigr)^{\varphi^f}$. On the other
  hand, by Theorem \ref{thm:geometric BM for types} we know that we
  must have $\Xbar_J=\Xbar\bigl(\sigmabar(\tau)_K\bigr)$ for some $K$, and we
  then have $\Xbar_J\subseteq\Xbar\bigl(\BC(\sigmabar(\tau)_K)\bigr)^{\varphi^f}$
  by definition. By Lemma \ref{lem: base change of JH factors}, this
  means that
  $\Xbar_J\subseteq\Xbar\bigl(\sigmabar(\tau')_{K'}\bigr)^{\varphi^f}$, so that
  $K'=J'$, and $K=J$, as required.
\end{proof}

We note the following useful corollary of the proof of Theorem \ref{thm:deformation rings principal series}.
\begin{cor}
  \label{cor: a weight if and only if a weight after base
    change}Assume that $\rhobar$ is generic, and that $\sigmabar$ is a
  Serre weight. Then $\sigmabar\in\cD(\rhobar)$ if and only if
  $\BC(\sigmabar)\in\cD(\rhobar|_{G_{L_v}})$.
\end{cor}

\begin{rem}
  \label{rem:two-constructions-in-PS-irreducible}
  Observe that in the principal series case we have given \emph{two}
  constructions of the space $X\bigl(\tau\bigr)$: once in
  Lemma~\ref{lem:first-steps} (i.e.\ once
  from~\cite[Thm.~5.2.1]{BreuilMezardRaffinee}) and then again by the
  base change argument in the proof of Theorem~ \ref{thm:deformation
    rings principal series} (i.e.\ by applying
  Lemma~\ref{lem:first-steps} to the case of principal series
  types over the quadratic extension).  We caution that we
  have not checked that these two constructions give precisely the
  same formal variables $X_j, Y_j$ (rather than differing by a unit).
\end{rem}

\section{Breuil's lattice conjecture}\label{sec: proof of the
  conjecture}

We now use the results of the preceding sections to prove the main
conjecture of \cite{breuillatticconj}, as well as its natural analogue
for cuspidal representations.

\subsection{Patched modules of lattices}
\label{subsec:patched modules of lattices}
Assume that $p>3$. Let $F_v$ be an unramified extension of $\Qp$, let
$\rhobar:G_{F_v}\to\GL_2(\F)$ be a generic representation, and let
$M_\infty$ be a (possibly fixed determinant) patching functor indexed by a set
of places and representations which includes $(F_v,\rhobar)$. 

Let $\tau$ be a non-scalar tame inertial type for $G_{F_v}$,
which we assume is either principal series or regular cuspidal.
Recall from Section~\ref{subsec: specific lattices and gauges} that for each
subset $J\in\cP_\tau$ we have fixed a lattice $\sigma^\circ_J(\tau)$
in $\sigma(\tau)$ with the property that the cosocle of
$\sigmabar^\circ_J(\tau)$ is precisely $\sigmabar_J(\tau)$.
We have also fixed a subset
$\Jbase\subseteq\cS$ and defined the involution $\base(J):=J\triangle\Jbase$
on the subsets of $\cS$; if $\tau$ is principal series, then
$\Jbase=\emptyset$ and $\base$ is just the
identity.

We assume that we have fixed inertial types at all places in the
indexing set for $M_\infty$ other than $v$, and we drop these places from the
notation from now on, so that,
if $L$ denotes the tensor product of
the fixed inertial types at all places in the indexing set other than $v$,
then we simply write (for example)
$M_\infty(\sigma^\circ_J(\tau))$,
rather than 
$M_{\infty}(\sigma^{\circ}_J(\tau)\otimes L).$

It follows from Theorem~\ref{thm: the output of gauges.tex} that if
$\base(J) \in \cP_{\tau}$ then we have
inclusions \[p^{|J|}\sigma^\circ_{\base(J)}(\tau)\subseteq\sigma^\circ_{\base(\emptyset)}(\tau)\subseteq\sigma^\circ_{\base(J)}(\tau),\]and
therefore
inclusions  \[p^{|J|}M_\infty(\sigma^\circ_{\base(J)}(\tau))\subseteq
M_\infty(\sigma^\circ_{\base(\emptyset)}(\tau))\subseteq M_\infty(\sigma^\circ_{\base(J)}(\tau)).\]

Since $p>3$ and $\rhobar$ is assumed to be generic, we can apply
Theorem \ref{thm:deformation rings principal series},
so we have
subsets $\Jmin\subseteq\Jmax \subseteq \cS$, and elements $X_j,Y_j$ of
$R_v^{\square,\tau}$ for each $j\in\Jmax\setminus\Jmin$, which we
regard as elements of $R^\tau_\infty$ via the natural map
$R_v^{\square,\tau}\to R^\tau_\infty$. It will be convenient for us to
also consider the
subsets \[\Jmin':=(\Jbase^c\cap\Jmin)\cup(\Jbase\cap\Jmax^c),\
\Jmax':=(\Jbase^c\cap\Jmax)\cup(\Jbase\cap\Jmin^c).\] Note that we
have $\Jmin\subseteq\base(J)\subseteq\Jmax$ if and only if
$\Jmin'\subseteq J\subseteq\Jmax'$, and that if $\tau$ is principal
series, then we have $\Jmin'=\Jmin,\Jmax'=\Jmax$.  Observe also that
in all cases we have $\Jmax' \setminus \Jmin' = \Jmax \setminus \Jmin$.

For each $j \in \cS$, we define an element $\varpi_j \in R^\tau_\infty$ according to the prescription
$$\varpi_j := \begin{cases} p & \text{ if } j \in \Jmin' \\
Y_j & \text{ if } j \in (\Jmax' \setminus \Jmin') \cap \Jbase\\  
X_j & \text{ if } j \in (\Jmax' \setminus \Jmin') \cap \Jbase^c \\
1 & \text{ if } j \not\in \Jmax' . \end{cases}
$$
If $J \in\base(\cP_\tau)$, then we define $\varpi_J := \prod_{j \in J} \varpi_{j}$.

Similarly,
for each $j \in \cS$ we define $\varpi_j' \in R^\tau_\infty$ according to the prescription
$$\varpi_j' := \begin{cases} 1 & \text{ if } j \in \Jmin' \\
X_j & \text{ if } j \in (\Jmax' \setminus \Jmin') \cap \Jbase\\
Y_j & \text{ if } j \in (\Jmax' \setminus \Jmin') \cap \Jbase^c \\
p & \text{ if } j \not\in \Jmax' , \end{cases}$$
and if $J\in\base(\cP_\tau)$, then we define $\varpi'_J := \prod_{j \in J} \varpi'_{j}$.
Note that $\varpi_J \varpi'_J = p^{|J|}.$

\begin{prop}
\label{TG:prop:annihilation}
Suppose that $\tau$ is either a principal series inertial type or a regular
cuspidal inertial type. For each $J\in\base(\cP_\tau)$, we have equalities
$$\varpi_{J}M_\infty\bigl(\sigma_{\base(J)}^{\circ}(\tau)\bigr) =
M_\infty\bigl(\sigma^{\circ}_{\base(\emptyset)}(\tau)\bigr)$$
and
$$\varpi'_{J} M_\infty\bigl(\sigma^{\circ}_{\base(\emptyset)}(\tau)\bigr)
=p^{|J|} M_\infty\bigl(\sigma_{\base(J)}^{\circ}(\tau)\bigr).$$
\end{prop}
\begin{proof}
To begin with, fix
$J$ and $j\in J^c$ with the property that
$\base(J),\base(J\cup\{j\})\in\cP_\tau$. By Theorem~\ref{thm: the output of gauges.tex}(4) there is an inclusion of lattices $ \sigma_{\base(J)}^{\circ}(\tau)
\subseteq \sigma_{\base(J\cup \{j\})}^{\circ}(\tau),$ 
and the
cokernel of this inclusion is a successive extension of the weights $\sigmabar_{\base(J')}(\tau)$,
where $J'$ runs over subsets of $\cS$ such that $j \in J'$ and
$\base(J') \in \cP_{\tau}$.
Note also that this cokernel is annihilated by $p$.

The cokernel
$M_\infty\bigl(\sigma_{\base(J\cup \{j\})}^{\circ}(\tau)/\sigma_{\base(J)}^{\circ}(\tau)\bigr)$ of
\numequation
\label{TG:eqn:J embedding}
M_\infty\bigl(\sigma_{\base(J)}^{\circ}(\tau)\bigr)
\hookrightarrow M_\infty\bigl(\sigma_{\base(J\cup \{j\})}^{\circ}(\tau)\bigr) 
\end{equation}
is thus a successive extension of the corresponding modules
$M_\infty(\sigmabar_{\base(J')}(\tau))$, and is annihilated by $p$, and so
its scheme-theoretic support is generically reduced, with underlying
reduced subscheme equal to
$\bigcup_{\Jmin' \subseteq J' \subseteq \Jmax'\atop j
  \in J'} \overline{X}\bigl(\sigmabar_{\base(J')}(\tau)\bigr).$
Since the patched module
$M_\infty\bigl(\sigma_{\base(J\cup \{j\})}^{\circ}(\tau)/\sigma_{\base(J)}^{\circ}(\tau)\bigr)$
is Cohen--Macaulay, by the
definition of a patching functor, it follows from 
\cite[Thm.~17.3(i)]{MR1011461} that its scheme-theoretic support has no embedded associated primes,
and so is in fact equal to
$$\bigcup_{\Jmin' \subseteq J' \subseteq \Jmax'\atop j
  \in J'} \overline{X}\bigl(\sigmabar_{\base(J')}(\tau)\bigr).$$

If $j\notin \Jmax'$, we see that this support is trivial, and
hence this cokernel is trivial, i.e.\ is annihilated by multiplication
by $1$.  If $j \in \Jmax' \setminus \Jmin'$, then by Theorem~\ref{thm:deformation rings principal series}(4) this cokernel is
annihilated by multiplication by $Y_j$ if $j \in \Jbase$ and by $X_j$
if $j \in \Jbase^c$.
If $j \in \Jmin'$, then this
cokernel is annihilated by $p$.  Thus in all cases, we find that the
cokernel of~(\ref{TG:eqn:J embedding}) is annihilated by $\varpi_j$.

Similarly, we deduce that the cokernel of
$p M_\infty\bigl(\sigma_{\base(J\cup\{j\})}^{\circ}(\tau)\bigr)
\hookrightarrow  M_\infty\bigl(\sigma_{\base(J)}^{\circ}(\tau)\bigr)$
is annihilated by $\varpi'_j,$
 since the cokernel of the inclusion
$p\sigma_{\base(J\cup\{j\})}^{\circ}(\tau)\subseteq\sigma_{\base(J)}^{\circ}(\tau)
$
is a successive extension of the weights $\sigmabar_{\base(J')}(\tau)$,
where $J'$ runs over the subsets of $\cS$ such that $j \notin J'$.

We conclude that there are inclusions
\numequation
\label{TG:eqn:J inclusion}
\varpi_j M_\infty\bigl(\sigma_{\base(J\cup \{j\})}^{\circ}(\tau)\bigr)
\subseteq  M_\infty\bigl(\sigma_{\base(J)}^{\circ}(\tau)\bigr)
\end{equation}
and
\numequation
\label{TG:eqn:2nd J inclusion}
\varpi'_jM_\infty\bigl(\sigma_{\base(J)}^{\circ}(\tau)\bigr)
\subseteq p M_\infty\bigl(\sigma_{\base(J\cup\{j\})}^{\circ}(\tau)\bigr).
\end{equation}
Multiplying~(\ref{TG:eqn:J inclusion}) by $\varpi'_j$, we find that also
\[p M_\infty\bigl(\sigma_{\base(J\cup\{j\})}^{\circ}(\tau)\bigr)
\subseteq \varpi'_jM_\infty\bigl(\sigma_{\base(J)}^{\circ}(\tau)\bigr) ,\]
and hence that~(\ref{TG:eqn:2nd J inclusion}) is in fact an equality.
Multiplying this equality
by $\varpi_j$ (and recalling that $M_\infty\bigl(\sigma^{\circ}(\tau)\bigr)$ is $p$-torsion
free),
we deduce that~(\ref{TG:eqn:J inclusion}) is also an equality.
The statement of the Proposition now follows by induction on
$|J|$, using Lemma~\ref{lem:chains in cP_tau}.
\end{proof}

\subsection{Breuil's lattice conjecture}
We now apply Proposition~\ref{TG:prop:annihilation}
to  the particular fixed
determinant patching functors constructed in Section~\ref{subsec:
  existence of patching functors}. (Of course, it could be applied in
a similar way to the unitary group patching functors constructed in Section~\ref{subsec: unitary patching
  functors}.) Continue to fix a prime $p>3$, let
$F$ be a totally real field in which $p$ is unramified, and fix a
continuous representation $\rhobar:G_F\to\GL_2(\F)$ with the property
that $\rhobar|_{G_{F(\zeta_p)}}$ is absolutely irreducible. If $p=5$,
assume further that the projective image of
$\rhobar|_{G_{F(\zeta_5)}}$ is not isomorphic to $A_5$.
Fix a place
$v|p$, and assume that $\rhobar|_{G_{F_v}}$ is generic. Let $\tau$ be
some tame inertial type.

We now adopt the notation of Sections~\ref{subsec: existence of
  patching functors} and~\ref{subsec: limit over p power
  level}, fixing in particular a quaternion algebra $D$, a finite
set of finite places $S$, and the maximal ideal $\m \subset
\T^{S,\univ}$ corresponding to $\rhobar$. For each
place $w\in S\setminus\{v\}$, choose an inertial type $\tau_w$ and a
lattice $\sigma^\circ(\tau_w)$ in $\sigma(\tau_w)$, and let
$\sigma^v:=\otimes_{w\in S\setminus\{v\}}\sigma^\circ(\tau_w)$ be the corresponding
representation of $K^v=\prod_{w\in S\setminus\{v\}}K_w$. Let $\lambda:\T_\m\to
E'$ be the system of Hecke eigenvalues corresponding to
some minimal prime ideal of $\T_\m$,
where $\T_\m:=\T(\sigma^v\sigma^\circ(\tau))_\m$ in the notation of Section~\ref{subsec:
  limit over p power level} (note that this is independent of the
choice of lattice $\sigma^\circ(\tau)$), and where $E'$ is a finite
extension of $E$. Set $K_v=\GL_2(\cO_{F_v})$. Let
$\pi:=(M^v(\sigma^v)^*[1/p])[\lambda]$, which by
strong multiplicity one for $D^\times$ is an irreducible tame
representation of $\GL_2(F_v)$. Then the eigenspace
$M^v(\sigma^v)^*[\lambda]$ is a lattice in $\pi$, 
and the
intersection \[\sigma^\circ(\pi):=(\sigma(\tau)\otimes_{E}E')\cap
M^v(\sigma^v)^*[\lambda]\] inside $\pi$ defines a $K_v$-stable
$\cO_{E'}$-lattice in $\sigma(\tau) \otimes_E E'$.

Let $\rho :G_F\to\GL_2(\cO_{E'})$ be the Galois representation
associated to $\pi$.
We now define a $\GL_2(\cO_{F_v})$-stable $\cO_{E'}$-lattice $\sigma^\circ(\rho)$ in
$\sigma(\tau)\otimes_E E'$. Suppose first that $\tau$ is an
irregular cuspidal type. Then by Lemmas~\ref{lem: generic
  representations have no weight p-1} and~\ref{lem: regular cuspidals
  allow relabelling},
$\cD(\rhobar|_{G_{F_v}})\cap\JH(\sigmabar(\tau))$ contains a single
weight $\sigmabar_J(\tau)$, and we set
$\sigma^\circ(\rho):=\sigma^{\circ,J}(\tau)\otimes_\cO{\cO_{E'}}$.

Now suppose that $\tau$ is a principal series or regular cuspidal
type. For each $J\in\base(\cP_\tau)$, we let  $\varpi_{J}(\lambda)$ denote the image of $\varpi_{J}$ under
the composite of $\lambda$ and the homomorphism
$R_\infty^\tau\to\mathbb{T}_\m$ induced by the composite $R_\infty\to
R_S^\univ\to\mathbb{T}_\m$. Define
\[ \sigma^\circ(\rho):= \sum_{J \in\base(\cP_\tau)}
\varpi_{J}(\lambda)
\sigma_{\base(J)}^\circ(\tau).\]

\begin{theorem}\label{thm:Breuil lattice conjecture}
  Continue to maintain the above assumptions, so that in particular we
  have $p>3$, and $\rhobar|_{G_{F_v}}$ is generic. Then the lattices 
$\sigma^{\circ}(\pi)$ and $\sigma^\circ(\rho)$ are homothetic.
\end{theorem}
\begin{proof}Suppose first that $\tau$ is an
irregular cuspidal type. By Lemma~\ref{lem:uniqueness of lattices with
  irreducible socle} and the definition of $\sigma^\circ(\rho)$, we
must show that the only Serre weight occurring in the socle of
$\sigmabar^\circ(\pi)$ is $\sigmabar_J(\tau)$; but this is an
immediate consequence of Corollaries~5.6.4 and~4.5.7
of~\cite{geekisin} (the Buzzard--Diamond--Jarvis conjecture).

Suppose now that $\tau$ is either a principal series type or a regular
cuspidal type. By Proposition~\ref{prop:the gauge determines the lattice} we may write
$\sigma^{\circ}(\pi) = \sum_{J \in \base(\cP_\tau)} p^{v_J} \sigma^\circ_{\base(J)}(\tau),$
where $p^{v_J}$ denotes an element of $\cO_{E'}$ of some valuation $v_J$.
Rescaling if necessary, we further assume that $v_{\base(\emptyset)} =
0$. 

We
now apply the results of Subsection~\ref{subsec:patched modules of lattices}
to the fixed determinant patching
functor $M_\infty$ of Subsection~\ref{sec: abstract patching}.  It follows from~(\ref{eqn:U comparison patching section}) that for any subset $J \in\base(\cP_\tau)$
there is a natural isomorphism
$$\Hom_{K_v}\bigl(\sigma^{\circ}_{\base(J)}(\tau),\sigma^{\circ}(\pi)\bigr)^*
\iso M\bigl(\sigma_{\base(J)}^{\circ}(\tau)\bigr)/\lambda,$$
and so (noting that $M\bigl(\sigma_{\base(J)}^{\circ}(\tau)\bigr)$ is
$p$-torsion free) Proposition~\ref{TG:prop:annihilation} shows that
$$
\Hom_{K_v}\bigl(\sigma^{\circ}_{\base(J)}(\tau),\sigma^{\circ}(\pi)\bigr)
=\varpi_{J}(\lambda)  \Hom_{K_v}\bigl(\sigma^{\circ}_{\base(\emptyset)}(\tau),\sigma^{\circ}(\pi)\bigr).$$
Proposition~\ref{prop:the gauge determines the lattice} then
shows that $p^{v_J} =\varpi_{J}(\lambda)$
for each $J \in\base(\cP_\tau)$
(or more precisely,
that they have the same valuation, which is all we need).
\end{proof}

\begin{rem}
  In the case that $\tau$ is a principal series type and $D$ is a
  definite quaternion algebra, this is precisely Conjecture 1.2 of
  \cite{breuillatticconj}, under the mild additional assumptions that we
  have imposed in order to use the Taylor--Wiles--Kisin method.
  (See displayed equation~(13) in Section~3 of
  \cite{breuillatticconj} together with
  \cite[Cor.~5.3]{breuillatticconj}, and note that our normalisations
  differ from those of~\cite{breuillatticconj}.) 
 Note that while Breuil does not
  fix types at places other than $v$, the fact that
  Theorem~\ref{thm:Breuil lattice conjecture} shows that the homothety
  class of the lattice is independent of the choices of types away
  from $v$ means that Breuil's apparently more general conjecture is an
  immediate consequence of our result.
\end{rem}

\section{The generic Buzzard--Diamond--Jarvis conjecture}\label{sec:
  BDJ}
\subsection{The support of a patching functor} We now explain how the results of the earlier sections
can be applied to the weight part of Serre's conjecture, by an
argument analogous to the arguments made to avoid the use of Ihara's
lemma in \cite{tay}.  Fix an unramified extension $F_v/\Qp$, and a
continuous representation $\rhobar:G_{F_v}\to\GL_2(\F)$. For
simplicity, we consider (fixed determinant) patching functors indexed
by just the single field $F_v$.
\begin{thm}
  \label{thm: generic BDJ in patching language} Suppose that $p>3$ and
  that $\rhobar$ is generic, let $M_\infty$ be a fixed determinant
  patching functor for $\rhobar$, and let $\sigmabar$ be a Serre
  weight. Then $M_\infty(\sigmabar)\ne 0$ if and only if
  $\sigmabar\in\cD(\rhobar)$.
\end{thm}
\begin{proof}First, suppose that $M_\infty(\sigmabar)\ne 0$. Since
  $M_\infty(\sigmabar)$ is supported on
  $\Xbar^\psi\bigl(\sigmabar\bigr)$ by the definition of a fixed
  determinant patching functor, we have
  $\Xbar^\psi\bigl(\sigmabar\bigr)\ne 0$. By Theorem
  \ref{thm:geometric BM for types}, this implies that
  $\sigmabar\in\cD(\rhobar)$.
  
We now prove the converse. Let $\tau$ be an inertial type such that
$\cD(\rhobar) \subset \JH(\sigmabar(\tau))$, as
in Proposition \ref{prop: types for existence}. We certainly have
$M_\infty(\sigmabar')\ne 0$ for some Serre weight $\sigmabar'$, and
this implies that $M_\infty(\sigma^\circ(\tau))\ne 0$ (because we have
shown that $\sigmabar'\in\cD(\rhobar)$, so that $\sigmabar'$ is a
Jordan--H\"older factor of $\sigmabar^\circ(\tau)$). By Theorem
\ref{thm:geometric BM for types} and the choice of
$\tau$, it suffices to show that $M_\infty(\sigmabar^\circ(\tau))$ is
supported on all of $\Xbar^\psi\bigl(\tau\bigr)$. 

In order to see this, note that by Theorem \ref{thm:deformation rings
  principal series} the generic fibre of $X^\psi\bigl(\tau\bigr)$ is
irreducible, so that $M_\infty(\sigma^\circ(\tau))$ is supported on
the whole generic fibre of $X^\psi\bigl(\tau\bigr)$. Since
$M_\infty(\sigma^\circ(\tau))$ is maximal Cohen--Macaulay over
$X^\psi\bigl(\tau\bigr)$, it follows (\emph{cf.}\ Lemmas 2.2(1) and 2.3 of
\cite{tay}) that $M_\infty(\sigmabar^\circ(\tau))$ is supported on all
of $\Xbar^\psi\bigl(\tau\bigr)$, as required.
\end{proof}
\begin{rem}
  Of course the analogous result also holds for patching functors
  without fixed determinant, and can be proved in the same way. The
  same analysis also goes through for patching functors indexed by a
  collection of fields, provided that at the fields of residue
  characteristic $p$, the fields are absolutely unramified over $\Qp$,
  and the Galois representations are generic.

Applying Theorem \ref{thm: generic BDJ in patching language} to the
specific fixed determinant patching functors constructed in Section
\ref{sec: cohomology} recovers cases of the Buzzard--Diamond--Jarvis
conjecture (\cite{bdj}). Of course, as written this proof depends on
the results of \cite{GLS12}, \cite{geekisin}, etc. which were used in
Section~\ref{sec: cohomology} to construct our patching functors, but
it is possible to prove enough properties of the patching functors to
recover Theorem \ref{thm: generic BDJ in patching language} without
using these results, and in particular without using the potential
diagonalizability of potentially Barsotti--Tate representations proved
in \cite{kis04}, \cite{MR2280776} --- see Section~\ref{subsec:avoiding
  GK} below.
\end{rem}

\section{Freeness and multiplicity one}\label{sec: freeness and
  multiplicity one}

In this section we examine multiplicity one questions for cohomology
with coefficients in the lattices under consideration.

\subsection{Freeness of patched modules}
Fix $F_v$, an
unramified extension of $\Qp$, and fix a continuous representation
$\rhobar:G_{F_v}\to\GL_2(\F)$. Let $M^{\min}_\infty$ be a minimal
fixed determinant patching functor with unramified coefficients,
indexed by $(F_v,\rhobar)$.   As noted in Subsection~\ref{subsec: specific
lattices and gauges}, the lattices $\sigma(\tau)^{\circ}_J$ may be
defined over the ring of integers in
an unramified extension of $\Q_p$, and we take our unramified
coefficient field $E$ to be large enough for all these lattices to be defined
over its ring of integers $\cO$.  
Our main theorem is then the
following.

\begin{thm}
  \label{thm: freeness of the sigma_J^circ}
Assume that $p>3$ and that $\rhobar$ is
  generic, and let $\tau$ be a non-scalar tame inertial type for $I_{F_v}$. Fix
  some $J\in\cP_\tau$. 
 Then the coherent sheaf
  $M^{\min}_\infty(\sigma^\circ_{J}(\tau))$, regarded as a finitely
  generated module over $R^{\psi,\tau}_\infty$, is free of rank one.
\end{thm}
\begin{rem}
  As usual, the analogous result also holds for minimal patching
  functors without fixed determinant, with an essentially identical proof.
\end{rem}

In the case when the local deformation space $X^\psi\bigl(\tau\bigr)$ is regular, such a
result follows from the method of Diamond \cite{MR1440309}, and indeed
we will use Diamond's method as an ingredient in our proof via
Lemma~\ref{lem: abstract version of Diamond freeness argument}.
However, the deformation spaces $X^\psi\bigl(\tau\bigr)$ are typically {\em not}
regular, and so more work is required than solely applying this
method.

By Theorem \ref{thm:deformation rings principal
    series},
 there are subsets $\Jmin
  \subseteq\Jmax\subseteq\cS$ with
  $\Jmin,\Jmax\in\cP_\tau$ such that $R^{\psi,\tau}_\infty$ is a formal
  power series ring over
  \[\cO\llbracket (X_j,Y_j)_{j\in\Jmax\setminus\Jmin}\rrbracket/(X_jY_j-p)_{j\in\Jmax\setminus\Jmin}.\]
 Since $M^{\min}_\infty$ is assumed to have unramified coefficients,
this ring  is regular if and only if $|J_{\max}\setminus J_{\min}| 
\leq 1$. In addition, the Jordan--H\"older
  factors of $\sigmabar(\tau)$ which are elements of $\cD(\rhobar)$ are
  precisely the $\sigmabar_{J'}(\tau)$ with $\Jmin\subseteq
  J'\subseteq\Jmax$, and the equations $X_j=0$ ($j\in J'\setminus\Jmin$)
  and $Y_j=0$ ($j\in \Jmax\setminus J'$) cut out the component
  $\Xbar^\psi\bigl(\sigmabar_{J'}(\tau)\bigr)$ of $\Xbar^\psi\bigl(\tau\bigr)$.

  For the remainder of this section we fix $J$ as in Theorem~\ref{thm:
    freeness of the sigma_J^circ}.  It will be convenient for us to reindex the $\sigmabar_{J'}(\tau)$
  and the $X_j,Y_j$ as follows.
  Set
  $\sigmabar'_{J'}(\tau):=\sigmabar_{J^c\triangle J'}(\tau)$, so that in
  particular $\sigmabar'_{\cS}(\tau)=\sigmabar_J(\tau)$. Set
  \[\Jmin':=(J\cap\Jmin)\cup(J^c\cap\Jmax^c),\
  \Jmax':=(J\cap\Jmax)\cup(J^c\cap\Jmin^c).\] As in Section~\ref{sec: proof of the
  conjecture}, note that $\Jmin' \subset J' \subset
\Jmax'$ if and only if $\Jmin \subset  J^c \triangle  J' \subset \Jmax$, and that
  $\Jmax'\setminus\Jmin'=\Jmax\setminus\Jmin$. For each 
  $i\in\Jmax'\setminus\Jmin'$, we set $X'_i=X_i,Y'_i=Y_i$ if $i\in
  J$, and $X'_i=Y_i,Y'_i=X_i$ if $i\notin
 J$, so that the equations $X'_j=0$ ($j\in
  J'\setminus\Jmin'$) and $Y'_j=0$ ($j\in \Jmax'\setminus J'$) 
 cut out the
  component $\Xbar^\psi_\infty\bigl(\sigmabar'_{J'}(\tau)\bigr)$ of
  $\Xbar^\psi_\infty\bigl(\tau\bigr)$.

In fact, we will be more interested in the special fibre
$\Rbar^{\psi,\tau}_\infty,$ where $\Rbar^{\psi,\tau}_\infty =
R^{\psi,\tau}_\infty/p R^{\psi,\tau}_\infty$ is a formal power series
ring over
\[
\F\llbracket (X'_j,Y'_j)_{j\in\Jmax'\setminus\Jmin'}\rrbracket
/(X'_jY'_j)_{j\in\Jmax'\setminus\Jmin'}.
\]

We let $\cW := \{J' \, | \, \Jmin' \subseteq J' \subseteq
\Jmax'\}$.
If $\cJ \subseteq \cW$, then we let $I_{\cJ}$ denote the radical ideal
in $\Rbar^{\psi,\tau}_\infty$ which cuts out the reduced induced structure on the
closed subset $\cup_{J' \in \cJ} \Xbar^\psi_\infty\bigl(\sigmabar'_{J'}(\tau)\bigr).$  
If $\cJ_1, \cJ_2 \subseteq \cW$, then evidently $I_{\cJ_1 \cup \cJ_2} = I_{\cJ_1}
\cap I_{\cJ_2}$.  On the other hand, while there is an inclusion $I_{\cJ_1} + I_{\cJ_2}
\subseteq I_{\cJ_1 \cap \cJ_2},$ it is not generally an equality (although
it is an equality generically on $\Xbar^\psi_\infty\bigl(\tau\bigr)$).

\begin{example}
\label{ex:non-generic failure}
Suppose that $|\Jmax' \setminus \Jmin'| = 2$, and
write $\Jmax' \setminus \Jmin' = \{j_1,j_2\}.$
\begin{enumerate}
\item
If we let $\mathcal J_1 := \{\Jmax'\}, \mathcal J_2 := \{\Jmin'\},$
then we have
$I_{\mathcal J_1} = (X'_{j_1}, X'_{j_2}), I_{\mathcal J_2} = (Y'_{j_1},Y'_{j_2}),$
and $I_{\mathcal J_1} + I_{\mathcal J_2} = (X'_{j_1},X'_{j_2},Y'_{j_1},Y'_{j_2}),$
while $I_{\mathcal J_1 \cap \mathcal J_2} = I_{\emptyset}$ is the unit
ideal.
\item 
If we let $\mathcal J_1 := \{\Jmax', \Jmax' \setminus \{j_2\}\},
\mathcal J_2 := \{\Jmin'\cup \{j_2\}, \Jmin'\},$
then $I_{\mathcal J_1} = (X'_{j_1}), I_{\mathcal J_2} = (Y'_{j_1})$, and
$I_{\mathcal J_1} + I_{\mathcal J_2} = (X'_{j_1},Y'_{j_1}),$
while $I_{\mathcal J_1 \cap \mathcal J_2} = I_{\emptyset}$ is again the unit
ideal.
\item
If we let $\mathcal J_1 := \{\Jmax',\Jmax'\setminus \{j_2\}\},
\mathcal J_2 := \{\Jmax',
\Jmin'\},$ 
then we have
$I_{\mathcal J_1} = (X'_{j_1}), I_{\mathcal J_2} = (X'_{j_1}Y'_{j_2},X'_{j_2}Y'_{j_1}),$
and $I_{\mathcal J_1} + I_{\mathcal J_2} = (X'_{j_1},X'_{j_2}Y'_{j_1}),$
while $I_{\mathcal J_1 \cap \mathcal J_2} = I_{\{\Jmax'\}} = (X'_{j_1},X'_{j_2}).$
\end{enumerate}
\end{example}

\begin{defn} 
We say that a subset $\cJ \subseteq \cW$ is an {\em interval}
if whenever $J_1 \subseteq J' \subseteq J_2$ with $J_1, J_2 \in \cJ$,
then $J' \in \cJ$.  We say that $\cJ$ is a {\em capped interval}
if it contains a unique maximal element, which we then refer to as the
{\em
cap} of $\cJ$.
\end{defn}
\begin{defn}
\label{def:faces}
If $J_1 \subseteq J_2$ are elements of $\cW$,
then we define
$$\cF(J_1,J_2) := \{J' \in \cW \, | \, J_1 \subseteq J' \subseteq J_2\}$$
and
$$\cF(J_1,J_2)^{\times} := \cF(J_1,J_2) \setminus \{J_1\}
= \{J' \in \cW \, | \, J_1 \subsetneq J' \subseteq J_2\}.$$
\end{defn}

\begin{lemma}
\label{lem:faces}
The quotient $I_{\cF(J_1,J_2)^{\times}}/I_{\cF(J_1,J_2)}$
is isomorphic to $\Rbar^{\psi,\tau}_\infty/I_{\{J_1\}}$
{\em (}and so in particular is cyclic{\em )},
and is generated by the image of the element
$\prod_{j \in J_2\setminus J_1}X'_j.$
\end{lemma}
\begin{proof}
The ideals $I_{\cF(J_1,J_2)}$ and $I_{\cF(J_1,J_2)^{\times}}$ admit simple descriptions,
namely
$$I_{\cF(J_1,J_2)} = (\{X'_j\}_{j \in J_1 \setminus \Jmin'},
\{Y'_j\}_{j\in \Jmax' \setminus J_2}),$$
while
$$I_{\cF(J_1,J_2)^{\times}} = \bigl(\{X'_j\}_{j \in J_1 \setminus \Jmin'},
\{Y'_j\}_{j\in \Jmax' \setminus J_{2}},
\prod_{j \in J_{2}\setminus J_1}  X'_j)
.$$
Thus 
$I_{\cF(J_1,J_2)^{\times}}/I_{\cF(J_1,J_2)}$
is generated by 
$\prod_{j \in J_2\setminus J_1}X'_j$,
as claimed.
Furthermore, this quotient is supported precisely on the component
$\Xbar\bigl(\sigmabar'_{J_1}(\tau)\bigr),$ and so its annihilator is $I_{\{J_1\}}.$
\end{proof}

\begin{remark}
The preceding lemma can be rephrased more geometrically, 
namely that the quotient $I_{\cF(J_1,J_2)^{\times}}/(I_{\cF(J_1,J_2)^{\times}}
\cap I_{\{J_1\}}) \cong
(I_{\cF(J_1,J_2)^{\times}}
+ I_{\{J_1\}})/ I_{\{J_1\}}$ is principal, which is to say that the component
$\Xbar_{\infty}^{\psi}\bigl(\sigmabar'_{J_1}(\tau)\bigr)$
intersects the union of components
$$\bigcup_{J' \in \cF(J_1,J_2)^{\times}}
\Xbar_{\infty}^{\psi}\bigl(\sigmabar'_{J'}(\tau)\bigr)$$
in a Cartier divisor on
$\Xbar_{\infty}^{\psi}\bigl(\sigmabar'_{J_1}(\tau)\bigr)$.
\end{remark}

\begin{lemma}
\label{lem:ideals}
If $\cJ_1$ and $\cJ_2$ are two capped intervals in $\cW$
that share a common cap,
then $I_{\cJ_1} + I_{\cJ_2} = I_{\cJ_1 \cap \cJ_2}.$
\end{lemma}
\begin{proof}
Write $\cJ := \cJ_1 \cap \cJ_2$, and let $J_{\rmcap}$ denote
the common cap of $\cJ_1$ and $\cJ_2$.
If $\cJ = \cJ_1$ then the statement
of the lemma is trivial, and so we may assume that $\cJ \subsetneq \cJ_1$.
We begin by assuming in addition that $| \cJ_1 \setminus \cJ_2| = 1$, and
write $\cJ_1 \setminus \cJ_2 = \{J_1\}.$

The evident inclusion
$I_{\cJ_2} \subseteq I_{\cJ}$ implies that
$$(I_{\cF(J_1,J_{\rmcap})} + I_{\cJ_2}) \cap
I_{\cJ} = (I_{\cF(J_1,J_{\rmcap})} \cap I_{\cJ}) + I_{\cJ_2} = I_{\cF(J_1,J_{\rmcap})\cup \cJ} + I_{\cJ_2}
= I_{\cJ_1} + I_{\cJ_2}.$$
Since furthermore $I_{\cJ} \subseteq I_{\cF(J_1,J_{\rmcap})^{\times}}$, we see that
in order to deduce the lemma in the case we are considering, it suffices to
prove that
$$I_{\cF(J_1,J_{\rmcap})} +I_{\cJ_2} = I_{\cF(J_1,J_{\rmcap})^{\times}}$$
(that is, it suffices to prove the result in the case that
$\cJ_1 = \cF(J_1,J_{\rmcap})$),
or equivalently, it suffices to prove that $I_{\cJ_2}$ surjects onto
$I_{\cF(J_1,J_{\rmcap})^{\times}}/I_{\cF(J_1,J_{\rmcap})}$
under the quotient map
$\Rbar^{\psi,\tau}_\infty \to \Rbar^{\psi,\tau}_\infty/I_{\cF(J_1,J_{\rmcap})}.$
Now Lemma~\ref{lem:faces} shows that the quotient
$I_{\cF(J_1,J_{\rmcap})^{\times}}/I_{\cF(J_1,J_{\rmcap})}$
is cyclically generated by (the image of)
$\prod_{j \in J_{\rmcap}\setminus J_1}X'_j$,
and so it suffices to note that this element also lies in $I_{\cJ_2}$,
since $J_1 \not\in \cJ_2$.

We now proceed by induction on $| \cJ_1 \setminus \cJ_2|$, assuming
that it is greater than $1$.
Let $J_1$ be a maximal element of $\cJ_1\setminus \cJ_2$, and write
$\cJ' = \cJ\cup \{J_1\}$ and $\cJ_2' = \cJ_2\cup\{J_1\}$. 
The maximality of $J_1$ assures us that
$\cJ'$ and $\cJ'_2$ are again both intervals.  Since $\cJ_1 \cap \cJ_2' = \cJ'$,
we may assume by induction that $$I_{\cJ_1} + I_{\cJ_2'} = I_{\cJ'}.$$
Adding $I_{\cJ_2}$ to both sides of this equality yields
$$I_{\cJ_1} + I_{\cJ_2} = I_{\cJ'} + I_{\cJ_2}.$$
Thus we are reduced to proving that $I_{\cJ'} + I_{\cJ_2} = I_{\cJ}.$
But $|\cJ'\setminus \cJ_2| = | \cJ' \setminus \cJ| = 1,$
and so we have reduced ourselves
to the situation already treated. This completes the proof of the lemma.
\end{proof}  

\begin{remark}
We remind the reader that as the examples
of~\ref{ex:non-generic failure} show, some sort of hypothesis on
$\cJ_1$ and $\cJ_2$ as in Lemma~\ref{lem:ideals} is necessary for that
lemma to hold.
\end{remark}

\begin{remark}
Note that
Theorem~\ref{thm: freeness of the sigma_J^circ} implies
Lemma~\ref{lem:ideals}. To see this, apply Theorem~\ref{thm: freeness
  of the sigma_J^circ} to the set $J'' \in \cP_{\tau}$ such that $J^c \triangle J''$ is the common cap of $\cJ_1$
and $\cJ_2$.  Then there are submodules $M_1,M_2$ of
$\sigmabar^\circ_{J''}(\tau)$ such that the Jordan--H\"older factors of
$\sigmabar^\circ_{J''}(\tau)/M_i$ are precisely the weights
$\sigmabar'_{J'}(\tau)$ with $J'\in\cJ_i$. Then $I_{\cJ_1}$ is the annihilator of
$\Mmin_\infty(\sigmabar^\circ_{J''}(\tau)/M_1)$, $I_{\cJ_2}$ is the
annihilator of $\Mmin_\infty(\sigmabar^\circ_{J''}(\tau)/M_2)$, and
$I_{\cJ_1\cap\cJ_2}$ is the annihilator of
$\Mmin_\infty(\sigmabar^\circ_{J''}(\tau)/(M_1+M_2))$, and the result is immediate.
\end{remark}

Before we begin the proof of  Theorem~\ref{thm: freeness of the
  sigma_J^circ}, we check the following proposition.

\begin{prop}
  \label{prop: when there are extensions of Serre weights, and
    existence of types when this happens}
Suppose that $\rhobar$ is generic, and that for some
    $\sigmabar,\sigmabar'\in\cD(\rhobar)$, there is a non-trivial
    extension $\thetabar$ of $\sigmabar$ by $\sigmabar'$. Then there
    is a tame inertial type $\tau'$ and a $K$-stable $\cO$-lattice
    $\sigma^\circ(\tau')$ in $\sigma(\tau')$ such that
    $\JH(\sigmabar(\tau'))\cap\cD(\rhobar)=\{\sigmabar,\sigmabar'\}$, and the
    extension $\thetabar$ is realised in the cosocle filtration of $\sigmabar^\circ(\tau')$.
\end{prop}
\begin{proof}
First, note that since
    $\Ext^1_{\F[\GL_2(k_v)]}(\sigmabar',\sigmabar)$ is one-dimensional
(by Proposition~\ref{prop:extensions of Serre weights}),
    it suffices to find $\sigmabar^\circ(\tau')$ which realises a non-split
    extension of $\sigmabar'$ by~$\sigmabar$.

    In fact, we claim that it suffices to find $\tau'$ such that
    $\JH(\sigmabar(\tau'))\cap\cD(\rhobar)=\{\sigmabar,\sigmabar'\}$. To see
    this, suppose that such a $\tau'$ exists, and take $\sigma^\circ(\tau')$ to
    be the lattice $\sigma(\tau')_{\sigmabar}$ provided by Lemma
    \ref{lem:uniqueness of lattices with irreducible socle}. It
    suffices to prove that $\sigmabar'$ necessarily lies in the first
    layer of the cosocle filtration of $\sigma^\circ(\tau')$ (since the
    extension of $\sigmabar$ by $\sigmabar'$ must be nontrivial, as
    the socle of $\sigma^\circ(\tau')$ is just $\sigmabar$); but
    this follows from Theorem~\ref{thm: socle filtration for tame type with irred socle}
 and Proposition~\ref{prop:extensions of Serre weights}.

    We now find such a type $\tau'$.
    Since
    $\cD(\rhobar)\subset\cD(\rhobar^{\textrm{ss}})$, we may assume
    that $\rhobar$ is semisimple. Without loss of generality we may
    assume that
    $\sigmabar=\sigmabar_{\vec{t},\vec{s}}$, $\sigmabar'=\sigmabar_{\vec{t'},\vec{s'}}$
    as in the statement of Proposition~\ref{prop:extensions of Serre
      weights}, so that there is some $k$ with $s_j=s'_{j}$ if $j\ne
    k, k+1$, $s'_k=p-2-s_k$, and $s_{k+1}'=s_{k+1}\pm 1$.
    Let $\rhobar':=\BC(\rhobar)$. By Corollary \ref{cor: a weight if and only if a weight after base
    change}, 
 it is enough to find a tame inertial type $\tau'$ such that if
  $\tau'':=\BC(\tau')$, then
  $\BC(\sigmabar),\BC(\sigmabar')$ are Jordan--H\"older factors of
  $\sigmabar(\tau'')$, and no other elements of $\JH(\sigmabar(\tau''))\cap\cD(\rhobar')$ are
  in the image of $\BC$. We will construct $\tau''$ directly as a principal
    series representation, and it will be immediate by the
    construction that $\tau''$ is in the image of $\BC$. 

The construction of such a  $\tau''$ (and thus of $\tau'$) is an easy
exercise; we sketch the details.
Suppose first that $\rhobar$ is
reducible, so that we may write  \[
\rhobar|_{I_{F_v}}\cong \omega_f^{\sum_{j=0}^{f-1}t_jp^j}
    \begin{pmatrix}
      \omega_f^{\sum_{j\in K}(s_j+1)p^j}&0\\0&\omega_f^{\sum_{j\notin K}(s_j+1)p^j}
    \end{pmatrix}\] for some subset $K\subseteq \CS$; then we let $K'=\BC_{\PS}(K)\subseteq\cS'$. If on the other hand
    $\rhobar$ is irreducible, then we may write  \[
\rhobar|_{I_{F_v}}\cong \omega_f^{\sum_{j=0}^{f-1}t_jp^j}
    \begin{pmatrix}
      \omega_{2f}^{\sum_{j\in K}(s_j+1)p^j}&0\\0&\omega_{2f}^{\sum_{j\notin K}(s_j+1)p^j}
    \end{pmatrix},\] and we let $K'=K\subseteq\cS'$.

 In either case we let
    $J'=K'\triangle\{k+1,\dots,k+f\}$, so that $J'$ is
    antisymmetric if $\rhobar$ is reducible, and symmetric if
    $\rhobar$ is irreducible. Define $\tau'':=I(\eta\otimes\eta')$, where
    $\eta'=[\cdot]^{\sum_{j=0}^{2f-1}x_jp^j}$,
    $\eta(\eta')^{-1}=[\cdot]^{\sum_{j=0}^{2f-1}c_jp^j}$,
    with \[x_{i}=
\begin{cases}
                 t_i+s_{i}+1-p & \text{if }i\in J' \\
                 t_i & \text{if }i\notin J',\end{cases}\] \[c_{i}=\begin{cases}
                 p-1-s_{i}-\delta_{(J')^c}(i-1)&\text{if }i\in J' \\
                 s_i+\delta_{J'}(i-1)&\text{if }i\notin J'. \end{cases}.\]
              Note by Lemma~\ref{lem: generic
               representations have no weight p-1} that each $c_i$ is
             in the range $[0,p-1]$; observe also that if $\rhobar$ is
             reducible (respectively irreducible) then there is a
             cuspidal (respectively principal series) inertial type $\tau'$ with
             $\BC(\tau')=\tau''$. 

             Now, by definition we have
             $\sigmabar_{J'}(\tau'')=\BC(\sigmabar)$, and
             $\sigmabar(\tau'')_{J'\triangle\{k,k+f\}}=\BC(\sigmabar')$.
             It is elementary to check (as an application of
             Proposition 4.3 of \cite{breuillatticconj}, and
             especially equation (19) from its proof) that there
             are precisely two other Jordan--H\"older factors of
             $\sigmabar(\tau'')$ which are contained in $\cD(\rhobar')$, namely
             $\sigmabar(\tau'')_{J'\triangle\{k\}}$ and
             $\sigmabar(\tau'')_{J'\triangle\{k+f\}}$, neither of which
             is in the image of $\BC$. The result follows.
\end{proof}
{\em Proof of Theorem}~\ref{thm: freeness of the sigma_J^circ}.
  Since $M^{\min}_\infty(\sigma^\circ_J(\tau))$ is a faithful
  $R^{\psi,\tau}_\infty$-module (as by Lemma~\ref{lem: abstract version of Diamond freeness argument}  the module $M^{\min}_\infty(\sigma^\circ_J(\tau))[1/p]$ is
  locally free 
over $R^{\psi,\tau}_\infty[1/p]$), the module $M^{\min}_\infty(\sigma^\circ_J(\tau))$ is free of
  rank one if and only if it is cyclic. By Nakayama's lemma,
  $M^{\min}_\infty(\sigma^\circ_J(\tau))$ is cyclic if and only if
  $M^{\min}_\infty(\sigmabar^\circ_J(\tau))$ is cyclic.

It follows easily from Theorem~\ref{thm: socle filtration for tame type with irred socle}
that for each interval $\cJ \subseteq \cW$, there 
is a subquotient $\sigmabar^{\cJ}$ of $\sigmabar^\circ_J(\tau)$ uniquely 
characterized by the property that its set of Jordan--H\"older factors
is precisely the set $\{\sigmabar'_{J'}(\tau)\}_{J' \in \cJ}.$

We will prove, by induction on $|\cJ|$, 
that for each capped interval $\cJ \subseteq \cW$,
the patched module $M^{\min}_\infty(\sigmabar^{\cJ})$ is cyclic,
with annihilator equal to $I_{\cJ}$.
(The assumption that $\cJ$ be {\em capped}
is crucial; see Remark~\ref{rem:the importance of being capped} below.)
The theorem will follow by taking $\cJ = \cW$;
indeed, although $M^{\min}_\infty(\sigmabar^{\cW})$ is {\em a priori}
a subquotient of $M^{\min}_\infty(\sigmabar^\circ_J(\tau))$, 
in fact the two are canonically isomorphic,
because $M^{\min}_\infty(\sigmabar'_{J'}(\tau))$ vanishes if $J'
\not\in \cW$, by Theorem \ref{thm:geometric BM for types}.

The induction begins with the following lemmas.
\begin{lemma}
\label{lem:first induction step}
If $\cJ$ is an interval with $|\cJ| \leq 2$, then $M^{\min}_\infty(\sigmabar^{\cJ})$ is cyclic,
with annihilator $I_{\cJ}$.
\end{lemma}
\begin{proof}
Since $\cJ$ is an interval with at most two elements,
it is either empty
(in which case $M^{\min}_\infty(\sigmabar^{\cJ}) = 0$ and the lemma is trivially true),
consists of a single element, or consists of two adjacent sets
$J_1$ and $J_2$. Suppose that $\cJ$ is non-empty.

Since $|\cJ|\le 2$, by Propositions~\ref{prop: types for elimination}
and~\ref{prop: when there are extensions of Serre weights,
  and existence of types when this happens}
we may
find an inertial type $\tau'$ such that $\JH(\sigmabar(\tau')) \cap \cD(\rhobar) = \{\sigmabar'_{J'}(\tau)\}_{J'\in\cJ}$. Since $X^\psi_\infty\bigl(\tau'\bigr)$ is regular,
$M^{\min}_\infty(\sigma^\circ(\tau'))$ is free of rank one over
$X^\psi_\infty\bigl(\tau'\bigr)$ for any choice of lattice $\sigma^\circ(\tau')$, by
Lemma~\ref{lem: abstract version of Diamond freeness argument}. Note
that by Theorem \ref{thm:geometric BM for types}, $I_\cJ$ cuts out the
reduced induced structure on the subset $\Xbar^\psi_\infty\bigl(\tau'\bigr)$ of $\Xbar^\psi_\infty\bigl(\tau\bigr)$.

We claim that we can choose $\sigma^\circ(\tau')$ so that
$M^{\min}_\infty(\sigmabar^{\circ}(\tau')) = M^{\min}_\infty(\sigmabar^{\cJ})$,
from which the lemma will follow.
In the case that
$\cJ$ is a singleton, this is trivial. In the case that
$\cJ=\{J_1,J_2\}$,
it follows from Proposition \ref{prop:
  when there are extensions of Serre weights, and existence of types
  when this happens} and
Theorem~\ref{thm: the output of gauges.tex}.
\end{proof}

We will also require the following lemma.

\begin{lem}
  \label{lem:cyclic extensions are cyclic}Let $R$ be a local ring, and
  let $M''\subsetneq M'\subseteq M$ be $R$-modules such that $M'$ and
  $M/M''$ are both cyclic. Then $M$ is cyclic.
\end{lem}
\begin{proof}
  Let $m$ be a lift to $M$ of a cyclic generator of $M/M''$. We claim
  that $M=Rm$. In order to see this, it suffices to show that
  $M''\subseteq Rm$, and thus it suffices to show that $M'\subseteq Rm$.

Choose $r\in R$ such that $m':=rm$ generates the cyclic submodule
$M'/M''$ of $M/M''$, so that $Rm'\subseteq M'$. Since $M'/M''$ is
non-zero, we see by Nakayama's lemma that $m'$ is in fact a generator
of $M'$, so $M'=Rm'$. Thus $M'=Rm'\subseteq Rm$, as required.
\end{proof}

We now complete the induction, using Lemma~\ref{lem:first induction step}
to provide the base case $|\cJ| \leq 2$. 
Suppose then that $|\cJ| > 2$.  We consider two cases: when
$\cJ$ has a unique minimal element, and when $\cJ$ has at least two minimal
elements.  

Suppose first that $\cJ$ contains a unique minimal element $J_0$, and let $J'$
be a minimal element of $\cJ\setminus \{J_0\}$.  Then we have (by Theorem~\ref{thm: socle filtration for tame type with irred socle}) inclusions
$\sigmabar^{\{J_0\}} \subsetneq \sigmabar^{\{J_0,J'\}} \subseteq 
\sigmabar^{\cJ},$ and an equality
$\sigmabar^{\cJ}/\sigmabar^{\{J_0\}} = \sigmabar^{\cJ\setminus\{J_0\}},$
which induce  (since $M^{\min}_\infty$ is an exact functor) inclusions
$M^{\min}_\infty(\sigmabar^{\{J_0\}}) \subsetneq M^{\min}_\infty(\sigmabar^{\{J_0,J'\}}) \subseteq 
M^{\min}_\infty(\sigmabar^{\cJ})$ (the first inclusion being strict since
$M^{\min}_\infty(\sigmabar^{\{J'\}}) \neq 0$, as $J' \in \cW$), and an equality
$M^{\min}_\infty(\sigmabar^{\cJ})/M^{\min}_\infty(\sigmabar^{\{J_0\}}) =
M^{\min}_\infty(\sigmabar^{\cJ\setminus\{J_0\}}).$
By induction, we know
that each of $M^{\min}_\infty(\sigmabar^{\{J_0,J'\}})$ and $M^{\min}_\infty(\sigmabar^{\cJ\setminus
\{J_0\}})$ is cyclic, and hence by Lemma~\ref{lem:cyclic extensions are cyclic}
we deduce that $M^{\min}_\infty(\sigmabar^{\cJ})$ is cyclic.

It remains to determine the annihilator of $M^{\min}_\infty(\sigmabar^{\cJ})$.
If we denote this annihilator by $I$, then certainly $I \subseteq I_{\cJ}$,
and $I_{\cJ\setminus \{J_0\}}/I \iso M^{\min}_\infty(\sigmabar'_{J_0}(\tau)) \iso \Rbar^{\psi,\tau}_\infty/I_{\{J_0\}}.$
Thus to show that $I = I_{\cJ},$ it suffices to prove that
$I_{\cJ \setminus\{J_0\}}/I_{\cJ} \iso \Rbar^{\psi,\tau}_\infty/I_{\{J_0\}}.$
For this, note that (letting $J_{\rmcap}$ denote the cap of $\cJ$)
\begin{multline*}
I_{\cJ\setminus \{J_0\}}/I_{\cJ}  =
I_{\cJ \setminus \{J_0\}}/ I_{\cJ\setminus \{J_0\}} \cap I_{\cF(J_0,J_{\rmcap})}
\isoto (I_{\cJ \setminus \{J_0\}}+I_{\cF(J_0,J_{\rmcap})})/I_{\cF(J_0,J_{\rmcap})}
\\
= I_{\cF(J_0,J_{\rmcap})^{\times}}/I_{\cF(J_0,J_{\rmcap})} \iso \Rbar^{\psi,\tau}_\infty/I_{\{J_0\}}
\end{multline*}
(the first and second equalities being evident, 
the third being an application of Lemma~\ref{lem:ideals},
and the isomorphism being provided by Lemma~\ref{lem:faces}),
as required.

Suppose now that $\cJ$ contains at (at least) two distinct minimal elements,
say $J_1$ and $J_2$.  Write $\cJ_1 := \cJ\setminus \{J_1\}$, $\cJ_2 := \cJ\setminus
\{J_2\}$.  Then
$\sigmabar^{\cJ}$ is naturally identified with
the fibre product of $\sigmabar^{\cJ_1}$
and $\sigmabar^{\cJ_2}$ over $\sigmabar^{\cJ_1\cap \cJ_2}$,
and so
$M^{\min}_\infty(\sigmabar^{\cJ})$ is naturally identified with
the fibre product of
$M^{\min}_\infty(\sigmabar^{\cJ_1})$
and $M^{\min}_\infty(\sigmabar^{\cJ_2})$ over $M^{\min}_\infty(\sigmabar^{\cJ_1\cap \cJ_2})$
(since $M^{\min}_\infty$ is an exact functor).
By induction, there are isomorphisms
$M^{\min}_\infty(\sigmabar^{\cJ_1}) \iso \Rbar^{\psi,\tau}_\infty/I_{\cJ_1}$,
$M^{\min}_\infty(\sigmabar^{\cJ_2})\iso \Rbar^{\psi,\tau}_\infty/I_{\cJ_2}$, and
$M^{\min}_\infty(\sigmabar^{\cJ_1\cap \cJ_2}) \iso \Rbar^{\psi,\tau}_\infty/I_{\cJ_1\cap \cJ_2}.$
From Lemma~\ref{lem:ideals} we know that $I_{\cJ_1\cap \cJ_2} = I_{\cJ_1} + I_{\cJ_2}$,
while clearly $I_{\cJ_1} \cap I_{\cJ_2} = I_{\cJ_1 \cup \cJ_2} = I_{\cJ}$.
Since  the fibre product
of $\Rbar^{\psi,\tau}_\infty/I_{\cJ_1}$ and $\Rbar^{\psi,\tau}_\infty/I_{\cJ_2}$ over
$\Rbar^{\psi,\tau}_\infty/(I_{\cJ_1} + I_{\cJ_2})$ is $\Rbar^{\psi,\tau}_\infty/(I_{\cJ_1}\cap I_{\cJ_2})$, we conclude that 
indeed $M^{\min}_\infty(\sigmabar^{\cJ})$ is isomorphic to $\Rbar^{\psi,\tau}_\infty/I_{\cJ}$,
completing the inductive step, and so completing the proof of the
theorem.
\qed

\begin{remark}
\label{rem:the importance of being capped}
The assumption in the preceding proof that $\cJ$ be a {\em capped} interval
is crucial for deducing that $M^{\min}_\infty(\sigmabar^{\cJ})$ is cyclic.
Indeed, if $\cJ$ admits two distinct maximal elements $J_1$ and $J_2$,
then $\sigmabar^{\cJ}$ admits a surjection onto $\sigmabar'_{J_1}(\tau) \oplus
\sigmabar'_{J_2}(\tau)$, and hence $M^{\min}_\infty(\sigmabar^{\cJ})$ admits a surjection
onto $M^{\min}_\infty(\sigmabar'_{J_1}(\tau))\oplus M^{\min}_\infty(\sigmabar'_{J_2}(\tau));$  
thus it cannot be cyclic.
\end{remark}

\begin{remark}
  \label{rem: tame patching is purely local}By
  Proposition~\ref{TG:prop:annihilation} and Theorem~\ref{thm:
    freeness of the sigma_J^circ}, we see that if the tame inertial
  type $\tau$ is either 
  principal series or regular cuspidal, then the
  restriction of $M^{\min}_\infty$ to the category of subquotients of
  lattices in $\sigma(\tau)$ is independent of the particular choice
  of
  $M^{\min}_\infty$.
  Indeed, we can explicitly define a particular choice of
  $M^{\min}_\infty$ with no formal variables in the following way: we
  fix some $J\in\cP_\tau$, and we let
  $M^{\min}_\infty(\sigma_J^\circ(\tau))$ be the structure sheaf of
  $R^{\psi,\tau}$. Then the other $M^{\min}_\infty(\sigma^{\circ}_{J'}(\tau))$
  are uniquely determined by Proposition~\ref{TG:prop:annihilation},
  and since any lattice in $\sigma(\tau)$ is in the span of the
  $\sigma^{\circ}_{J'}(\tau)$ by Proposition~\ref{prop:the gauge determines
    the lattice}, we see that this determines $M^{\min}_\infty$
  completely. Any other choice of $M^{\min}_\infty$ is then equivalent
  to one obtained from this one by base extension to
  $R^{\psi,\tau}_\infty$.
\end{remark}

\subsection{Application to a conjecture of Demb\'el\'e}
We now give a proof of a conjecture of Demb\'el\'e
(\cite{dembeleappendix}; see also the introduction
to~\cite{breuillatticconj}), by applying Theorem~\ref{thm: freeness of
  the sigma_J^circ} with $\Mmin_\infty$ the minimal fixed determinant
patching functor of Section~\ref{subsec:minimal level}, and $\tau$ a
principal series type. (Using the construction of Section~\ref{subsec:
  unitary patching functors}, a similar result could also be proved
for unitary groups.)

We put ourselves in the setting of Section~\ref{subsec: existence of
  patching functors}. In particular, we have a prime $p\ge 5$, a totally real field $F$
in which $p$ is unramified, and a modular representation
$\rhobar:G_F\to\GL_2(\F)$ such that $\rhobar|_{G_{F(\zeta_p)}}$ is
absolutely irreducible, and if $p=5$ then we assume further that the projective image of $\rhobar(G_{F(\zeta_5)})$ is not
  isomorphic to $A_5$. We also have a quaternion algebra $D$ with
centre $F$ which is ramified at all but at most one of the infinite
places of $F$, and
which is ramified at a set $\Sigma$ of finite places of $F$ that
does not contain any places lying over~$p$. We assume further that
\begin{itemize}
\item $\rhobar|_{G_{F_w}}$ is generic for all places $w|p$, and
\item if $w\in\Sigma$, then $\rhobar|_{G_{F_w}}$ is not scalar.
\end{itemize}
Let $S$ be the union of $\Sigma$, the set of places dividing $p$, and
the places where $\rhobar$ is ramified. Fix a place $v|p$ of $F$; then
for each place $w\in S\setminus\{v\}$, we defined a compact open
subgroup $K_w$ of $(\cO_D)^\times_w$ in Section~\ref{subsec:minimal
  level}, and a finite free $\cO$-module $L_w$ with an action of
$K_w$. At some places $w\in S\setminus\{v\}$ we also defined Hecke
operators $T_w$ and scalars $\beta_w\in\F^\times$. We then set
$K_v=\GL_2(\cO_{F_v})$, $K=\prod_wK_w$, and defined a space
$S^{\min}(\sigma_v)_\m:=S(\sigma_v\otimes(\otimes_{w\in S,w\ne
  v}L_w))_{\m'}$ for each $\sigma_v$ a finitely generated $\cO$-module
with a continuous action of $\GL_2(\cO_{F_v})$, having a central
character which lifts
$(\varepsilonbar\det\rhobar|_{I_{F_v}})\circ\Art_{F_v}$. By
definition, $S^{\min}(\sigma_v)_\m$ is  a space of modular
forms (strictly speaking, it is a localisation  of a
space of modular forms). We then constructed $\Mmin(\sigma_v)$ by
taking a dual, and additionally by
factoring out the Galois action in the indefinite case.

We extend this definition slightly in the following way. We write
$I_v$ for the Iwahori subgroup of $\GL_2(\cO_{F_v})$ consisting of
matrices which are upper triangular modulo $p$,
and for each character $\chi:I_v\to\F^\times$ such that $\det \chi =
  \varepsilonbar \det \rhobar |_{I_{F_v}}$, we
 define a space $S^{\min}(K^vI_v,\chi):=S(K^vI_v,\chi(\otimes_{w\in S,w\ne
  v}L_w))_{\m'}$ exactly as above. By factoring out the Galois
action in the indefinite case, and then taking duals,
 we also define a module $\Mmin(K^vI_v,\chi)$.

\begin{thm}
  \label{thm:multiplicity one for the usual lattice, including the I_1
    invariants statement}With the above notation and assumptions {\em (}so
  in particular $p\ge 5$, $\rhobar|_{G_{F_w}}$ is generic for all
  places $w|p$, and if $w\in\Sigma$, then $\rhobar|_{G_{F_w}}$ is not
  scalar{\em )}, we have $\dim_\F \Mmin(K^vI_v,\chi)^*[\m']\le 1$.
\end{thm}
\begin{proof}If $\Mmin(K^vI_v,\chi)=0$ then there is nothing to
  prove. Otherwise, we must show that $\dim_\F
  \Mmin(K^vI_v,\chi)^*[\m']=1$. By Frobenius reciprocity, we
  have \[\Mmin(K^vI_v,\chi)^*[\m']=\Mmin(\Ind_{I_v}^{K_v}\chi)^*[\m']=\Mmin(\sigmabar^\circ_{\cS}(\chi))^*[\m'].\]

Suppose first that $\chi$ is not
  scalar (that is, we can write $\chi=\eta\otimes\eta'$ with
  $\eta\ne\eta'$). Then, applying
  Theorem~\ref{thm: freeness of the sigma_J^circ} with
  $\sigma(\tau)=\sigma(\chi)$, $J=\cS$, and $\Mmin_\infty$ the
  patching functor of Section~\ref{subsec:minimal level}, we see that
  $\Mmin_\infty(\sigma^\circ_{\cS}(\chi))$ is a cyclic
  $R^{\psi,\tau}_\infty$-module. By construction,
  $\Mmin(\sigmabar^\circ_{\cS}(\chi))^*[\m']$ is dual to the
  cyclic $k$-module
  \[\Mmin_\infty(\sigma^\circ_{\cS}(\chi))/\m_{R^{\psi,\tau}_\infty}\Mmin_\infty(\sigma^\circ_{\cS}(\chi)),\]
  and we are done in this case.
  
  It remains to deal with the case that $\chi$ is scalar. Twisting, we
  may without loss of generality assume that $\chi$ is trivial. Then the
  Jordan--H\"older factors of $\sigmabar^\circ_{\cS}(\chi)$ are
  $\sigmabar_{\vec{0},\vec{0}}$ and
  $\sigmabar_{\vec{0},\vec{p}-\vec{1}}$, so  we
  have $\Mmin(\sigmabar^\circ_{\cS}(\chi))=0$ by Lemma~\ref{lem: generic
    representations have no weight p-1} and
Theorem~\ref{thm:geometric BM for types},
and we are done.
\end{proof}

\begin{rem}
  \label{rem:comparison to dembele's conjecture}
The relationship of this result to
Conjecture B.1 of~\cite{dembeleappendix} is as follows. Our condition
that there are no places $w\in\Sigma$ with $\rhobar|_{G_{F_w}}$ scalar
follows from Demb\'el\'e's assumption that there are no places
$w\in\Sigma$ with $\mathbf{N} w\equiv 1\pmod{p}$, and our only
additional restrictions are that $p\ge 5$ and
$\rhobar|_{G_{F(\zeta_p)}}$ is absolutely irreducible. Conjecture B.1
of~\cite{dembeleappendix} predicts the dimension of a certain space of
mod $p$ modular forms at pro-$p$ Iwahori level (note that while the
statement of the conjecture does not specify this, it is intended that
the mod $p$ modular forms considered are new at places away from $p$). There is a natural
variant of this conjecture in our setting (with our slightly different
choice of tame level), and this conjecture follows from
Theorem~\ref{thm:multiplicity one for the usual lattice, including the
  I_1 invariants statement} and \cite[Proposition
14.7]{Breuil_Paškūnas_2012} (which shows that the number of characters
$\chi$ with $\Mmin(K^vI_v,\chi)^*[\m']\ne 0$ is precisely the dimension
conjectured in~\cite{dembeleappendix}).
\end{rem}
\begin{cor}
  \label{cor:there is a BP rep in cohomology}Maintaining the
  assumptions of Theorem~{\em \ref{thm:multiplicity one for the usual lattice, including the I_1
    invariants statement}}, the $\GL_2(F_v)$-representation
  $\varinjlim_{K_v}\Mmin(K_v,\F)^*_{\m'}$
  contains one of the representations
  constructed in~\cite{Breuil_Paškūnas_2012}.
\end{cor}
\begin{proof}
This follows immediately from Theorem~\ref{thm:multiplicity one for the usual lattice, including the I_1
    invariants statement}, because if $I^1_v$ denotes the pro-$p$
  Sylow subgroup of $I_v$, then $\Mmin(K^vI_v^1,\chi)^*[\m']$ is
  stabilised by $
  \begin{pmatrix}
    0&1\\p&0
  \end{pmatrix}$ (\emph{cf.}~the discussion between Th\'eor\`emes 1.4 and 1.5
  of~\cite{breuillatticconj}).
\end{proof}
\begin{rem}
  \label{rem:alternative proof of existence of BP rep via explicit
    lattice}Breuil has shown that Corollary~\ref{cor:there is a BP rep
    in cohomology} can also be deduced from Theorem~\ref{thm:Breuil
    lattice conjecture}, without using the results of this section;
  see the proof of Th\'eor\`eme 10.1 of~\cite{breuillatticconj}.
\end{rem}

\subsection*{Acknowledgements}We would like to thank Christophe Breuil
and Fred Diamond for sending us a preliminary version of their paper
\cite{breuildiamond}, which was very helpful in writing Section
\ref{subsec:minimal level}. We would like to thank Christophe Breuil,
Fred Diamond, Hui Gao, David Helm, Florian Herzig and Vytautas Pa{\v{s}}k{\=u}nas for
helpful conversations. The debt that this paper owes to the work of
Christophe Breuil will be obvious to the reader, and it is a pleasure
to acknowledge this.

\appendix
\section{Unipotent Fontaine--Laffaille modules and $\phi$-modules}
\label{sec:unipotent-FL}

\subsection{Unipotent objects.}   Recall that results in $p$-adic Hodge theory that are valid in the
Fontaine--Laffaille range can
often be  extended slightly to the so-called \emph{unipotent} case
(see below for specific examples of what we mean by this). 
In the proof of Lemma~\ref{lem: identification of components of the special fibre for PS
    with rhobar reducible}, we need to compare the Galois
  representations associated to certain unipotent $\varphi$-modules
  with the Galois representations associated to unipotent
  Fontaine--Laffaille modules.  Breuil makes a similar comparison
  (without allowing the unipotent case)  in the proofs of
  \cite[Prop.~7.3]{breuillatticconj} and
  \cite[Prop.~A.3]{breuillatticconj}, making use of  certain results from
  \cite{MR1695849}.  In this appendix we will extend those results
  to the unipotent case, so that we can carry out
  the same argument.  Since these extensions are minor, we will generally
  indicate where the proofs of \cite{MR1695849} need to be changed,
  rather than repeating any proofs in their entirety.  We attempt to keep
  our notation as consistent as possible with that of
  \cite{MR1695849}; for instance we write $\phi$ in this appendix
  where we wrote $\varphi$ in the body of the paper.

\begin{remark}
  \label{rem:coefficients-unipotent}
  As in Appendix~A of \cite{breuillatticconj}, there exist  $\F$-coefficient versions of
  all of the following results (rather than  just $\Fp$-coefficients);
  in fact we can take the coefficients to be any Artinian local
  $\Fp$-algebra (such as the the ring $\F_J$ in the proof of Lemma~\ref{lem: identification of components of the special fibre for PS
    with rhobar reducible}).   These more general results follow entirely formally
  from their $\Fp$-coefficient versions, and so to simplify our
  notation we will omit coefficients in what follows.
\end{remark}

Throughout this appendix we let $k$ be a perfect field of
characteristic $p > 0$, and let $K_0$ denote the fraction field of
$W(k)$.  Fix a uniformizer $\pi$ of $K_0$.    We begin by recalling the definitions of several categories of objects that we
wish to consider.

\begin{defn}
  \label{defn:fontaine-laffaille-categories}
  Suppose that $0 \le h \le p-1$.  We define the category $\MFh$ (the
  category of mod $p$ Fontaine--Laffaille modules of height $h$) to be
  the category whose objects consist of:
  \begin{itemize}
  \item a finite-dimensional $k$-vector space $M$,
   \item a filtration $(\Fil^i M)_{i \in \Z}$ such that $\Fil^i M = M$
     for $i \le 0$ and $\Fil^i M = 0$ for $i \ge h+1$, and 
  \item for each integer $i$, a semi-linear map $\phi_i : \Fil^i M \to
    M$ such that $\phi_i |_{\Fil^{i+1} M} = 0$ and $\sum \Im \phi_i = M$.
  \end{itemize}
  Morphisms in $\MFh$ are $k$-linear maps that are compatible with
  filtrations and commute with the $\phi_i$.  
\end{defn}

When $h \le p-2$, there is an exact and fully faithful functor
from the category $\MFh$ to the category of $\Fp$-representations of
the absolute Galois group of $K_0$; this is no longer the case when
$h=p-1$, but as we will see, it remains true if we restrict to a certain subcategory of $\uMF^{f,p-1}_k$.

\begin{defn}
  \label{defn:unipotent-FL-modules}
  We say that an object $M$ of $\MFh$ is \emph{multiplicative} if
  $\Fil^{h} M = M$; we say that $M$ is \emph{unipotent} if it has no
  nontrivial multiplicative quotients.   Let $\uMF^{u,p-1}_k$ denote the full
 subcategory of $\uMF^{f,p-1}_k$
  consisting of the unipotent objects.
\end{defn}

We turn next to strongly divisible modules.  Let $\oS = k\langle u \rangle$ be the divided power polynomial
algebra in the variable $u$, and let $\Fil^h \oS$ be the
ideal generated by $\gamma_i(u)$ for $i \ge h$.   Write $c$ for the
element $-\gamma_p(u)-(\pi/p) \in \oS^{\times}$, and for $0
\le h \le p-1$ define
$\phi_h : \Fil^h \oS \to \oS$ by setting $\phi_h(\gamma_h(u)) =
c^h / h!$ and $\phi_h(\gamma_i(u)) = 0$ for $i > h$, except that
when $h=p-1$ we set $\phi_{p-1}(\gamma_p(u)) = c^p /(p-1)! \in \oS^{\times}$.   Set
$\phi = \phi_0$.

\begin{defn}
  \label{defn:strongly-divisible-modules-mod-p}
    Suppose that $0 \le h \le p-1$.  We define the category $\strdiv^h$ (the
  category of mod $p$ strongly divisible modules of height $h$) to be
  the category whose objects consist of:
  \begin{itemize}
  \item a free, finite rank $\oS$-module $\EM$,

  \item a sub-$\oS$-module $\Fil^h \EM$ of $\EM$ containing $(\Fil^h
    \oS)\EM$

   \item a $\phi$-semi-linear map $\phi_h : \Fil^h \EM \to \EM$ such
     that for each $s \in \Fil^h \oS$ and $x \in \EM$ we have
     $\phi_h(sx) = c^{-h} \phi_h(s) \phi_h(u^h x)$ and such that
     $\phi_h(\Fil^h \EM)$ generates $\EM$ over $\oS$.
  \end{itemize}
  Morphisms in $\strdiv^h$ are $\oS$-linear maps that are compatible with
  filtrations and commute with $\phi_h$.  \end{defn}

\begin{defn}
  \label{defn:unipotent-str-div}
  An object of $\strdiv^h$ is called \emph{multiplicative} if $\Fil^h \EM = \EM$,
  and \emph{unipotent} if it has no non-zero multiplicative quotients.  Let
  $\strdiv^{u,p-1}$ denote the full subcategory of $\strdiv^{p-1}$
  consisting of the unipotent objects.
\end{defn}

Finally we define several categories of \emph{\'etale $\phi$-modules}.

\begin{defn}
  \label{defn:phi-modules}
  Let $\phiD$ be the category whose objects are finite-dimensional
  $k(\!(u)\!)$-vector spaces $\mathfrak{D}$ equipped with an injective
  map
  $\mathfrak{D} \to \mathfrak{D}$ that is semilinear with respect to
  the $p$th power map on $k(\!(u)\!)$; morphisms in $\phiD$ are
  $k(\!(u)\!)$-linear maps that commute with $\phi$.

 Let $\phiM$ be the category whose objects are finite-dimensional
  $k \llbracket u \rrbracket$-modules $\mathfrak{M}$ equipped with an injective
  map
  $\mathfrak{M} \to \mathfrak{M}$ that is semilinear with respect to
  the $p$th power map on $k \llbracket u \rrbracket$; morphisms in $\phiM$ are
  $k \llbracket u \rrbracket$-linear maps that commute with $\phi$.
  
Write $\phi^*$ for the  $k \llbracket u \rrbracket$-linear map $\Id
  \otimes \phi : k \llbracket u \rrbracket \otimes_{\phi} \mathfrak{M} \to
  \mathfrak{M}$.  We say that an object $\mathfrak{M}$ of $\phiM$ has height $h$ if
  $\mathfrak{M}/\Im\!(\phi^*)$ is killed by $u^h$.  Write $\phiM^h$ for
  the full subcategory of $\phiM$ of objects of height $h$, and
  $\phiD^h$ for the full subcategory of $\phiD$ consisting of objects
  isomorphic to $\mathfrak{M}[\frac 1u]$ for some $\mathfrak{M} \in
  \phiM^h$.\end{defn}

  \begin{defn}\label{defn:etale-phi-unip} We say that an object $\mathfrak{M}$ of $\phiM^h$ is \emph{unipotent} if it
  has no non-zero  quotients $\mathfrak{M}'$ such that
  $\phi(\mathfrak{M}') \subset u^h \mathfrak{M}'$.  Let $\phiM^{u,p-1}$ denote the
  full subcategory of $\phiM^{p-1}$ consisting of the unipotent
  objects; let $\phiD^{u,p-1}$ denote the full subcategory of
  $\phiD^{p-1}$ consisting of objects isomorphic to
  $\mathfrak{M}[\frac 1u]$ for $\mathfrak{M} \in \phiM^{u,p-1}$.  
\end{defn}

We recall the following fundamental result of Fontaine.

\begin{thm} $\mathrm{(}$\cite[B.1.7.1]{MR1106901}$\mathrm{)}$
  \label{thm:fontaine-equivalence}  The functor $\mathfrak{M} \leadsto
  \mathfrak{M}[\frac 1u]$ is an equivalence of categories from
  $\phiM^h$ to $\phiD^h$ for $0 \le h \le p-2$, as well as from
  $\phiM^{u,p-1}$ to $\phiD^{u,p-1}$. 
  \end{thm}

\subsection{Generic fibres.}  We fix an algebraic closure $\overline{K}_0$ of $K_0$,
and write $G_{K_0} = \Gal(\overline{K}_0/K_0)$. 
We also fix a system of elements $\pi_n$ in
$\overline{K}_0$, for $n \geq 0$, 
such that $\pi_0 = \pi$ and $\pi_n^p = \pi_{n-1}$ for $n
\ge1$.  We write $K_{\infty} = \cup_n K_0(\pi_n)$, and $G_{\infty} =
\Gal(\overline{K}_0/K_\infty)$.  

For each of the categories in
Definitions~\ref{defn:fontaine-laffaille-categories}--\ref{defn:etale-phi-unip}
we will now define a ``generic fibre'' functor to $k$-representations
of $G_{K_0}$ (for Fontaine--Laffaille modules) or $G_{\infty}$ (for strongly divisible
modules and $\phi$-modules).

Let $A_{\cris}$ be the ring (of the same name) defined in \cite[\S
3.1.1]{MR1621389}.  The ring $A_{\cris}$ has a filtration $\Fil^h
A_{\cris}$ for $h \ge 0$ and an endomorphism $\phi$ with the property that
$\phi(\Fil^h A_{\cris}) \subset p^h A_{\cris}$ for $0 \le h \le p-1$;
for such $h$ we may define $\phi_h : \Fil^h A_{\cris} \to A_{\cris}$ to be $\phi/p^h$.
By \cite[Lem.~3.1.2.2]{MR1621389} the ring $A_{\cris}/pA_{\cris}$ is
the ring denoted $R^{DP}$ in \cite{MR1695849}.  We have an induced
filtration on  $A_{\cris}/pA_{\cris}$  as well as induced maps $\phi_h
: \Fil^h (A_{\cris}/pA_{\cris}) \to A_{\cris}/pA_{\cris}$ for $0 \le h
\le p-1$.  The ring $A_{\cris}$ (and so also $A_{\cris}/pA_{\cris}$)
has a natural action of $G_{K_0}$ that commutes with all of the above
structures.  We regard $A_{\cris}/p A_{\cris}$ as an $\barS$-algebra
by sending the divided power $\gamma_i(u)$ to
$\gamma_i(\underline{\pi})$, where $\underline{\pi} \in A_{\cris}/p
A_{\cris}$ is the element defined in \cite[\S3.2]{MR1695849}; then the
group $G_{K_{\infty}}$ acts trivially on the image of $\barS$ in $A_{\cris}/p
A_{\cris}$.

\begin{defn}
  \label{defn:functors}  We define the following 
  functors.
  \begin{enumerate}
  \item If $M \in \MFh$ for $0 \le h \le p-2$ or $M \in
    \uMF^{u,p-1}_k$, we set
$$ T(M) =
\Hom_{\Fil^{\bullet},\phi_{\bullet}}(M,A_{\cris}/pA_{\cris}),$$
the $k$-linear morphisms that preserve filtrations and commute with the
$\phi_i$'s.  This is a $G_{K_0}$-representation via $g(f)(x) = g(f(x))$
for $g \in G_{K_0}$, $f \in T(M)$, and $x \in M$.

  \item If $\EM \in \strdiv^h$ for $0 \le h \le p-2$ or $M \in
    \strdiv^{u,h}$ with $h=p-1$, we set 
$$ T(\EM) =\Hom_{\Fil^{h},\phi_{h}}(\EM,A_{\cris}/pA_{\cris}),$$
the $\oS$-linear morphisms that preserve filtrations $\Fil^h$ and
commute with 
$\phi_h$.  This is a $G_{\infty}$-representation via $g(f)(x) = g(f(x))$
for $g \in G_{\infty}$, $f \in T(\EM)$, and $x \in \EM$.

  \item If $\mathfrak{D} \in \phiD$, we set
  $$ T(\mathfrak{D}) = \Hom_{\phi}(\mathfrak{D},k(\!(u)\!)^{\sep}),$$
  the $k[[u]]$-linear morphisms that commute with $\phi$.  This is
  naturally a $G_{\infty}$-representation by the theory of \emph{corps
    des normes} (which gives an isomorphism between
  $\Gal(k(\!(u)\!)^{\sep}/k(\!(u)\!))$ and $G_{\infty}$; see
  \cite[3.2.3]{MR719763}).

\item If $\mathfrak{M} \in \phiM^{h}$ for $0 \le h \le p-2$ or
 $\mathfrak{M} \in \phiM^{u,p-1}$, we set $T(\mathfrak{M}) :=
 T(\mathfrak{M}[\frac 1u])$.
  \end{enumerate}
We refer to $T(M)$ (respectively $T(\EM)$, $T(\mathfrak{D})$,
$T(\mathfrak{M})$) as the generic fibre of $M$ (respectively $\EM$,
$\mathfrak{D}$, $\mathfrak{M}$).
\end{defn}

\begin{prop}
  \label{prop:dimensions-are-correct}  We have the following.
  \begin{enumerate}
  \item If $M \in \MFh$ for $0 \le h \le p-2$ or $M \in \uMF^{u,p-1}_k$, then $$\dim_k T(M) = \dim_k M.$$

  \item If $\EM \in \strdiv^h$ for $0 \le h \le p-2$ or 
  $\EM \in  \strdiv^{u,p-1}$, then $$\dim_k T(\EM) = \rank_{\oS} \EM.$$

  \item If $\mathfrak{D} \in \phiD$, then $$\dim_k T(\mathfrak{D}) =
    \dim_{k(\!(u)\!)} \mathfrak{D}.$$
  \end{enumerate}
 \end{prop}

 \begin{proof}
  Parts  (1) and (3) are classical: for instance, (1) is
  \cite[Thm.~6.1]{fl}, and (3) is \cite[Prop.~1.2.6]{MR1106901}.
  (In (3) the functor $T$ is even an anti-equivalence of
  categories.)  

  As for (2), the case $0 \le h \le p-2$ is
  \cite[Lem.~3.2.1]{MR1695849}.  The proof for $\strdiv^{u,p-1}$ is
  exactly the same, except that the reference to
  \cite[2.3.2.2]{MR1468834} must be supplemented with a reference to
  \cite[Lem.~2.3.3.1]{MR1468834}.
 \end{proof}

\subsection{Equivalences.}  Regard $k\llbracket u \rrbracket$
naturally as a subring of $\oS$.  Following \cite[\S 4]{MR1695849}, if $0
\le h \le p-1$ we define a functor $\Theta_h : \phiM^h \to \strdiv^h$
as follows.  If $\mathfrak{M}$ is an object of $\phiM^h$, we define
$\Theta_h(\phiM)$ to be the object $\EM$ constructed as follows:
\begin{itemize}
\item $\EM = \oS \otimes_{\phi} \mathfrak{M}$,
\item $\Fil^h \EM = \{ y \in \EM \, : \, (\Id \otimes \phi)(y) \in
  (\Fil^h \oS)\otimes_{\phi} \mathfrak{M}\}$,
\item $\phi_h : \Fil^h \EM \to \EM$ is the composition of $\Id \otimes
  \phi : \Fil^h \EM \to  (\Fil^h \oS)\otimes_{\phi} \mathfrak{M}$ with
  $\phi_h \otimes \Id : (\Fil^h \oS)\otimes_{\phi} \mathfrak{M} \to \oS
  \otimes_{\phi} \mathfrak{M} \cong \EM$.
\end{itemize}

\begin{thm}
  \label{thm:equivalence-A}
  The functor $\Theta_h$ induces an equivalence of categories from
  $\phiM^h$ to $\strdiv^h$ for $0 \le h \le p-2$, as well as an
  equivalence of categories from $\phiM^{u,p-1}$ to $\strdiv^{u,p-1}$.
\end{thm}

\begin{proof}
  For $0 \le h \le p-2$ this is \cite[Thm.~4.1.1]{MR1695849}, while 
  the unipotent case is explained in the proof of Theorem 2.5.3
  of \cite{HuiGao}.  (Note, however, that some of our terminology is
  dual to that of \cite{HuiGao}: what we call a
  multiplicative object of $\strdiv^h$ or $\MFh$, Gao calls \'etale.)
  We briefly recall the argument from \cite{HuiGao}. If $A$ is the
  matrix of $\phi$ on $\mathfrak{M}$ with respect to some fixed basis,
  the condition that $\mathfrak{M}$ is unipotent is precisely the
  condition that the product $\prod_{n=1}^{\infty} \phi^n (u^{p-1}
  A^{-1})$ converges to $0$, and this is exactly what is
  required for the proof of full faithfulness in
  \cite[Thm.~4.1.1]{MR1695849} to go through (as well as to show that
  $\Theta_h(\mathfrak{M})$ is actually unipotent); the essential
  surjectivity in  \cite[Thm.~4.1.1]{MR1695849} is true when $h=p-1$
  without any unipotence condition.
\end{proof}

\begin{prop}
  \label{prop:isomA}
  If $\mathfrak{M} \in \phiM^{u,p-1}$, we have a canonical
  isomorphism of $G_{\infty}$-representations
  $T(\mathfrak{M}) \cong T(\Theta_{p-1}(\mathfrak{M}))$.
\end{prop}

\begin{proof}
  The same statement for $\Theta_h : \phiM^h \to \strdiv^h$ with $0
  \le h \le p-2$ is \cite[Prop.~4.2.1]{MR1695849}.  The proof in the
  unipotent case is identical, noting that the condition $\prod_{n=1}^{\infty} \phi^n (u^{p-1}
  A^{-1})=0$ is precisely what is needed for the last two sentences of the
  argument in \cite{MR1695849} to go through.
\end{proof}

Following \cite[\S 5]{MR1695849}, if $0
\le h \le p-1$ we define a functor $\EF_h : \MFh \to \strdiv^h$
as follows.  If $M$ is an object of $\MFh$, we define
$\EF_h(M)$ to be the object $\EM$ constructed as follows:
\begin{itemize}
\item $\EM = \oS \otimes_{k} M$,
\item $\Fil^h \EM = \sum_{i=0}^h \Fil^i \oS \otimes_k \Fil^{h-i} M$,
\item $\phi_h = \sum_{i=0}^h \phi_i \otimes \phi_{h-i}$.
\end{itemize}

While we expect it to be true that if $M$ is unipotent then so is $\EF_{p-1}(M)$,
and that the resulting functor $\EF^u_{p-1} :
  \uMF^{u,p-1}_k \to \strdiv^{u,p-1}$ is fully faithful, we will not
  need these assertions, and so we do not prove them.  Instead, we note the
  following.

\begin{prop}
  \label{prop:isomB}  If $M \in \uMF^{u,p-1}_k$ and $\EF_{p-1}(M)$
  is unipotent, then we have a canonical
  isomorphism of $G_{\infty}$-representations $T(M)|_{G_{\infty}} \cong T(\EF_{p-1}(M))$.
\end{prop}

\begin{proof}
  The same argument as in the paragraph before Lemme~5.1 of
  \cite{MR1695849} (i.e., the proof of \cite[3.2.1.1]{MR1621389})
  shows that the natural map $T(M) \to T(\EF_{p-1}(M))$ is injective. 
  Now the result follows by parts (1) and (2) of
  Proposition~\ref{prop:dimensions-are-correct} together with the
  assumption that $\EF_{p-1}(M)$ is unipotent.
\end{proof}

\section{Remarks on the geometric Breuil--M\'ezard
  philosophy}\label{sec: appendix on geom BM}
In this appendix we will explain some
variants on the arguments of the main part of the paper, which are
either more conceptual, or which avoid the use of results from papers
such as \cite{geekisin}, \cite{blggU2}, and~\cite{GLS12}.
 Since
nothing in the main body of the paper depends on this appendix, and
since a full presentation of the arguments would be rather long, we
only sketch the proofs.
\subsection{Avoiding the use of results on the weight part of Serre's
  conjecture}\label{subsec:avoiding GK}In this section we will indicate how the main results of the
paper (namely, the application of the results of Sections~\ref{sec:
  proof of the conjecture}--\ref{sec: freeness and multiplicity one}
to the specific examples of patching functors constructed from spaces
of modular forms in Section~\ref{sec: cohomology}) could be proved
without relying on the main results of~\cite{geekisin}, and without
using the potential diagonalizability of potentially Barsotti--Tate
representations proved in~\cite{kis04} and~\cite{MR2280776}, but
rather just the explicit computations of potentially Barsotti--Tate
deformation rings in Section~\ref{sec: deformation spaces}. Of course,
the results of Section~\ref{sec: deformation spaces} depend on
\cite{geekisin} (via Theorem~\ref{thm:geometric BM for types}, which
depends on~\cite{geekisin} via \cite{emertongeerefinedBM}), so we will
need to explain how to prove the results of both Sections~\ref{sec:
  cohomology} and~\ref{sec: deformation spaces} (or at least enough of
them to prove the main results of Sections~\ref{sec: proof of the
  conjecture}--\ref{sec: freeness and multiplicity one}) without using
these results. We will still make use of~\cite{GLS12} in order to
prove Theorem~\ref{thm:geometric BM for types}, but we remark that it
should be possible to establish all of the results of
Section~\ref{sec: deformation spaces} by directly extending the
computations of Section~5 of~\cite{BreuilMezardRaffinee} to the tame
cuspidal case, so it should ultimately be possible to remove this
dependence as well.

In fact, if we examine Sections~\ref{sec: cohomology} and~\ref{sec:
  deformation spaces}, we see that there are two things that need to
be established: Theorem~\ref{thm:deformation rings principal
  series} (which only depends on the above items via
Theorem~\ref{thm:geometric BM for types}),
 and the claim that if $\sigmabar$ is a Serre weight and
$M_\infty$ is constructed from the spaces of modular forms for a
quaternion algebra using the Taylor--Wiles--Kisin method as in
Section~\ref{subsec: existence of patching functors}, then the action
of the universal lifting ring $R$ on $M_\infty(\sigmabar)$ factors
through $R^{\psi,\sigmabar}$ (or in fact through its reduction mod $\unif_E$).
(We recall that this is not immediately deducible from local-global
compatibility because of parity issues: if $\sigma$ denotes the algebraic
representation of $\GL_2(O_{F_v})$ lifting $\sigmabar$, then the weights 
of $\sigma$ may not satisfy the parity condition necessary for constructing
a local system on the congruence quotients associated to our given
quaternion algebra.)
Note that we only need to prove these claims under the
assumption that $\rhobar$ is generic (in the local case), or that
$\rhobar|_{G_{F_v}}$ is generic for all $v|p$ (in the global case),
and we will make this assumption from now on.

The claim about $M_\infty(\sigmabar)$ will follow from the
results of Section~\ref{sec: deformation spaces}, in the following
fashion. Suppose first that $\sigmabar\notin\cD(\rhobar)$; we will
show that $M_\infty(\sigmabar)=0$ by an argument similar to those
used in~\cite{geekisin}. To do this, note that by Lemma 4.5.2
of~\cite{geekisin} (and its proof), in the Grothendieck group of
mod $p$ representations of $\GL_2(k_v)$ we can write $\sigmabar$ as a
linear combination  $\sum_\tau n_\tau\sigmabar(\tau)$ where $\tau$ runs over the
tame inertial types (allowing for the moment $\tau$ to be the trivial or small
Steinberg type). Then
if $e$ denotes the Hilbert--Samuel multiplicity of a coherent sheaf on
$X^\psi_\infty$, we have \[e(M_\infty(\sigmabar))=\sum_\tau n_\tau
e(M_\infty(\sigmabar(\tau))).\] By the argument of Lemma 4.3.9
of~\cite{geekisin}, and the irreducibility of the generic fibres of
the deformation spaces (see the discussion below), $\sum_\tau n_\tau
e(M_\infty(\sigmabar(\tau)))$ is a multiple of $\sum_\tau n_\tau
e(R^{\psi,\tau})$ (and the constant of proportionality is independent
of~$\sigmabar$), so it suffices to check that this last quantity is
zero. 

To see this, note that the $e(R^{\psi,\tau})$ are completely
determined by Theorem~\ref{thm:deformation rings principal series}
and the observation that since $\rhobar$ is assumed generic,
$R^{\psi,\tau}=0$ if $\tau$ is the trivial or Steinberg type (this
reduces to checking that the reduction mod $p$ of a semistable
(possibly crystalline) Galois representation of Hodge type $0$ is not
generic, which is clear). In fact, we see that
$e(R^{\psi,\tau})=|\JH(\sigmabar(\tau))\cap\cD(\rhobar)|$. Applying the
argument of the previous paragraph to \emph{all} $\sigmabar$
simultaneously, 
and noting that again by Lemma 4.5.2
of~\cite{geekisin} (and its proof) the $e(M_\infty(\sigmabar(\tau)))$
determine the $e(M_\infty(\sigmabar))$, we see that there must
be some constant $c$ such that for any $\sigmabar$,
we have $e(M_\infty(\sigmabar))=c$ if $\sigmabar\in\cD(\rhobar)$, and $0$
otherwise. In particular, if $\sigmabar\notin\cD(\rhobar)$, then
$M_\infty(\sigmabar)=0$, as required.

 Now suppose that
$\sigmabar\in\cD(\rhobar)$. Then by Proposition~\ref{prop: types for
  elimination} there is a tame inertial type $\tau$ such that
$\JH(\sigmabar(\tau))\cap\cD(\rhobar)=\{\sigmabar\}$, so by the result of the
previous paragraph we have
$M_\infty(\sigmabar^\circ(\tau))=M_\infty(\sigmabar)$ for any lattice
$\sigma^\circ(\tau)$ in $\sigma(\tau)$. Also, by
Theorem~\ref{thm:geometric BM for types}, the reduction mod $\unif_E$ of
$R^{\psi,\sigmabar}$ is equal to the reduction mod $\unif_E$ of $R^{\psi,\tau}$,
so it suffices to note that the action of $R$ on
$M_\infty(\sigmabar^\circ(\tau))$ factors through $R^{\psi,\tau}$.

It remains to prove Theorem~\ref{thm:geometric BM for
  types}. Examining the proof of~\cite[Thm.\
5.5.4]{emertongeerefinedBM}, we see that we have to establish the
existence of a suitable globalization in the sense of Section~5.1
of~\cite{emertongeerefinedBM}), and we have to avoid the use of Lemma
4.4.1 of~\cite{geekisin} (the potential diagonalizability of
potentially Barsotti--Tate representations.). (Note that we still
appeal to Lemma 4.4.2 of~\cite{geekisin}, and that this is where we
make use of~\cite{GLS12}.) However, Lemma 4.4.1 of~\cite{geekisin} is
only used to establish that certain patched spaces of modular forms
are supported on every component of the generic fibre of each local
deformation ring which we consider, and this will be automatic
provided we know that these generic fibres are domains. By a standard
base change argument, it suffices to know this after a quadratic base
change, so we can reduce to the case of tame principal series types,
which follows from Lemma~\ref{lem: principal series def ring from BM}
(the proof of which makes no use of Theorem~\ref{thm:geometric BM for
  types}).

It remains to check the existence of a suitable globalization, which
amounts to checking Conjecture A.3 of~\cite{emertongeerefinedBM}, the
existence of a potentially diagonalizable lift of $\rhobar$. In the
proof of~\cite[Thm.\ 5.5.4]{emertongeerefinedBM}, this is done by
appealing to results of~\cite{geekisin}, which use the potential
diagonalizability of potentially Barsotti--Tate
representations. However, as we are assuming that $\rhobar$ is
generic, if we choose some weight $\sigmabar\in\cD(\rhobar)$ then
a crystalline lift of $\rhobar$ of Hodge type $\sigmabar$ is
Fontaine--Laffaille and thus potentially diagonalizable, as required.

\begin{remark}
  \label{rem: removing $A_5$ restriction}We suspect that the above
  analysis would make it possible to remove the assumption that when
  $p=5$ the projective image of $\rhobar(G_{F(\zeta_5)})$ is not
  isomorphic to $A_5$ from our main global theorems,
 but we have not attempted to check the details.
\end{remark}

\begin{remark}
  \label{reprove geebdj without ordinary assumption}In particular, if
  we apply the above analysis to Theorem~\ref{thm: generic BDJ in
    patching language}, we see that we can reprove the main result
  of~\cite{geebdj} without making use of the results on the
  components of Barsotti--Tate deformation rings proved
  in~\cite{kis04}, \cite{MR2280776}, and without needing to make an
  ordinarity assumption. (Of course, this ordinarity assumption was
  already removed by~\cite{blggordII}.)
\end{remark}

\subsection{Deformation rings and the structure of
  lattices}\label{subsec:deformation rings to compute lattices}The main theme of this paper is the close connection
between the structure of the lattices in tame types for $\GL_2(F_v)$,
and the structures of the corresponding potentially Barsotti--Tate
deformation rings for generic residual representations. We believe
that this connection is even tighter than we have been able to show;
in particular, we suspect that it should be possible to use explicit
descriptions of the deformation rings to recover Theorem~\ref{thm:
  socle filtration for tame type with irred socle}
without making use of the results
of~\cite{Breuil_Paškūnas_2012}.

Unfortunately, we have not been able to prove this, and our partial
results are fragmentary. One difficulty is that in general there will
be irregular weights occurring as Jordan--H\"older factors of the
reductions of the lattices, and it would be necessary to have an
explicit description of the deformation rings for non-generic residual
representations. However, even in cases where all the Jordan--H\"older
factors are regular, we were still unable to prove complete
results. As an illustration of what we were able to prove, we have the
following modest result, whose proof will illustrate the kind of
arguments that we have in mind.

Let $p>3$ be prime, let $F_v$ be an unramified extension of $\Qp$, let
$\tau$ be a tame type, and consider the lattice $\sigma^\circ_J(\tau)$ for
some $J\in\cP_\tau$.
\begin{prop}
  \label{prop: patching implies extension of weights is
    non-split}Suppose that $\sigmabar_{J_1}(\tau),
  \sigmabar_{J_2}(\tau)$ are two weights in adjacent layers of the
  cosocle filtration of $\sigmabar^\circ_J(\tau)$, and that there exists a
  nonsplit extension between $\sigmabar_{J_1}(\tau),
  \sigmabar_{J_2}(\tau)$ as $\GL_2(F_v)$-representations. Then the
  extension induced by the cosocle filtration of
  $\sigmabar^\circ_J(\tau)$ is also nonsplit.
\end{prop}
\begin{proof}
  By an explicit (but tedious) computation ``dual'' to that used in
  the proof of Proposition~\ref{prop: when there are extensions of Serre weights, and
    existence of types when this happens}, the existence of a  nonsplit extension between $\sigmabar_{J_1}(\tau),
  \sigmabar_{J_2}(\tau)$ implies that there is a semisimple generic
  representation $\rhobar$ such that
  $\JH(\sigmabar(\tau)) \cap\cD(\rhobar)=\{\sigmabar_{J_1}(\tau),\sigmabar_{J_2}(\tau)\}$. Since
  $\rhobar$ is semisimple, it is easy to globalise it as the local mod $p$
  representation corresponding to some CM Hilbert modular form, and
  the constructions of Section~\ref{subsec:minimal level} then give a
  minimal fixed determinant patching functor $\Mmin_\infty$ with unramified
  coefficients $\cO$, indexed by $(F_v,\rhobar)$.

As was already noted in Subsection~\ref{subsec: specific lattices and gauges},
it follows from
Lemmas~\ref{lem:uniqueness of lattices with irreducible socle}
and~\ref{lem:weights in tame types have multiplicity one} that
$\sigma^\circ_J(\tau)$ is defined over the ring of integers of an
unramified extension of $\Qp$, so extending scalars if necessary, we
can assume that it is defined over $\cO$. By
Theorem~\ref{thm:deformation rings principal series}, and the
assumptions that $|\JH(\sigmabar(\tau)) \cap \cD(\rhobar)|=2$ and $\cO$ is unramified, we see
that the deformation space $X^\psi\bigl(\tau\bigr)$ is regular, so by
Lemma~\ref{lem: abstract version of Diamond freeness argument} we see
that $\Mmin_\infty(\sigma_J^\circ(\tau))$ is free of rank one over
$X^\psi_\infty\bigl(\tau\bigr)$, and thus in particular
$\Mmin_\infty(\sigmabar_J^\circ(\tau))$ is free of rank one over $\Xbar^\psi_\infty\bigl(\tau\bigr)$.

Now, if the extension between $\sigmabar_{J_1}(\tau),
\sigmabar_{J_2}(\tau)$ induced by the cosocle filtration on
$\sigmabar^\circ_J(\tau)$ were split, we would have \[
\Mmin_\infty(\sigmabar_J^\circ(\tau))=\Mmin_\infty(\sigmabar_{J_1}(\tau))\oplus
\Mmin_\infty(\sigmabar_{J_1}(\tau))\]which would not be free of rank
one, which is a contradiction; so the extension must be nonsplit, as
claimed.
\end{proof}
\bibliographystyle{amsalpha}
\bibliography{breuil-conj.bib}

\end{document}